\documentclass[11pt,a4paper]{article}

\usepackage[utf8]{inputenc}
\usepackage[margin=3.1cm]{geometry}
\usepackage{amsmath,amssymb,amsthm}
\usepackage{graphicx}
\usepackage{float}
\usepackage{booktabs}
\usepackage{comment}
\usepackage{enumitem}
\usepackage{tikz}
\usetikzlibrary{arrows.meta,positioning,calc,decorations.pathreplacing}
\usepackage[colorlinks=true,allcolors=blue!55!black]{hyperref}
\usepackage{xcolor}
\usepackage{aliascnt}
\usepackage{cleveref}
\hypersetup{pdftitle={Boundary stabilization of a moving interface in a 2x2 weakly hyperbolic system of balance laws: application to continuous sedimentation},pdfauthor={M. O. Amirat, V. Andrieu, M. Bajodek, C. Valentin}}

\newcommand{\appiss}{\hyperref[app:iss]{Appendix~\ref*{app:iss}}}
\newcommand{\appsteady}{\hyperref[app:steady]{Appendix~\ref*{app:steady}}}
\newcommand{\appwp}{\hyperref[app:wp]{Appendix~\ref*{app:wp}}}

\newtheorem{theorem}{Theorem}
\crefname{theorem}{Theorem}{Theorems}
\newaliascnt{proposition}{theorem}
\newtheorem{proposition}[proposition]{Proposition}
\aliascntresetthe{proposition}
\crefname{proposition}{Proposition}{Propositions}
\newaliascnt{lemma}{theorem}
\newtheorem{lemma}[lemma]{Lemma}
\aliascntresetthe{lemma}
\crefname{lemma}{Lemma}{Lemmas}
\newaliascnt{corollary}{theorem}

\aliascntresetthe{corollary}
\crefname{corollary}{Corollary}{Corollaries}
\theoremstyle{definition}
\newaliascnt{definition}{theorem}
\newtheorem{definition}[definition]{Definition}
\aliascntresetthe{definition}
\crefname{definition}{Definition}{Definitions}
\newaliascnt{assumption}{theorem}
\newtheorem{assumption}[assumption]{Assumption}
\aliascntresetthe{assumption}
\crefname{assumption}{Assumption}{Assumptions}
\theoremstyle{remark}
\newaliascnt{remark}{theorem}
\newtheorem{remark}[remark]{Remark}
\aliascntresetthe{remark}
\crefname{remark}{Remark}{Remarks}

\newcommand{\Rp}{\mathbb{R}_{+}}
\newcommand{\eps}{\varepsilon}
\newcommand{\epss}{\varepsilon_s}
\newcommand{\epsin}{\varepsilon_{in}}
\newcommand{\Om}{\Omega}
\newcommand{\Hs}{\mathcal{H}}
\newcommand{\Vc}{\mathcal{V}}
\newcommand{\Nc}{\mathcal{N}}
\newcommand{\Uc}{\mathcal{U}}
\newcommand{\Dpar}{\mathcal{D}_\partial}
\newcommand{\Lk}{\mathbf{L}}
\newcommand{\Ab}{\mathbf{A}}
\newcommand{\Bb}{\mathbf{B}}
\newcommand{\Cb}{\mathbf{C}}
\newcommand{\Db}{\mathbf{D}}
\newcommand{\Pb}{\mathbf{P}}
\newcommand{\Qb}{\mathbf{Q}}
\newcommand{\Wb}{\mathbf{W}}
\newcommand{\Nb}{\mathbf{N}}
\newcommand{\Go}{\mathbf{G}_o}
\newcommand{\Yo}{Y_o}
\newcommand{\Kint}{\mathbf{K}_{int}}
\newcommand{\Pib}{\boldsymbol{\Pi}}
\newcommand{\Xib}{\boldsymbol{\Xi}}
\newcommand{\dd}{\,\mathrm{d}}
\newcommand{\jump}[1]{\left[#1\right]}

\title{\bfseries Boundary stabilization of a moving interface in a
\texorpdfstring{$2\times2$}{2x2} weakly hyperbolic system of balance laws:
application to continuous sedimentation}

\author{M.~O.~Amirat \and V.~Andrieu \and M.~Bajodek \and C.~Valentin
\thanks{M.~O.~Amirat, V.~Andrieu and C.~Valentin: Université Lyon 1, CNRS, LAGEPP, UMR 5007, Villeurbanne, France. M.~Bajodek: 
Université Lyon 1, CNRS, LAGEPP, UMR 5007, CPE Lyon, Villeurbanne, France.}}

\date{}

\begin{document}
\maketitle

\begin{abstract}
\noindent
We design a boundary feedback for the level of the sediment bed of a continuous settler and prove that it exponentially stabilizes the linearized
closed loop under a small-gain condition.
The unit is modeled by a $2\times2$ system of balance laws in which
the bed surface, the interface, is not a smooth part of the solution but a moving \emph{shock}
between a weakly hyperbolic free-settling zone and a strictly hyperbolic
compression zone. We describe it by the Rankine--Hugoniot relations and map the
problem onto a fixed domain. Moreover, we  prove local well-posedness with a continuation
criterion given by a single margin function. Then we linearize the system near an equilibrium. The linearization
couples a dissipative distributed layer to one scalar, the interface
position, whose equation does not contain the control. We solve this by
infinite-dimensional \emph{forwarding}: a Sylvester kernel cancels the
interface trace and isolates the interface pole, and a proportional feedback
then sets its decay rate.
Every hypothesis is verified on an illustrative example, a laboratory-scale
continuous settler containing a bed of compact polystyrene beads.
\end{abstract}


\section{Introduction}
\label{sec:intro}

The clarifier is the separation stage of an activated-sludge wastewater
treatment plant: the mixed liquor coming from the aeration tank is left to
settle, clarified water leaves at the top, and the thickened sludge is
withdrawn at the bottom and partly recirculated. Inside the unit, a sharp
horizontal front, the \emph{sludge blanket}, separates a dilute
suspension above from a compressed sediment below (\Cref{fig:unit}). Keeping
the sludge blanket at a prescribed depth is the central operational problem: if it is too high, solids
escape with the effluent; if it is too low, the recirculated sludge is too thin.
The only actuation available is the underflow withdrawal rate.
The model and the design below apply to any suspension with a contact
threshold, so in the rest of the paper we call the front the \emph{interface} and the
compressed sediment the \emph{bed}.

Classical models come from Kynch's kinematic theory \cite{kynch1952theory}
and close a scalar conservation law for the solid volume fraction with an
algebraic settling law \cite{burger2005model}; they propagate concentration
fronts but ignore momentum. Replacing that closure by a momentum balance
\cite{chauchat2013modelling, francca1999study} yields the $2\times2$ system we control here,
validated on full-scale data in~\cite{valentin2023knowledge}.

The problem differs in two ways from the standard boundary control theory of
hyperbolic systems \cite{bastin2016stability}.

First, the target steady state is \emph{not smooth}. In the boundary stabilization
literature, one usually starts from a regular initial condition close to a
smooth steady state and designs a feedback law that keeps the solution  smooth
and close to it. In  other words, enough dissipation is introduced to prevent
any shock formation (see \cite{bastin2016stability}, with many applications,
in particular to channel flows). Shocks are nevertheless  of high physical
relevance, since  they may represent moving interfaces. In our case, the interface described above \emph{is} represented as a shock and its position is
the regulated output.

Second, in the free-settling zone the system we consider is only \emph{weakly} hyperbolic. Where the
interparticle stress vanishes, the flux Jacobian has a double eigenvalue, the
particle velocity, with a one-dimensional eigenspace: it is a Jordan block.
Such systems appear in elasticity~\cite{keyfitz1980system} and in pressureless
gas dynamics~\cite{bouchut1994zero}. They cannot be diagonalized, and the
natural energy space is not $L^2$ for both components: one of them must be
measured in $H^1$ (see our study in \cite{amirat2026insights}).

The problem we solve is the following.

\begin{quote}
\emph{Given an admissible stationary profile $y^*$ with interface at $z_s^*$,
design a feedback law for the underflow velocity $v_m$ that stabilizes the
interface level $z_s(t)$ at $z_s^*$ together with the distributed
profiles.}
\end{quote}
The law we obtain is a full-state feedback (it uses $z_s(t)$ and the profiles
on both zones). The design of an observer is left for future work.

\begin{figure}[t]
\centering
\includegraphics[width=0.6\textwidth]{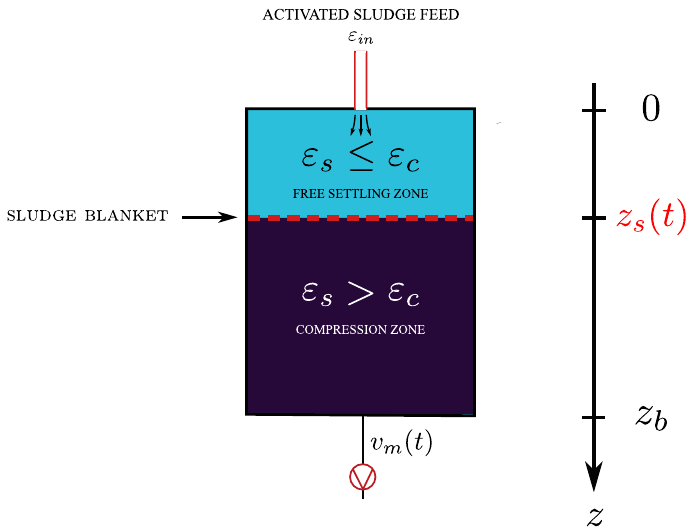}
\begin{comment}
\begin{tikzpicture}[>=Latex, font=\small]
  \fill[blue!5]   (0,0) rectangle (2.6,-2.7);
  \fill[brown!18] (0,-2.7) rectangle (2.6,-4.6);
  \draw[thick] (0,0) rectangle (2.6,-4.6);
  \draw[very thick] (0,-2.7) -- (2.6,-2.7);
  \node[font=\footnotesize] at (1.3,-1.15) {dilute};
  \node[font=\footnotesize] at (1.3,-1.52) {suspension};
  \node[font=\footnotesize] at (1.3,-2.47) {interface};
  \node[font=\footnotesize] at (1.3,-3.45) {compressed};
  \node[font=\footnotesize] at (1.3,-3.82) {sediment};
  \draw[->, thick] (1.3,0.75) -- (1.3,0.08);
  \node[font=\footnotesize] at (1.3,0.95) {feed};
  \draw[->, thick] (1.3,-4.68) -- (1.3,-5.35);
  \node[font=\footnotesize] at (1.3,-5.55) {underflow $v_m(t)$};
  \draw[->] (-0.5,0) -- (-0.5,-5.0) node[below,font=\footnotesize] {$z$};
  \foreach \y/\t in {0/0,-2.7/z_s(t),-4.6/z_b}
    \draw (-0.42,\y) -- (-0.58,\y) node[left,font=\footnotesize] {$\t$};
  \draw[dashed] (-0.5,-2.7) -- (0,-2.7);
  \draw[dashed] (-0.5,0) -- (0,0);
  \draw[dashed] (-0.5,-4.6) -- (0,-4.6);
  \draw[decorate,decoration={brace,amplitude=4pt}]
    (2.75,-0.03) -- (2.75,-2.67)
    node[midway,right=5pt,font=\footnotesize,align=left]
    {free-settling\\zone, $\epss<\eps_c$};
  \draw[decorate,decoration={brace,amplitude=4pt}]
    (2.75,-2.73) -- (2.75,-4.57)
    node[midway,right=5pt,font=\footnotesize,align=left]
    {compression\\zone, $\epss>\eps_c$};
\end{tikzpicture}
\end{comment}
\caption{The two particle-containing zones of the unit, from the feed level
$z=0$ to the bottom $z=z_b$ (the clarified-water zone above the feed is not
represented). The solid fraction is below the critical value $\eps_c$, at
which the particles come into contact, in the free-settling zone above the
interface, and above $\eps_c$ in the compression zone below it. The
interface  level $z_s(t)$ is  what must be  regulated, and the underflow velocity $v_m$ is the
only actuation.}
\label{fig:unit}
\end{figure}

After the fixed-domain reformulation and linearization, the plant
interconnects a distributed state on the two zones with one scalar, the
interface position error $\eta$. This scalar obeys a first-order equation
$\dot\eta=\Go\Yo$, driven by the outgoing wave traces $\Yo$, and the control
does not appear in it: it enters only at the boundaries and in the volume of
the distributed layer. A stable subsystem driving an integrator on which the control does not act
directly is the configuration for which \emph{forwarding} was
introduced~\cite{mazenc1996adding,sepulchre1997constructive}. Forwarding was
extended to the boundary control of exponentially stable hyperbolic systems
in~\cite{terrand2019adding}, to dissipative infinite-dimensional systems
coupled with an ODE in~\cite{marx2021forwarding} and to output regulation for
contraction semigroups in~\cite{vanspranghe2023output}; it is close to the
integral action of~\cite{dos2008boundary}. The same structural problem, a
hyperbolic layer coupled to an ODE, is also solved by
backstepping~\cite{dimeglio2018stabilization,deutscher2018output}, which
redistributes the damping through a Volterra transformation; forwarding keeps
the dissipation the layer already has and only moves the interface pole.

The paper makes five contributions.
\begin{enumerate}[label=(\arabic*),leftmargin=2em,itemsep=2pt]
\item We define \emph{admissible interface solutions} of \eqref{eq:model}: the
discontinuity is isolated by the Rankine--Hugoniot relations instead of being
resolved in $BV$, and the solution is classical on each side. A $C^1$
solution cannot cross the phase transition with a non-zero gradient
(\Cref{thm:breakdown}), but the coupled
state--interface problem is locally well posed, with a continuation criterion
given by one scalar margin function and four breakdown cases (\Cref{thm:wp},
\appwp{}).
\item We map the two zones onto fixed intervals by a time-dependent change of
the space variable, as in \cite{li1985boundary,bastin2019exponential}, and
linearize in Riemann coordinates that are diagonal on the compression zone and
in Jordan form on the free-settling zone (\Cref{sec:fixed}).
\item We prove an ISS Lyapunov inequality for the distributed layer on
$L^2\times H^1\times L^2\times L^2$, with respect to the interface error, the
outgoing traces at the interface, the input and its derivative; the $H^1$
component is imposed by the Jordan block (\Cref{sec:iss}).
\item We build a forwarding transformation $w=\eta-\Lk[Y]$ whose kernel $\Lk$
solves a Sylvester ODE with boundary matching conditions. It exists under a scalar non-resonance condition and a non-degeneracy
condition, and its parameter is the root $c^*$ of a
scalar equation (\Cref{sec:fwd}). The feedback $u=\mathfrak{K}w$ gives
$\dot w=-\kappa w$ without any smallness assumption, once the kernel exists
(\Cref{ass:fwd}); exponential stability of
the layer and the interface together holds under a small-gain condition
(\Cref{thm:fwd_GES}). The control is $C^\infty$ in time and uses neither a
trace nor a spatial derivative of the state, but it uses the whole
distributed state.
\item On an illustrative example where every hypothesis is
checked, the pole $c^*$ coincides, to the solver tolerance, with a zero of the
open-loop characteristic determinant, computed without the kernel equations
(\Cref{sec:num}).
\end{enumerate}
The stability certificate is a statement about the linearized loop and about
its decay rate: the small-gain condition is sufficient, it holds on the
example only for stiff and not too thick beds, and the transient constant it produces is
large (\Cref{ssec:feasibility}). The same example shows that the admissible
class is small exactly where the certificate holds (\Cref{ssec:design}); this
tension is discussed in \Cref{sec:conclusion}.

We now compare our work with the literature.
\begin{itemize}[leftmargin=1.5em,itemsep=2pt]
\item \emph{Shock steady states.} Bastin, Coron, Hayat and Shang stabilize the
position of a shock for the inviscid Burgers
equation~\cite{bastin2019exponential} and of a hydraulic
jump~\cite{bastin2019boundary}; in both cases the two sides of the shock are of
the same type. Comparable situations occur in traffic flow with phase
transitions~\cite{colombo2002hyperbolic,goatin2006aw}, the latter built on
the model of~\cite{aw2000resurrection}.
\item \emph{Shocks as free boundaries.} Shock fronts treated as free-boundary
problems with piecewise smooth solutions go back to
Majda~\cite{majda1983existence} and M\'etivier~\cite{metivier2001stability},
and in one space dimension to Li and Yu~\cite{li1985boundary};
\appwp{} follows this approach. Non-smooth solutions are otherwise treated in
$BV$~\cite{dafermos2005hyperbolic}: dissipative boundary conditions for
$2\times2$ systems with an entropy condition~\cite{coron2017dissipative},
controllability of scalar laws and of isentropic
Euler~\cite{ancona1998scalar,glass2007controllability}, its
limits~\cite{bressan2002boundary}, and stabilization of scalar
laws~\cite{perrollaz2013asymptotic}. We keep instead a single interface and
classical solutions on each side, which allows a Lyapunov design.
\item \emph{Stefan problems.} A fixed-domain reformulation also gives an
interface driven by traces~\cite{koga2020control}, but the phase boundary
separates two parabolic regions: the interface condition is a Stefan condition,
not a Rankine--Hugoniot pair, and no incoming characteristic has to be
selected.
\item \emph{Lyapunov methods.} The inequality of \Cref{sec:iss} follows the
strict-Lyapunov approach~\cite{coron2007strict,bastin2016stability,coron2015dissipative}
and the ISS framework of~\cite{karafyllis2019input,prieur2012iss}.
\item \emph{Control of sedimentation.} Feedback regulation of continuous
sedimentation has been studied on the scalar model by Barton, Li and
Spencer~\cite{barton1992control}, who regulate the bed level of a thickener
through the feed and underflow fluxes, and by
Diehl~\cite{diehl2001operating,diehl2008regulation}, whose operating charts
and regulator are built on the entropy solutions of the scalar law.
Our conference papers \cite{amirat2024stability,amirat2025boundary} treat the
stability and the boundary control of the linear $2\times2$ weakly hyperbolic
block on a fixed domain, without interface.
\end{itemize}
A shock separating a weakly and a strictly hyperbolic zone, its
well-posedness and the forwarding design for it are new. To our
knowledge, this is the first study of the $2\times2$ momentum-balance model
with the interface as a controlled shock.

We use the following notation. For a function $g$ that is discontinuous at
$z=z_s(t)$ we write
$g^{\pm}(t)=\lim_{z\to z_s(t)^{\pm}}g(t,z)$ and denote by
$\jump{g}(t)=g^{+}(t)-g^{-}(t)$ the \textit{jump} of $g$ across $z_s(t)$. Primes denote the derivative with respect to $\epss$ for the
constitutive laws $\sigma_e$ and $\mathcal{R}$, and $\partial_x$ for functions
of the reference space variable $x$ alone (\Cref{sec:fixed}). For a symmetric matrix $M$, $M\succ0$ means positive
definite. $\Vc$ always denotes the Lyapunov functional of
\Cref{def:Lyapunov}.

\section{Problem statement}
\label{sec:model}

\subsection{A \texorpdfstring{$2\times2$}{2x2} model of continuous sedimentation}
Let $\Om=(0,z_b)$ denote the clarifier depth interval, $z=0$ the feed inlet and
$z=z_b$ the underflow outlet. We present a one-dimensional, two-phase flow model describing continuous sedimentation in clarifiers. For the scenarios considered, the interface is assumed  to remain below the feed inlet. The natural state variables to be considered are the \textit{solid volume fraction}
$\epss:\Rp\times\bar\Om\to[0,1)$, along with the \textit{solid flux} $f_s=\epss v_s$, where $v_s$ represents
the average macroscopic particle velocity. Hence, writing mass and momentum balances yields the following $2 \times 2$ \textit{nonlinear} system of balance laws (see 
\cite{valentin2023knowledge} for model derivation):
\begin{equation}\label{eq:model}
\left\{
\begin{aligned}
&\partial_t\epss+\partial_zf_s=0,\\
&\partial_tf_s+\partial_zF(\epss,f_s)=S_0(\epss,f_s)+\Psi_0(\epss)\,v_m(t),
\end{aligned}\right.
\end{equation}
with
\begin{equation}\label{eq:FSP}
  F(\epss,f_s)=\frac{f_s^2}{\epss}+\sigma_e(\epss),\quad
  S_0(\epss,f_s)=K\epss-\mathcal{R}(\epss)f_s,\quad
  \Psi_0(\epss)=\mathcal{R}(\epss)\,\epss .
\end{equation}
Here $K=g(\rho_s-\rho_f)/\rho_s>0$  is the apparent gravity, with $\rho_s$,
$\rho_f$ the solid and fluid densities,  and $v_m(t)=Q_u(t)/A_{\rm set}$, with $Q_u$ the underflow rate and
$A_{\rm set}$ the cross-sectional area of the settler, is the underflow
velocity, which is the control input. The mixture properties are characterized by the constitutive
laws $\sigma_e$ and $\mathcal{R}$, which depend on a critical volume fraction threshold $\eps_c\in(0,1)$
(see \cite{landman1994solid}) as follows.
The \textit{effective interparticle stress} satisfies
$\sigma_e\in C^0([0,1))\cap C^2([\eps_c,1))$, the derivatives at $\eps_c$
being one-sided, with
\begin{equation}\label{eq:sigmae}
  \sigma_e\equiv0 \text{ on } [0,\eps_c],\qquad
  \sigma_e'>0 \text{ on } (\eps_c,1),\qquad
  L^{+}\triangleq\lim_{\epss\to\eps_c^{+}}\sigma_e'(\epss)>0,
\end{equation}
and \textit{the drag resistance} satisfies $\mathcal{R}\in C^0([0,1))\cap C^1([\eps_c,1))$,
again with a one-sided derivative at $\eps_c$, with
$\mathcal{R}\equiv\gamma>0$ on $[0,\eps_c]$ and $\mathcal{R}'>0$ on
$(\eps_c,1)$.

The suspension is fed at the inlet with the solid fraction
$\epsin(t)\in(0,\eps_c)$, a given function of time which is not controlled.
The boundary conditions are, at the feed inlet,
\begin{equation}\label{eq:BCtop}
  \epss(t,0)=\epsin(t),\qquad f_s(t,0)=\epsin(t)\Bigl(v_m(t)+\tfrac{K}{\gamma}\Bigr),
\end{equation}
and, at the underflow outlet,
\begin{equation}\label{eq:BCbot}
  f_s(t,z_b)=\epss(t,z_b)\,v_m(t).
\end{equation}
 Since $\mathcal{R}(\epsin)=\gamma$, the feed condition \eqref{eq:BCtop} means
  that the source term vanishes at the inlet:
  $S_0(\epsin,f_s(t,0))+\Psi_0(\epsin)v_m(t)=0$.

Since we are interested in discontinuous shock solutions of system \eqref{eq:model}--\eqref{eq:FSP}, we emphasize that the system is  written  in \textit{conservative variables} and in divergence
form, which fixes the propagation speed of the shocks we are interested in.
It can be expressed in the quasi-linear form $\partial_ty+A(y)\partial_zy=S(y,v_m)$, where
$y=(\epss,f_s)^\top$, $S(y,v_m)=\bigl(0,\,S_0(y)+\Psi_0(\epss)v_m\bigr)^\top$,
and the flux Jacobian is
\begin{equation}\label{eq:jac}
  A(y)=\begin{pmatrix}0&1\\ \sigma_e'(\epss)-v_s^2 & 2v_s\end{pmatrix}.
\end{equation}
Since $\sigma_e$ is active only above $\eps_c$, the hyperbolic  type of
\eqref{eq:model}--\eqref{eq:FSP} depends on $\epss$:  the
eigenvalues of $A(y)$ (characteristic velocities) are
\begin{equation*}
    \lambda_1 = \lambda_2 = v_s, \quad \text{if} \; \varepsilon_s < \varepsilon_c, 
\end{equation*}
\begin{equation*}
    \lambda_3 = v_s + \sqrt{\sigma_e^{\prime}(\varepsilon_s)}, \quad \lambda_4 = v_s - \sqrt{\sigma_e^{\prime}(\varepsilon_s)}, \quad \text{if} \; \varepsilon_s > \varepsilon_c.
\end{equation*}
In the free-settling zone ($\sigma_e \equiv 0$), the eigenspace of $A(y)$ is
spanned by $(1,v_s)^\top$ alone:  $A(y)$ is not diagonalizable.
Therefore, the system is strictly hyperbolic in the compression zone and only \textit{weakly
hyperbolic} in the free-settling zone, with a transition occurring at $\eps_c$.
 \begin{remark}[Sed number]\label{rem:sed}
On the compression zone, $\lambda_3=v_s+\sqrt{\sigma_e'(\epss)}$ and
$\lambda_4=v_s-\sqrt{\sigma_e'(\epss)}$. When $v_s>0$, $\lambda_3>0$, and
$\lambda_4<0$ if and only if
\begin{equation}\label{sed_number}
  \mathrm{Sed}\triangleq\frac{v_s}{\sqrt{\sigma_e'(\epss)}}<1 .
\end{equation}
This  ratio plays the role of the Froude number for the
Saint-Venant equations and of the Mach number for the Euler equations. For a stationary admissible profile, \eqref{sed_number} is
\Cref{def:admissible}(ii) with $\dot z_s=0$ (see \eqref{eq:subcritical}).
\end{remark}

\subsection{Admissible interface solutions}

By \eqref{eq:sigmae}, $\sigma_e'$ jumps from $0$ to $L^+$ at $\eps_c$: the
flux is only Lipschitz there, and the system changes type. We now show that
this requires a more general notion of solution. For $\epss>0$ the conservative and
primitive variables are equivalent for $C^1$ solutions, and  in the
primitive variables $(\epss,v_s)$ the
smooth dynamics read
$\partial_t\epss+v_s\partial_z\epss+\epss\partial_zv_s=0$ and
$\partial_tv_s+v_s\partial_zv_s+\frac{\sigma_e'(\epss)}{\epss}\partial_z\epss
=K-\mathcal{R}(\epss)(v_s-v_m)$.

\begin{theorem}[$C^1$ solutions cannot cross $\varepsilon_c$ with a
non-zero gradient]
\label{thm:breakdown}
Let $(\epss,v_s)\in[C^1(\mathcal{W})]^2$, with $\epss>0$, be a classical
solution on an open set $\mathcal{W}\subset(\Rp\setminus\{0\})\times\Om$, the equations being
understood pointwise on $\{\epss\neq\eps_c\}$, where the flux is $C^1$.
There is no point
$(t^*,z^*)\in\mathcal{W}$ at which
$\epss(t^*,z^*)=\eps_c$ and $\partial_z\epss(t^*,z^*)\neq0$.
\end{theorem}

\begin{proof}
Write the momentum equation in the primitive variables $(\epss,v_s)$,
\begin{equation}\label{eq:brk_mom}
  \partial_tv_s+v_s\partial_zv_s
  +\frac{\sigma_e'(\epss)}{\epss}\,\partial_z\epss
  =K-\mathcal{R}(\epss)\bigl(v_s-v_m\bigr),
\end{equation}
and let $(t,z)\to(t^*,z^*)$ within each of
\begin{equation}\label{eq:brk_sides}
  \mathcal{W}^{-}=\mathcal{W}\cap\{\epss<\eps_c\},\qquad
  \mathcal{W}^{+}=\mathcal{W}\cap\{\epss>\eps_c\},
\end{equation}
both of which meet every neighbourhood of $(t^*,z^*)$ because
$\partial_z\epss(t^*,z^*)\neq0$: along the line $t=t^*$, $\epss$ crosses
$\eps_c$ at $z^*$. Since $(\epss,v_s)\in[C^1(\mathcal{W})]^2$, every term of
\eqref{eq:brk_mom} has a limit at $(t^*,z^*)$ along either side, and the two
limits agree for all of them but the third, for which \eqref{eq:sigmae} gives
\begin{equation}\label{eq:brk_limits}
\lim_{\substack{(t,z)\to(t^*,z^*)\\(t,z)\in\mathcal{W}^{-}}}
  \frac{\sigma_e'(\epss)}{\epss}\,\partial_z\epss=0,
  \qquad
  \lim_{\substack{(t,z)\to(t^*,z^*)\\(t,z)\in\mathcal{W}^{+}}}
  \frac{\sigma_e'(\epss)}{\epss}\,\partial_z\epss
  =\frac{L^{+}}{\eps_c}\,\partial_z\epss(t^*,z^*).
\end{equation}
Subtracting the two one-sided limits of \eqref{eq:brk_mom} therefore yields
\begin{equation}\label{eq:brk_contra}
  \frac{L^{+}}{\eps_c}\,\partial_z\epss(t^*,z^*)=0,
\end{equation}
and $L^{+}>0$ forces $\partial_z\epss(t^*,z^*)=0$, contradicting the
hypothesis.
\end{proof}

A $C^1$ solution can therefore reach $\eps_c$ only with a flat profile: a
compression zone that forms inside a free-settling suspension with a non-flat
front cannot be described by a $C^1$ solution.

We therefore work with piecewise classical solutions separated by a moving
interface, under algebraic admissibility conditions.

\begin{definition}[Admissible interface solution]\label{def:admissible}
Let $T>0$, $v_m\in C^2([0,T);\Rp)$ and $\epsin\in C^1([0,T);(0,\eps_c))$. A state
$y=(\epss,f_s)^\top$ is an \emph{admissible interface solution} of
\eqref{eq:model}--\eqref{eq:BCbot} if it is a weak solution of
\eqref{eq:model} in $(0,T)\times\Om$ that satisfies
\eqref{eq:BCtop}--\eqref{eq:BCbot} for all $t\in[0,T)$, and there is a curve
$z_s\in C^1([0,T);(0,z_b))$ splitting $\Om$ into $\Om_1(t)=(0,z_s(t))$ and
$\Om_2(t)=(z_s(t),z_b)$, on each of which $y$ is a classical $C^1$ solution up
to the boundary, with the additional regularity $\partial_z^2v_s\in C^0$ on the closure of
the free-settling zone $\Om_1(t)$, $v_s=f_s/\epss$ being the particle velocity,
and if for all $t\in[0,T)$ the following hold:
\begin{enumerate}[label=(\roman*),leftmargin=2em,itemsep=1pt]
\item \emph{phase separation:}
  $0<\epss<\eps_c$ on $\Om_1(t)$, $\eps_c<\epss<1$ on $\Om_2(t)$, and
  $\epss(t,z_s(t)^{-})<\eps_c<\epss(t,z_s(t)^{+})$;
\item \emph{kinematic admissibility:}
  $\inf_{\Om_1(t)}\bigl(v_s-\tfrac{z}{z_s}\dot z_s\bigr)>0$ and, on
  $\Om_2(t)$,
  $\inf\bigl(v_s+\sqrt{\sigma_e'}-\tfrac{z_b-z}{z_b-z_s}\dot z_s\bigr)>0$,
  $\sup\bigl(v_s-\sqrt{\sigma_e'}-\tfrac{z_b-z}{z_b-z_s}\dot z_s\bigr)<0$.
\end{enumerate}
\end{definition}

Condition (ii) says how the four characteristics meet the interface. At
$z_s(t)$ it reads $v_s^->\dot z_s$ and $\lambda_4^+<\dot z_s<\lambda_3^+$. Above the interface the two waves travel at the same speed $v_s$,  larger than $\dot z_s$, so both reach it. Below it $\lambda_3>\dot z_s$ and $\lambda_4<\dot z_s$, so the first moves away and the second reaches it. Three characteristics arrive and one leaves, which is the count of the Lax condition for a shock of the slow
family~\cite{lax1973hyperbolic,dafermos2005hyperbolic}, stated there for
strictly hyperbolic systems; here the state above the interface is in the
weakly hyperbolic zone, and only this count is used (\appwp{}). The infima in (ii) keep every characteristic speed away from the speed of the
interface, so that the boundaries are non-characteristic (\appwp{}).

\begin{remark}[Loss of one derivative due to the Jordan block]\label{rem:derivloss}
On the free-settling zone $\mathcal{R}=\gamma$ and $\sigma_e=0$, so in
primitive variables
\begin{equation*}
  \partial_tv_s+v_s\partial_zv_s=K-\gamma(v_s-v_m),\qquad
  \partial_t\epss+v_s\partial_z\epss=-\epss\,\partial_zv_s .
\end{equation*}
The first equation is closed; the second is a transport equation whose source
contains $\partial_zv_s$. Let $X(t;\alpha)$ be the characteristic,
$\partial_tX=v_s(t,X)$, $X(0;\alpha)=\alpha$, and $p=\partial_zv_s$.
Differentiating the first equation in $z$ gives $\frac{\dd}{\dd t}p(t,X)=-p^2-\gamma p$
along $X$, and the second equation then gives (\Cref{lem:app_char})
\begin{equation}\label{eq:derivloss}
  \epss\bigl(t,X(t;\alpha)\bigr)=\epss(0,\alpha)\,
  \frac{\gamma}{\gamma+\partial_zv_s(0,\alpha)\bigl(1-e^{-\gamma t}\bigr)} .
\end{equation}
Hence $\epss(t,\cdot)\in C^1$ requires $\partial_zv_s(0,\cdot)\in C^1$, that
is, $\partial_z^2v_s\in C^0$. At the inlet, $v_s(t,0)=v_m(t)+K/\gamma$ by
\eqref{eq:BCtop}, and the velocity equation gives
\begin{equation*}
  \partial_zv_s(t,0)=-\frac{\dot v_m(t)}{v_s(t,0)},
\end{equation*}
so $\partial_z^2v_s(t,0)$ contains $\ddot v_m(t)$, and \Cref{thm:wp} needs
$v_m\in C^2$. This is the nonlinear counterpart of the $H^1$ component of
\eqref{eq:H}. In the linearized closed loop, $u$ is $C^\infty$
(\Cref{thm:fwd_GES}(ii)); we do not apply \Cref{thm:wp} to the closed loop
in this paper (see \Cref{ssec:nonlinear}).
\end{remark}

\begin{lemma}[Rankine--Hugoniot]\label{lem:RH}
Along an admissible interface solution,
\begin{equation}\label{eq:RH}
  \jump{\epss}\dot z_s=\jump{f_s},\qquad
  \jump{f_s}\dot z_s=\jump{F},
\end{equation}
and since $\jump{\epss}>0$ these determine
$\dot z_s=\jump{f_s}/\jump{\epss}$ and
$\jump{f_s}^2=\jump{\epss}\jump{F}$.
\end{lemma}

\begin{proof}
See e.g.\ \cite[Ch.~1]{godlewski_raviart_2021}.
\end{proof}

\subsection{Well-posedness}
\label{ssec:wp}

We first check that \Cref{def:admissible} defines a solvable
initial--boundary value problem.  Two sets of conditions are needed: one on the initial state, which must
already have the structure of an interface solution, and one linking the
boundary inputs to the initial data at $t=0$.

\begin{assumption}[Initial structural consistency]\label{ass:wp}
The initial state $y_0=(\eps_{s,0},f_{s,0})^\top$ is piecewise $C^1$ with a
single discontinuity at $z_s(0)\in(0,z_b)$, splitting $\Om$ into
$\Om_{1,0}=(0,z_s(0))$ and $\Om_{2,0}=(z_s(0),z_b)$, with
$v_{s,0}=f_{s,0}/\eps_{s,0}\in C^2(\overline{\Om_{1,0}})$ on the free-settling
zone, and with
$\dot z_{s,0}\triangleq\jump{f_{s,0}}/\jump{\eps_{s,0}}$ it satisfies
\begin{enumerate}[label=(\roman*),leftmargin=2em,itemsep=1pt]
\item \emph{initial phase separation:}
  $0<\eps_{s,0}<\eps_c$ on $\Om_{1,0}$ and $\eps_c<\eps_{s,0}<1$ on
  $\Om_{2,0}$, with $\eps_{s,0}(z_s(0)^{-})<\eps_c<\eps_{s,0}(z_s(0)^{+})$;
\item \emph{initial kinematic admissibility:} the three strict inequalities of
  \Cref{def:admissible}(ii) hold at $t=0$ with $\dot z_s$ replaced by
  $\dot z_{s,0}$;
\item \emph{initial Rankine--Hugoniot:} the traces of $y_0$ lie on the
  Hugoniot locus of \Cref{lem:RH},
  $\jump{f_{s,0}}^2=\jump{\eps_{s,0}}\jump{F(y_0)}$.
\end{enumerate}
\end{assumption}

\begin{definition}[$C^2$-compatibility]\label{def:compat}
The inputs $v_m\in C^2(\Rp;\Rp)$ and $\epsin\in C^1(\Rp;(0,\eps_c))$ are
\emph{$C^2$-compatible} with $y_0$ if, at $t=0$, the traces of $y_0$ match the
boundary conditions \eqref{eq:BCtop}--\eqref{eq:BCbot},
\begin{equation}\label{eq:compat0}
  \eps_{s,0}(0)=\epsin(0),\quad
  f_{s,0}(0)=\epsin(0)\Bigl(v_m(0)+\tfrac{K}{\gamma}\Bigr),\quad
  f_{s,0}(z_b)=\eps_{s,0}(z_b)v_m(0),
\end{equation}
and their first time derivatives match the dynamics evaluated at the
boundaries,
\begin{equation}\label{eq:compat1}
\begin{gathered}
  \dot\epsin(0)=\partial_t\epss(0,0),\qquad
  \dot\epsin(0)\Bigl(v_m(0)+\tfrac{K}{\gamma}\Bigr)+\epsin(0)\dot v_m(0)
  =\partial_tf_s(0,0),\\
  \partial_tf_s(0,z_b)=\partial_t\epss(0,z_b)\,v_m(0)
  +\eps_{s,0}(z_b)\,\dot v_m(0),
\end{gathered}
\end{equation}
the initial time derivatives being computed from $y_0$ through
\eqref{eq:model}, and if, at the inlet corner,  the second-order condition
\begin{equation}\label{eq:compat2}
  \ddot v_m(0)=v_{s,0}^2\,\partial_z^2v_{s,0}
  +v_{s,0}\,\partial_zv_{s,0}\bigl(\partial_zv_{s,0}+\gamma\bigr)
  -\dot v_m(0)\,\partial_zv_{s,0}\quad\text{at }z=0
\end{equation}
holds, where $v_{s,0}=f_{s,0}/\eps_{s,0}$. This condition is the second
time derivative at $t=0$ of $f_s(t,0)/\epss(t,0)=v_m(t)+K/\gamma$, computed with
the velocity equation $\partial_tv_s+v_s\partial_zv_s=K-\gamma(v_s-v_m)$ of
the free-settling zone, and is a constraint on $\partial_z^2v_{s,0}(0)$.
Finally, with $H\triangleq\jump{f_s}^2-\jump{\epss}\jump{F}$, which vanishes
at $t=0$ by \Cref{ass:wp}(iii), the interface condition
\begin{equation}\label{eq:compat_int}
  \frac{\dd}{\dd t}H\Big|_{t=0}=0
\end{equation}
holds, the time derivatives of the four traces being computed from $y_0$
through \eqref{eq:model} in the frame moving at the speed $\dot z_{s,0}$.
\end{definition}

The \emph{margin function} appearing in \Cref{thm:wp} below collects the
norms of the state and the inverses of the quantities that must not vanish
for the construction to work,
\begin{equation}\label{eq:margin}
\begin{aligned}
  \Nc_0(t)&\triangleq
  \sum_{i=1}^{2}\bigl\|y(t,\cdot)\bigr\|_{C^0(\overline{\Om_i(t)})}
  +\frac{1}{\jump{\epss}(t)}+\frac{1}{m_{\rm kin}(t)}+\frac{1}{m_{\rm phase}(t)}\\
  &\qquad+\frac{1}{\mathrm{dist}\bigl(z_s(t),\partial\Om\bigr)}
  +\frac{1}{\inf_{\Om}\epss(t,\cdot)}
  +\frac{1}{1-\sup_{\Om}\epss(t,\cdot)},\\
  \Nc(t)&\triangleq\Nc_0(t)
  +\sum_{i=1}^{2}\bigl\|\partial_zy(t,\cdot)\bigr\|_{C^0(\overline{\Om_i(t)})}
  +\bigl\|\partial_z^2v_s(t,\cdot)\bigr\|_{C^0(\overline{\Om_1(t)})},
\end{aligned}
\end{equation}
where the norms on $\overline{\Om_i(t)}$ use the one-sided traces at
$z_s(t)$, and the three geometric margins are
\begin{equation}\label{eq:margins}
\begin{aligned}
  m_{\rm kin}&\triangleq\min\Bigl\{
    \inf_{\Om_1(t)}\Bigl(v_s-\tfrac{z}{z_s}\dot z_s\Bigr),\
    \inf_{\Om_2(t)}\Bigl(\sqrt{\sigma_e'(\epss)}
    -\Bigl|v_s-\tfrac{z_b-z}{z_b-z_s}\dot z_s\Bigr|\Bigr)\Bigr\},\\
  m_{\rm phase}&\triangleq\min\Bigl\{
    \eps_c-\sup_{\Om_1(t)}\epss(t,\cdot),\
    \inf_{\Om_2(t)}\epss(t,\cdot)-\eps_c\Bigr\},\qquad
  \mathrm{dist}\bigl(z_s,\partial\Om\bigr)\triangleq\min\{z_s,\,z_b-z_s\}.
\end{aligned}
\end{equation}
The two terms of $m_{\rm kin}$ are the margins in the three inequalities of
\Cref{def:admissible}(ii), since $\min\{c+a,c-a\}=c-|a|$: they bound the
characteristic speeds relative to the affine motion of each zone away from
zero.
The last two terms of $\Nc_0$ keep $\epss$ away from $0$, where the flux
Jacobian \eqref{eq:jac} is unbounded at fixed $f_s$, and from $1$, near
which $\sigma_e''$ and $\mathcal{R}'$ need not be bounded.

\begin{theorem}[Well-posedness of admissible interface solutions]\label{thm:wp}
Let $y_0$ satisfy \Cref{ass:wp}, and let the inputs $v_m\in C^2(\Rp;\Rp)$ and
$\epsin\in C^1(\Rp;(0,\eps_c))$  be $C^2$-compatible with $y_0$ (\Cref{def:compat}), with
\begin{equation*}
  \|v_m\|_{C^2(\Rp)}+\|\epsin\|_{C^1(\Rp)}<\infty,\qquad
  0<\inf_{\Rp}\epsin\le\sup_{\Rp}\epsin<\eps_c .
\end{equation*}
Then there is a maximal time $T_{\max}\in(0,+\infty]$ and a unique admissible
interface solution of \eqref{eq:model}--\eqref{eq:BCbot} on $[0,T_{\max})$
(\Cref{def:admissible}) with $y(0,\cdot)=y_0$ and interface starting at
$z_s(0)$; moreover $z_s\in C^2([0,T_{\max}))$. If
$T_{\max}<\infty$, then $\limsup_{t\to T_{\max}^-}\Nc(t)=\infty$, and there is
a sequence $t_n\to T_{\max}^-$ along which one of the following occurs:
\begin{enumerate}[label=(\roman*),leftmargin=2.2em,itemsep=1pt]
\item \emph{boundary collision:} $\mathrm{dist}\bigl(z_s(t_n),\partial\Om\bigr)\to0$;
\item \emph{interior gradient blow-up:}  $\max_{i=1,2}\|\partial_zy(t_n,\cdot)\|_{C^0(\overline{\Om_i(t_n)})}\to\infty$;
\item \emph{loss of admissibility:} $m_{\rm phase}(t_n)\to0$ or
  $m_{\rm kin}(t_n)\to0$;
\item \emph{degeneracy of the model:} $\inf_{\Om}\epss(t_n,\cdot)\to0$ or
  $\sup_{\Om}\epss(t_n,\cdot)\to1$.
\end{enumerate}
\end{theorem}

The other terms of $\Nc$ give no further scenario: $\jump{\epss}\ge2m_{\rm phase}$,
and the state and $\partial_z^2v_s$ stay bounded as long as the margins and
$\partial_zy$ do (\Cref{lem:app_ctrl}).

The proof is in \appwp{}. \Cref{thm:wp} is a continuation criterion: the
solution exists as long as $\Nc$ stays finite, and items (i)--(iv) list the
ways in which $\Nc$ can blow up. We do not claim that they are the only ways
in which a solution of \eqref{eq:model} can fail.

\subsection{Stationary admissible profiles}
\label{ssec:steady}

Let $\epsin\in(0,\eps_c)$ and $v_m^*>0$ be constant. A stationary admissible
solution $y^*(z)=(\epss^*(z),f_s^*)^\top$ with interface at $z_s^*\in(0,z_b)$ has a
spatially uniform flux $f_s^*$, and \eqref{eq:BCtop} gives
\begin{equation}\label{eq:steadytop}
  f_s^*=\epsin\Bigl(v_m^*+\tfrac{K}{\gamma}\Bigr),\qquad
  \epss^*\equiv\epsin \text{ on } (0,z_s^*),
\end{equation}
$\epsin$ being an equilibrium of the free-settling steady equation. On
$\Om_2^*\triangleq(z_s^*,z_b)$ the profile solves
\begin{equation}\label{eq:steadyODE}
  \frac{d}{dz}F(\epss^*,f_s^*)=S_0(\epss^*,f_s^*)+\Psi_0(\epss^*)v_m^* ,
\end{equation}
with the interface value fixed by the stationary jump $\jump{F^*}=0$,
\begin{equation}\label{eq:steadyRH}
  \frac{(f_s^*)^2}{\epsin}=\frac{(f_s^*)^2}{\epss^*(z_s^{*+})}
  +\sigma_e\bigl(\epss^*(z_s^{*+})\bigr),
\end{equation}
and the bottom condition $\epss^*(z_b)=f_s^*/v_m^*$. We assume throughout
that such a profile exists for the target $z_s^*$; \appsteady{} gives conditions
for its existence.  Admissibility reduces to
$\eps_c<\epss^*<1$ on $[z_s^*,z_b]$ and the strict subcriticality
\begin{equation}\label{eq:subcritical}
  v_s^*(z)<\sqrt{\sigma_e'\bigl(\epss^*(z)\bigr)},\qquad z\in[z_s^{*},z_b],
\end{equation}
with $\epss^*(z_s^*)$ understood as the trace $\epss^*(z_s^{*+})$.
Given $z_s^*$, the pair $(v_m^*,\epss^*)$ is obtained by a \textit{shooting method} on
$v_m^*$, carried out in \appsteady{}, and $\lambda\triangleq f_s^*/\epsin=v_m^*+K/\gamma$ is the
free-settling characteristic speed.

\section{Fixed-domain reformulation and the linearized interconnection}
\label{sec:fixed}

Since the equilibrium is discontinuous at the interface, a small displacement
of the interface is not a small deviation of the state in the physical
variables, and the system cannot be linearized in them. We  first fix the
boundaries by a change of variables. The system then becomes four coupled
quasilinear hyperbolic PDEs, two on each zone,  with an ordinary differential
equation for the interface. It  is then linearized in the usual way.

\subsection{From a moving to a fixed boundary}

Let $\Om_1^*=(0,z_s^*)$, $\Om_2^*=(z_s^*,z_b)$ and define, for each $t$, the
piecewise affine map
\begin{equation}\label{eq:diffeo}
  x=\begin{cases}
    z\,\dfrac{z_s^*}{z_s(t)}, & z\in(0,z_s(t)),\\[8pt]
    z_b-(z_b-z)\dfrac{z_b-z_s^*}{z_b-z_s(t)}, & z\in(z_s(t),z_b).
  \end{cases}
\end{equation}
Since $z_s(t)\in(0,z_b)$, \eqref{eq:diffeo} is, at each $t$, an affine
bijection of each zone onto its reference interval, with the time regularity of
$z_s$, which is $C^2$ by \Cref{thm:wp}.
Unlike in \appwp{}, which maps onto  the initial position $z_s(0)$, the
target  $z_s^*$ is used here, so that the equilibrium profile does not depend
on time in the new variable. Let $y_c^{(i)}(t,x)\triangleq y(t,z)$, where $z\in\Om_i(t)$ and
$x\in\Om_i^*$ are related by \eqref{eq:diffeo}. Then \eqref{eq:model} becomes
\begin{equation}\label{eq:FBVP}
  \partial_ty_c^{(i)}+A_c^{(i)}(x,y_c^{(i)},z_s,\dot z_s)\,\partial_xy_c^{(i)}
  =S_c^{(i)}(y_c^{(i)},v_m),\qquad x\in\Om_i^*,
\end{equation}
where $A$ is the flux Jacobian \eqref{eq:jac},
\begin{equation}\label{eq:Ac}
  A_c^{(1)}=\frac{z_s^*}{z_s}\,A(y_c^{(1)})-\frac{x\,\dot z_s}{z_s}\,\mathbf{I},
  \qquad
  A_c^{(2)}=\frac{z_b-z_s^*}{z_b-z_s}\,A(y_c^{(2)})
  -\frac{(z_b-x)\,\dot z_s}{z_b-z_s}\,\mathbf{I},
\end{equation}
and $S_c^{(i)}(y,v_m)=\bigl(0,\,S_0(y)+\Psi_0(\epss)\,v_m\bigr)^\top$ is the
source of \eqref{eq:model}, which the change of variables does not modify.
The boundary conditions \eqref{eq:BCtop}--\eqref{eq:BCbot} do not change, and
the jump relations \eqref{eq:RH} are now written at the fixed point
$x=z_s^*$. Note that $\dot z_s$ appears in the coefficients of
\eqref{eq:Ac}, and that the equilibrium $y^*$ is smooth on each $\Om_i^*$,
since its jump is at $x=z_s^*$.

\subsection{Linearization and Riemann coordinates}

Let
\begin{equation*}
  u(t)=v_m(t)-v_m^*,\qquad \eta(t)=z_s(t)-z_s^*,\qquad
  \tilde y_c^{(i)}=y_c^{(i)}-y^*,\qquad Y^{(i)}=\mathbf{M}_i\tilde y_c^{(i)},
\end{equation*}
with $\tilde y_c^{(i)}=(\tilde\epss,\tilde f_s)^\top$, $Y^{(1)}=(Y_1,Y_2)^\top$
and $Y^{(2)}=(Y_3,Y_4)^\top$, where the matrices $\mathbf{M}_i$, given below, put the Jacobian in Jordan
form on $\Om_1^*$ and in diagonal form on $\Om_2^*$. The feed fraction
$\epsin$ is kept at its equilibrium value. We linearize \eqref{eq:FBVP} around
$(y^*,z_s^*)$ and write the result in the variables $Y^{(i)}$.

\begin{remark}[Transported and physical deviations]\label{rem:frame}
All deviations below are deviations of the transported state,
\begin{equation*}
  \tilde y_c(t,x)=y_c(t,x)-y^*(x),\qquad y_c(t,x)=y\bigl(t,z(t,x)\bigr),
\end{equation*}
with $z(t,\cdot)$ the inverse of \eqref{eq:diffeo}, and not the physical
deviation $y(t,z)-y^*(z)$. For example, the equilibrium mapped onto a displaced
interface, $y(t,z)=y^*(x(t,z))$ with $z_s\neq z_s^*$, has $\tilde y_c\equiv0$
and $\eta\neq0$, while its physical deviation is of size $\jump{\epss^*}$ on
an interval of length $|\eta|$ next to $z_s^*$. This has two consequences.
First, the interface traces are $y^*(z_s^{*\pm})+\tilde y_c(t,z_s^{*\pm})$, so
$\partial_zy^*$ does not appear in the jump relations: the interface enters
the linearized PDEs only through the coefficients $z_s$ and $\dot z_s$ of
\eqref{eq:FBVP}, that is, through the term $\Cb^{(2)}\xi$ of
\eqref{eq:zone2}. Second, the physical deviation cannot be used to linearize:
to leading order, it differs from $\tilde y_c$ by a quantity of size
$|\eta|^{1/2}\jump{\epss^*}$ in $L^2$, which is not $O(|\eta|)$.
\end{remark}

On the free-settling zone the equilibrium is constant, so the coefficients
are constant.
 With $\mathbf{M}_1=\bigl[\begin{smallmatrix}1&0\\
-\lambda&1\end{smallmatrix}\bigr]$, so that $Y_1=\tilde\epss$ and
$Y_2=\tilde f_s-\lambda\tilde\epss$,
\begin{equation}\label{eq:zone1}
  \partial_tY^{(1)}+
  \underbrace{\begin{bmatrix}\lambda&1\\0&\lambda\end{bmatrix}}_{\Ab^{(1)}}
  \partial_xY^{(1)}
  =\underbrace{\begin{bmatrix}0&0\\0&-\gamma\end{bmatrix}}_{\Bb^{(1)}}Y^{(1)}
  +\underbrace{\begin{bmatrix}0\\ \gamma\epsin\end{bmatrix}}_{\Db^{(1)}}u,
  \qquad \Cb^{(1)}=\mathbf{0},
\end{equation}
where $\Cb^{(1)}$ is the matrix of the terms in $(\eta,\dot z_s)$. It is zero
since  $\partial_xy^*$ is zero on $\Om_1^*$. The matrix
$\Ab^{(1)}$ is a Jordan block: the system is only weakly hyperbolic on this
zone.

On the compression zone, with
$c_s\triangleq\sqrt{\sigma_e'(\epss^*)}$,
$\lambda_3=v_s^*+c_s>0>\lambda_4=v_s^*-c_s$
and $\mathbf{M}_2^{-1}=\bigl[\begin{smallmatrix}1&1\\
\lambda_3&\lambda_4\end{smallmatrix}\bigr]$,
\begin{equation}\label{eq:zone2}
  \partial_tY^{(2)}+\Ab^{(2)}(x)\partial_xY^{(2)}
  =\Bb^{(2)}(x)Y^{(2)}+\Cb^{(2)}(x)\xi+\Db^{(2)}(x)u,
  \qquad \Ab^{(2)}=\mathrm{diag}(\lambda_3,\lambda_4),
\end{equation}
where $\xi=(\eta,\dot z_s)^\top$ collects the two interface quantities coming
from \eqref{eq:Ac}, all coefficients being evaluated at $x$,  and
\begin{equation}\label{eq:C2}
  \Cb^{(2)}(x)=\frac{\partial_x\epss^*(x)}{\lambda_4-\lambda_3}
  \begin{bmatrix}
  -\frac{\lambda_3\lambda_4}{z_b-z_s^*} & \frac{z_b-x}{z_b-z_s^*}\lambda_4\\[4pt]
  \frac{\lambda_3\lambda_4}{z_b-z_s^*} & -\frac{z_b-x}{z_b-z_s^*}\lambda_3
  \end{bmatrix},
  \qquad
  \Db^{(2)}(x)=\frac{\Psi_0(\epss^*)}{2\sqrt{\sigma_e'(\epss^*)}}
  \begin{bmatrix}1\\-1\end{bmatrix},
\end{equation}
\begin{equation}\label{eq:B2}
  \Bb^{(2)}(x)=\frac{1}{\lambda_4-\lambda_3}
  \left(\begin{bmatrix}-1\\ 1\end{bmatrix}
  \begin{bmatrix}\beta_3&\beta_4\end{bmatrix}
  +\begin{bmatrix}\lambda_3\\ -\lambda_4\end{bmatrix}
  \begin{bmatrix}\lambda_3'&\lambda_4'\end{bmatrix}\right),
\end{equation}
\begin{equation}\label{eq:beta34}
  \beta_j=K-\mathcal{R}'(\epss^*)\bigl(f_s^*-\epss^*v_m^*\bigr)
  +\mathcal{R}(\epss^*)\bigl(v_m^*-\lambda_j\bigr)
  -\partial_x\epss^*\Bigl[\sigma_e''(\epss^*)
  +\frac{2v_s^*}{\epss^*}\bigl(v_s^*-\lambda_j\bigr)\Bigr],
  \qquad j=3,4,
\end{equation}

with $v_s^*-\lambda_3=-c_s$ and $v_s^*-\lambda_4=c_s$. The matrix
\eqref{eq:B2} has three contributions: the Jacobian of the source, which gives
the terms in $K$ and $\mathcal{R}$; the derivative of the flux Jacobian along
the profile, $(D_yA[\,\cdot\,])\partial_xy^*$, which gives the bracket in
$\sigma_e''$; and the $x$-dependence of $\mathbf{M}_2$, which gives the second
rank-one term. The last two vanish when $\partial_x\epss^*=0$.

The boundary conditions become
\begin{equation}\label{eq:bc}
  Y_1(t,0)=0,\quad Y_2(t,0)=\epsin u(t),\qquad
  Y_4(t,z_b)=Y_3(t,z_b)+\Db^{(bot)}u(t),
\end{equation}
with $\Db^{(bot)}=\epss^*(z_b)/(\lambda_4(z_b)-v_m^*)=-\epss^*(z_b)/c_s(z_b)<0$,
since $v_s^*(z_b)=v_m^*$; note that the bottom reflection coefficient is
exactly $1$. Since  the feed fraction is fixed, $Y_1=\tilde\epss$ vanishes at
$x=0$, and by \eqref{eq:zone1} it is driven on $\Om_1^*$ only by
$\partial_xY_2$, through the Jordan block.

The jump relations \eqref{eq:RH} give the missing boundary condition, on the
wave $Y_3$ that enters the compression zone at the interface. In
\eqref{eq:FBVP} the interface is at $x=z_s^*$, so the traces are
$y^{\pm}=y^*(z_s^{*\pm})+\tilde y_c^{\pm}$, with
$\tilde y_c^{\pm}\triangleq\tilde y_c(t,z_s^{*\pm})$. At equilibrium,
$\dot z_s=0$, $\jump{f_s^*}=0$ since $f_s^*$ is uniform, and
$\jump{F(y^*)}=0$ by \eqref{eq:steadyRH}. Hence, to first order,
\eqref{eq:RH} gives
\begin{equation}\label{eq:RHlin}
  \jump{\epss^*}\dot\eta=\tilde f_s^{+}-\tilde f_s^{-},\qquad
  \mathrm{d}F(y^{*+})\,\tilde y_c^{+}=\mathrm{d}F(y^{*-})\,\tilde y_c^{-},
\end{equation}
where, by \eqref{eq:FSP},
$\mathrm{d}F(y)\tilde y=(\sigma_e'(\epss)-v_s^2)\tilde\epss+2v_s\tilde f_s$.
On the free-settling side, $\sigma_e'=0$, $v_s^*=\lambda$,
$\tilde\epss=Y_1$ and $\tilde f_s=\lambda Y_1+Y_2$. On the compression side,
$\sigma_e'-v_s^{*2}=-\lambda_3\lambda_4$, $2v_s^*=\lambda_3+\lambda_4$,
$\tilde\epss=Y_3+Y_4$ and $\tilde f_s=\lambda_3Y_3+\lambda_4Y_4$. Hence, at
$x=z_s^*$,
\begin{equation}\label{eq:tracesRiemann}
\begin{aligned}
  \tilde f_s^{-}&=\lambda Y_1+Y_2, &
  \mathrm{d}F(y^{*-})\tilde y_c^{-}&=\lambda^2Y_1+2\lambda Y_2,\\
  \tilde f_s^{+}&=\lambda_3Y_3+\lambda_4Y_4, &
  \mathrm{d}F(y^{*+})\tilde y_c^{+}&=\lambda_3^2Y_3+\lambda_4^2Y_4,
\end{aligned}
\end{equation}
and the second relation of \eqref{eq:RHlin} gives the incoming wave in terms
of the three outgoing traces:
\begin{equation}\label{eq:Kint}
  Y_3(t,z_s^*)=\Kint\,\Yo(t),\qquad
  \Kint=\Bigl[\tfrac{\lambda^2}{\lambda_3^2},\
  \tfrac{2\lambda}{\lambda_3^2},\ -\tfrac{\lambda_4^2}{\lambda_3^2}\Bigr],\qquad
  \Yo\triangleq(Y_1,Y_2,Y_4)(t,z_s^*).
\end{equation}
With \eqref{eq:bc}, the four incoming traces $Y_1(0)$, $Y_2(0)$,
$Y_3(z_s^*)$ and $Y_4(z_b)$ are given by the outgoing ones and by $u$. The
last entry of $\Kint$, $-k^2$ with $k\triangleq|\lambda_4|/\lambda_3<1$, is the
reflection coefficient of the interface for the compression zone.

\subsection{The linearized interface dynamics}
\label{ssec:kinematics}

The first relation of \eqref{eq:RHlin}
reduces the pair $\xi$ to the single state $\eta$: substituting
\eqref{eq:tracesRiemann} and eliminating $Y_3$ by \eqref{eq:Kint} gives a
scalar ODE for the interface,
\begin{equation}\label{eq:kinematics}
  \dot\eta=\Go\,\Yo,\qquad
  \Go=\frac{1}{\jump{\epss^*}}\Bigl[\tfrac{\lambda(\lambda-\lambda_3)}{\lambda_3},\
  \tfrac{2\lambda-\lambda_3}{\lambda_3},\
  \tfrac{\lambda_4(\lambda_3-\lambda_4)}{\lambda_3}\Bigr],
\end{equation}
all coefficients evaluated at $z_s^*$; we write $\Go=(G_{o,1},G_{o,2},G_{o,4})$,
the indices referring to the components of $\Yo$.

This gives
\begin{equation}\label{eq:closedcoupling}
  \Cb^{(2)}\xi=\Cb^{(2)}_\eta\,\eta+\Cb^{(2)}_{:,2}\,\Go\Yo,
  \qquad \Cb^{(2)}_\eta\triangleq\Cb^{(2)}_{:,1}.
\end{equation}

\subsection{The linearized plant}

Collecting \eqref{eq:zone1}--\eqref{eq:closedcoupling}, the system to be
stabilized is
\begin{equation}\label{eq:plant}
\left\{\begin{aligned}
 &\partial_tY^{(i)}+\Ab^{(i)}\partial_xY^{(i)}
  =\Bb^{(i)}Y^{(i)}+\Cb^{(i)}_\eta\eta+\Cb^{(i)}_{:,2}\Go\Yo+\Db^{(i)}u,
  && i=1,2,\\
 &\dot\eta=\Go\Yo,\\
 &\text{boundary conditions \eqref{eq:bc}},
\end{aligned}\right.
\end{equation}
on the energy space
\begin{equation}\label{eq:H}
  \Hs=L^2(\Om_1^*)\times H^1(\Om_1^*)\times L^2(\Om_2^*)\times L^2(\Om_2^*),
\end{equation}
the $H^1$ term being imposed by the Jordan block: the first row of
$-\Ab^{(1)}\partial_xY^{(1)}$ contains $\partial_xY_2$, which must lie in
$L^2$. The traces $\Yo$ in \eqref{eq:plant} are those of regular solutions
($Y_1,Y_3,Y_4\in H^1$ and $Y_2\in H^2$ at each time, as in \Cref{thm:ISS});
in the semigroup formulation of \Cref{ssec:feedback}, this regularity is
part of the domain of the generator.
\Cref{fig:plant} shows the couplings of \eqref{eq:plant}.

\begin{figure}[t]
\centering
\begin{tikzpicture}[x=0.82cm,y=1cm,>=Latex,font=\small,
  wave/.style={->,line width=1.1pt},
  lab/.style={font=\footnotesize,inner sep=1pt},
  box/.style={draw,rounded corners=2pt,align=center,inner sep=4pt,font=\footnotesize}]
  \fill[blue!7]  (0,0) rectangle (5,2.6);
  \fill[gray!12] (5,0) rectangle (12,2.6);
  \draw[gray!60] (0,0) rectangle (12,2.6);
  \draw[line width=2pt,red!55!black] (5,-0.05) -- (5,2.65);
  \node[lab,gray!40!black] at (2.5,0.28) {$\Om_1^*$: weakly hyperbolic};
  \node[lab,gray!40!black] at (8.5,0.28) {$\Om_2^*$: strictly hyperbolic};
  \draw[wave,blue!60!black] (0.35,2.05) -- (4.65,2.05)
    node[lab,above,pos=0.25,blue!60!black] {$Y_1$, speed $\lambda$};
  \draw[wave,blue!60!black] (0.35,1.05) -- (4.65,1.05)
    node[lab,above,pos=0.25,blue!60!black] {$Y_2$, speed $\lambda$};
  \draw[->,line width=1.1pt,orange!80!black] (4.25,1.12) -- (4.25,1.95);
  \node[lab,orange!80!black,anchor=east] at (4.15,1.55) {$\partial_xY_2$};
  \draw[wave,blue!60!black] (5.35,2.05) -- (11.65,2.05)
    node[lab,above,midway,blue!60!black] {$Y_3$, $\lambda_3>0$};
  \draw[wave,orange!80!black] (11.65,1.05) -- (5.35,1.05)
    node[lab,above,midway,orange!80!black] {$Y_4$, $\lambda_4<0$};
  \draw[->,gray!60!black] (0,-0.55) -- (12.6,-0.55) node[right] {$x$};
  \foreach \p/\t in {0/0,5/z_s^*,12/z_b}
    \draw[gray!60!black] (\p,-0.45) -- (\p,-0.65) node[below,lab] {$\t$};
  \node[box,draw=red!55!black,fill=red!4] (eta) at (5,4.4)
    {interface $\eta$\\[1pt] $\dot\eta=\Go\Yo$};
  \node[box,draw=violet!70!black,fill=violet!5] (mf) at (10,4.4)
    {terms of the fixed-domain map\\[1pt]
     $\Cb^{(2)}_\eta\eta+\Cb^{(2)}_{:,2}\Go\Yo$};
  \draw[->,thick,red!55!black] (4.4,2.65) -- ([xshift=4pt]eta.south west);
  \node[lab,red!55!black,align=right,anchor=east] at (4.2,3.35)
    {outgoing\\traces $\Yo$};
  \draw[->,thick,red!55!black] ([xshift=-4pt]eta.south east) -- (5.6,2.65);
  \node[lab,red!55!black,align=left,anchor=west] at (5.8,3.35)
    {$Y_3(z_s^*)=\Kint\Yo$\\by \eqref{eq:Kint}};
  \draw[->,thick,violet!70!black] (eta.east) -- (mf.west);
  \draw[->,thick,violet!70!black] (mf.south) -- (10,2.65);
  \draw[->,thick,blue!60!black] (-1.9,2.05) -- (0.3,2.05)
    node[lab,above,pos=0.4,blue!60!black] {$Y_1(0)=0$};
  \node[box,draw=teal!70!black,fill=teal!6] (bot) at (14.6,1.05)
    {outlet\\[1pt] $Y_3(z_b)$\\ $+\,\Db^{(bot)}u$};
  \draw[->,thick,blue!60!black] (11.65,2.05) -| (bot.north)
    node[lab,above,pos=0.25,blue!60!black] {$Y_3(z_b)$};
  \draw[->,thick,orange!80!black] (bot.west) -- (11.7,1.05)
    node[lab,above,midway,orange!80!black] {$Y_4(z_b)$};
  \node[box,draw=teal!70!black,fill=teal!6] (u) at (6.2,-2.0) {underflow $u$};
  \draw[->,thick,teal!70!black] (u.west) -- (-0.4,-2.0) -- (-0.4,1.05)
    -- (0.3,1.05);
  \node[lab,teal!40!black,anchor=west] at (0.1,0.72) {$Y_2(0)=\epsin u$};
  \draw[->,thick,teal!70!black] (u.east) -| (bot.south);
  \draw[->,thick,dashed,teal!70!black] (2.5,-2.0) -- (2.5,0.05);
  \draw[->,thick,dashed,teal!70!black] (8.5,-2.0) -- (8.5,0.05);
  \node[lab,teal!40!black,fill=white] at (2.5,-1.35) {drag $\Db^{(1)}u$};
  \node[lab,teal!40!black,fill=white] at (8.5,-1.35) {drag $\Db^{(2)}u$};
\end{tikzpicture}
\caption{Structure of the linearized plant \eqref{eq:plant}. The outgoing
traces $\Yo$ drive the interface $\eta$ through $\Go$ and set the incoming wave
$Y_3$ at $x=z_s^*$ through \eqref{eq:Kint}; $\eta$ acts on $\Om_2^*$ through
the fixed-domain map. The
underflow $u$ enters at both ends and through the drag (dashed). The feed
fraction is fixed, so $Y_1$ is driven only by $\partial_xY_2$.}
\label{fig:plant}
\end{figure}
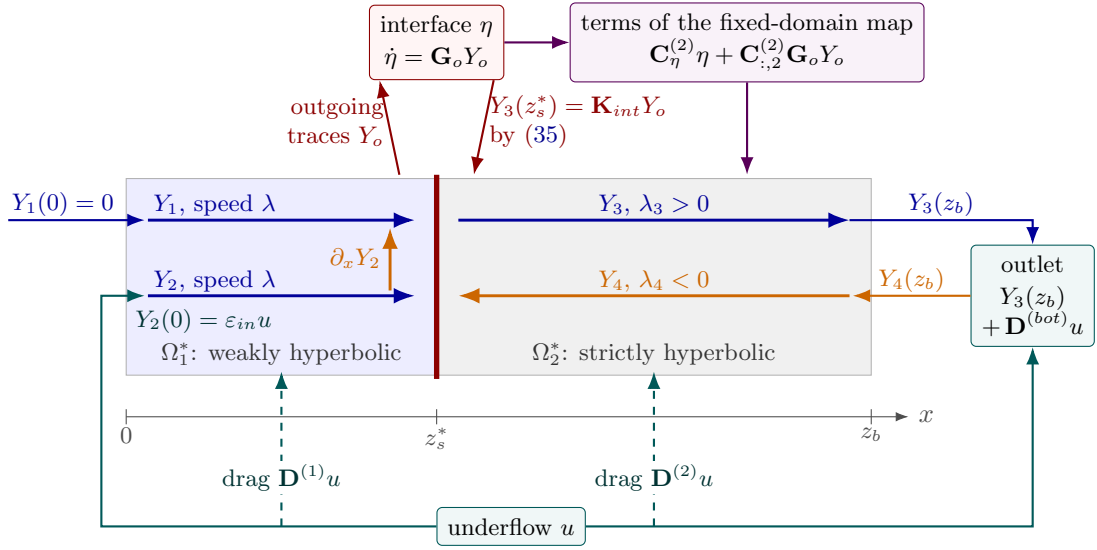

\section{A dissipation inequality for the distributed layer}
\label{sec:iss}
In this section, we  consider the PDE part of \eqref{eq:plant}, with $\eta$
and $u$ seen as inputs, and  build a weighted Lyapunov functional $\Vc$ that
decreases up to terms in $\eta$, $\Go\Yo$, $u$ and $\dot u$
(\Cref{thm:ISS}); \Cref{sec:fwd} uses this inequality.

\subsection{Weighted Lyapunov functional}
\label{ssec:lyap}
\begin{definition}[Lyapunov functional]\label{def:Lyapunov}
Let $p_1,p_2,p_{2x},p_3,p_4>0$ and $\mu_1,\mu_2>0$. We set
\begin{equation}\label{eq:V}
  \Vc(Y)\triangleq\int_0^{z_s^*}\!\!e^{-2\mu_1x}
  \bigl[p_1Y_1^2+p_2Y_2^2+p_{2x}(\partial_xY_2)^2\bigr]\dd x
  +\int_{z_s^*}^{z_b}\!Y^{(2)\top}\Pb(x)Y^{(2)}\dd x ,
\end{equation}
with
\begin{equation}\label{eq:P}
  \Pb(x)\triangleq\mathrm{diag}
  \bigl(p_3e^{-2\mu_2(x-z_s^*)},\,p_4e^{2\mu_2(x-z_b)}\bigr).
\end{equation}
We write $\Wb_1(x)=e^{-2\mu_1x}\mathrm{diag}(p_1,p_2)$ and $\Wb_2=\Pb$ for the
weight matrices acting on $(Y_1,Y_2)$ and $(Y_3,Y_4)$.
\end{definition}

The term in $(\partial_xY_2)^2$ is needed because the Jordan block of
\eqref{eq:zone1} couples $Y_1$ to $\partial_xY_2$. On $\Om_2^*$, each
exponential weight decreases in the direction of its wave ($Y_3$ downwards,
$Y_4$ upwards), as in the strict Lyapunov functionals
of~\cite{amirat2024stability}.

Since the intervals are bounded, the exponentials are bounded away from $0$
and $\infty$, so there are $c_1,c_2>0$ with
\begin{equation}\label{eq:normequiv}
  c_1\|Y\|_\Hs^2\le\Vc(Y)\le c_2\|Y\|_\Hs^2 .
\end{equation}

The weights $p_3,p_4,\mu_2$ of the compression zone must satisfy two
conditions: one on the boundary terms and one on the volume terms of the
derivative of $\Vc$ on $\Om_2^*$ (the computation is in \appiss{}).

\emph{Boundary terms.} The boundary terms of
$\frac{\dd}{\dd t}\int_{\Om_2^*}Y^{(2)\top}\Pb\,Y^{(2)}\dd x$ are
\begin{equation*}
  \Bigl[\,|\lambda_4|\,p_4Y_4^2-\lambda_3\,p_3e^{-2\mu_2(z_b-z_s^*)}Y_3^2
  \Bigr]_{x=z_b}
  +\Bigl[\,\lambda_3\,p_3Y_3^2-|\lambda_4|\,p_4e^{-2\mu_2(z_b-z_s^*)}Y_4^2
  \Bigr]_{x=z_s^*},
\end{equation*}
where $\lambda_3,\lambda_4$ are taken at the same point as the traces. At the
bottom, $Y_4(z_b)=Y_3(z_b)+\Db^{(bot)}u$ by \eqref{eq:bc}, and Young's
inequality with parameter $\nu_b>0$ gives
\begin{equation}\label{eq:youngb}
  Y_4(z_b)^2\le(1+\nu_b)\,Y_3(z_b)^2+(1+\nu_b^{-1})\bigl(\Db^{(bot)}\bigr)^2u^2 .
\end{equation}
At the interface, by \eqref{eq:Kint}, with all traces at $z_s^*$,
\begin{equation}\label{eq:kratio}
  Y_3=K_{int,1}Y_1+K_{int,2}Y_2-k^2Y_4,\qquad
  K_{int,1}=\frac{\lambda^2}{\lambda_3(z_s^*)^2},\quad
  K_{int,2}=\frac{2\lambda}{\lambda_3(z_s^*)^2},
\end{equation}
and Young's inequality with parameter $\nu_i>0$ gives
\begin{equation}\label{eq:youngi}
  Y_3^2\le(1+\nu_i)\,k^4Y_4^2
  +2(1+\nu_i^{-1})\bigl(K_{int,1}^2Y_1^2+K_{int,2}^2Y_2^2\bigr).
\end{equation}
With \eqref{eq:youngb} and \eqref{eq:youngi}, the coefficients of
$Y_3(z_b)^2$ and $Y_4(z_s^*)^2$ in the boundary terms are
\begin{gather*}
  -\lambda_3(z_b)\,p_3e^{-2\mu_2(z_b-z_s^*)}+(1+\nu_b)|\lambda_4(z_b)|\,p_4,\\
  -|\lambda_4(z_s^*)|\,p_4e^{-2\mu_2(z_b-z_s^*)}+(1+\nu_i)k^4\lambda_3(z_s^*)\,p_3 .
\end{gather*}
Since $k^4\lambda_3(z_s^*)/|\lambda_4(z_s^*)|=k^3$, both are negative if
\begin{equation}\label{eq:window}
  (1+\nu_i)\,k^3e^{2\mu_2(z_b-z_s^*)}
  \;<\;\frac{p_4}{p_3}\;<\;
  \frac{e^{-2\mu_2(z_b-z_s^*)}\,\lambda_3(z_b)}{|\lambda_4(z_b)|\,(1+\nu_b)}.
\end{equation}
The left inequality comes from the interface, the right one from the bottom.
Since $k<1$, this interval is non-empty for $\nu_i,\nu_b$ small enough as soon
as $\mu_2<\bar\mu_2(y^*)$, with
\begin{equation}\label{eq:mu2bar}
  \bar\mu_2(y^*)\triangleq\frac{1}{4(z_b-z_s^*)}
  \ln\frac{\lambda_3(z_b)\lambda_3(z_s^*)}{|\lambda_4(z_b)||\lambda_4(z_s^*)|}.
\end{equation}
Since $k^2=(1-\mathrm{Sed})^2/(1+\mathrm{Sed})^2$, with $\mathrm{Sed}$ given
by \eqref{sed_number} at $z_s^{*+}$, and the bottom reflection coefficient is
$1$, the window \eqref{eq:window} is narrow when $\mathrm{Sed}$ is small, and
the interface dissipation $d_4$ of \Cref{ssec:dissip} is then a difference of
close numbers ($k^2=0.9990$ in \Cref{sec:num}).

\emph{Volume terms.} Apart from the terms in $\eta$, $\Go\Yo$ and $u$, the
volume term is
\begin{equation*}
  -\int_{\Om_2^*}Y^{(2)\top}\Qb\,Y^{(2)}\dd x,
\end{equation*}
with
\begin{equation}\label{eq:Q}
  \Qb\triangleq\Pb\,\mathrm{diag}
  \bigl(2\mu_2\lambda_3-\lambda_3',\,2\mu_2|\lambda_4|-\lambda_4'\bigr)
  -\Pb\Bb^{(2)}-\Bb^{(2)\top}\Pb ,
\end{equation}
and it is bounded by $-\bar\alpha\int_{\Om_2^*}Y^{(2)\top}\Pb\,Y^{(2)}\dd x$,
with
\begin{equation}\label{eq:alphabar}
  \bar\alpha(y^*;p_3,p_4,\mu_2)\triangleq
  \inf_{x\in\Om_2^*}\lambda_{\min}
  \bigl(\Pb(x)^{-1/2}\Qb(x)\Pb(x)^{-1/2}\bigr),
\end{equation}
where $\lambda_{\min}$ is the smallest eigenvalue.

\begin{assumption}[Dissipation]\label{ass:dissip}
There are a rate $\mu_2\in\bigl(0,\bar\mu_2(y^*)\bigr)$, weights
$p_3,p_4>0$ and Young parameters $\nu_i,\nu_b>0$ such that \eqref{eq:window}
holds and $\bar\alpha>0$, with $\bar\alpha$ given by \eqref{eq:alphabar}.
\end{assumption}

Condition \eqref{eq:window} restricts $\mu_2$ and $p_4/p_3$;
$\mu_2<\bar\mu_2$ is sufficient for the interval to be non-empty, not
necessary (the exact threshold has the cube of
$\lambda_3(z_s^*)/|\lambda_4(z_s^*)|$ in the logarithm). The main restriction is
$\bar\alpha>0$: the dissipation created by the weights must be larger than
what the coupling $\Bb^{(2)}$ produces.

\subsection{The dissipation inequality}\label{ssec:dissip}

We now define the constants of the estimate. The terms in $Y_1(z_s^*)^2$ and
$Y_2(z_s^*)^2$ of \eqref{eq:youngi} are compensated by the outgoing terms of
$\Om_1^*$ at $x=z_s^*$, $\lambda e^{-2\mu_1z_s^*}(p_1Y_1^2+p_2Y_2^2)$.
Together with the coefficient of $Y_4(z_s^*)^2$ computed in \Cref{ssec:lyap},
this gives the boundary dissipation
\begin{equation}\label{eq:Dpar}
  \Dpar(Y)\triangleq d_1Y_1(z_s^*)^2+d_2Y_2(z_s^*)^2+d_4Y_4(z_s^*)^2 ,
\end{equation}
\begin{equation}\label{eq:d14}
\begin{aligned}
  d_1&=p_1\lambda e^{-2\mu_1z_s^*}
      -2(1+\nu_i^{-1})p_3\lambda_3(z_s^*)\,K_{int,1}^2,\\
  d_2&=p_2\lambda e^{-2\mu_1z_s^*}
      -2(1+\nu_i^{-1})p_3\lambda_3(z_s^*)\,K_{int,2}^2,\\
  d_4&=p_4e^{-2\mu_2(z_b-z_s^*)}|\lambda_4(z_s^*)|
      -p_3\lambda_3(z_s^*)(1+\nu_i)\,k^4 .
\end{aligned}
\end{equation}
The coefficients of $u^2$ and $\dot u^2$ are
\begin{equation}\label{eq:sigu}
\begin{gathered}
  \sigma_u=\nu_2^{-1}\frac{1-e^{-2\mu_1z_s^*}}{2\mu_1}p_2\gamma^2\epsin^2
  +\lambda p_2\epsin^2
  +(1+\nu_b^{-1})\,p_4|\lambda_4(z_b)|\bigl(\Db^{(bot)}\bigr)^2,\\
  \sigma_{{\rm vol},u}=\nu_2^{-1}\int_{\Om_2^*}\Db^{(2)\top}\Pb\,\Db^{(2)}\dd x,
  \qquad
  \sigma_{\dot u}=\frac{p_{2x}\epsin^2}{\lambda}.
\end{gathered}
\end{equation}
The three terms of $\sigma_u$ come from the source $\gamma\epsin u$ of
\eqref{eq:zone1}, from the top condition $Y_2(0)=\epsin u$ and from the bottom
condition of \eqref{eq:bc}; $\sigma_{{\rm vol},u}$ comes from
the term $\Db^{(2)}u$ of \eqref{eq:zone2}. The coefficients of $\eta^2$ and
$(\Go\Yo)^2$ are $\sigma_\eta=I_\eta/\nu_2$ and $\sigma_o=I_o/\nu_2$, with
\begin{equation}\label{eq:Ieo}
  I_\eta\triangleq\int_{\Om_2^*}\Cb^{(2)\top}_\eta\Pb\,\Cb^{(2)}_\eta\dd x,
  \qquad
  I_o\triangleq\int_{\Om_2^*}\Cb^{(2)\top}_{:,2}\Pb\,\Cb^{(2)}_{:,2}\dd x .
\end{equation}
Finally, when $d_1,d_2,d_4>0$, the Cauchy--Schwarz inequality gives
\begin{equation}\label{eq:Sigmab}
  (\Go\Yo)^2\le\Sigma_b\,\Dpar(Y),\qquad
  \Sigma_b\triangleq\frac{G_{o,1}^2}{d_1}+\frac{G_{o,2}^2}{d_2}
  +\frac{G_{o,4}^2}{d_4} .
\end{equation}

\begin{theorem}[Dissipation inequality for the distributed layer]\label{thm:ISS}
Let \Cref{ass:dissip} hold with the parameters $(p_3,p_4,\mu_2,\nu_i,\nu_b)$,
and let $\nu_1>0$ (the Young parameter of the Jordan coupling term
$-2p_1Y_1\partial_xY_2$, see \appiss{}). There are weights $p_1,p_2,p_{2x}>0$ and a rate
$\mu_1>0$ such that $d_1>0$ and $d_2>0$ ($d_4>0$ follows from the left
inequality of \eqref{eq:window}) and such that, for every
$\nu_2\in(0,\bar\alpha/3)$, every $u\in C^1$ and every
classical solution of the PDE part of \eqref{eq:plant} with $Y_1,Y_3,Y_4\in H^1$ and
$Y_2\in H^2$ at each time,
\begin{equation}\label{eq:ISS}
  \dot\Vc(Y)\le-\bigl(\bar\alpha-3\nu_2\bigr)\Vc(Y)-\Dpar(Y)
  +\sigma_\eta\,\eta^2+\sigma_o\,(\Go\Yo)^2
  +(\sigma_u+\sigma_{{\rm vol},u})u^2+\sigma_{\dot u}\dot u^2 .
\end{equation}
\end{theorem}

The proof is in \appiss{}. We keep \eqref{eq:ISS} in this form because
\Cref{sec:fwd} uses the terms in $\Vc$ and in $\Dpar$ separately. The term in
$\dot u$ appears because $Y_2(t,0)=\epsin u(t)$ and \eqref{eq:zone1} give
$\partial_xY_2(t,0)=-\epsin\dot u(t)/\lambda$, so the input must be $C^1$.
Increasing $\mu_2$ increases the diagonal part of $\Qb$ but decreases $d_4$, so the weights
must be chosen together (\Cref{cond:sg}).
\section{Forwarding design}
\label{sec:fwd}

With $u\equiv0$, \eqref{eq:plant} couples the layer of \Cref{sec:iss} with
the integrator $\dot\eta=\Go\Yo$, which has no damping, so
$\Vc+\rho\eta^2$ is not a Lyapunov functional. The position $\eta$ also acts on
the layer through $\Cb^{(2)}_\eta$, which is non-zero when
$\partial_x\epss^*\not\equiv0$. Following \cite{terrand2019adding}, we use a forwarding transformation to construct the control law.

\subsection{The forwarding coordinate and the Sylvester kernel}
\label{ssec:kernel}

Let $\Lk_i\in W^{1,\infty}(\Om_i^*;\mathbb{R}^{1\times2})$, $i=1,2$, to be
chosen, and
\begin{equation}\label{eq:transform}
  \Lk[Y]\triangleq\sum_{i=1,2}\int_{\Om_i^*}\Lk_i(x)Y^{(i)}(x)\dd x,\qquad
  w(t)\triangleq\eta(t)-\Lk[Y(t)].
\end{equation}
Let $(Y,\eta)$ be a solution of \eqref{eq:plant} of class $C^1$ in $t$ and
$x$. Differentiating $w$, using \eqref{eq:plant} and $\dot\eta=\Go\Yo$, and
integrating the term $\Lk_i\Ab^{(i)}\partial_xY^{(i)}$ by parts, we get
\begin{equation}\label{eq:dotw}
  \dot w=\gamma_v\,\Go\Yo+\mathcal{F}\,\eta
  +\sum_{i=1,2}\Bigl(\bigl[\Nb_iY^{(i)}\bigr]_{\partial\Om_i^*}
  -\int_{\Om_i^*}\bigl(\Nb_i'+\Lk_i\Bb^{(i)}\bigr)Y^{(i)}\dd x
  -\int_{\Om_i^*}\Lk_i\Db^{(i)}\dd x\,u\Bigr),
\end{equation}
where $\Nb_i\triangleq\Lk_i\Ab^{(i)}$, $\Nb_i'\triangleq\partial_x\Nb_i$, the
boundary terms are
\begin{equation}\label{eq:bracket}
  \bigl[\Nb_1Y^{(1)}\bigr]_{\partial\Om_1^*}\triangleq
  \Nb_1Y^{(1)}\big|_{x=0}^{x=z_s^{*-}},\qquad
  \bigl[\Nb_2Y^{(2)}\bigr]_{\partial\Om_2^*}\triangleq
  \Nb_2Y^{(2)}\big|_{x=z_s^{*+}}^{x=z_b},
\end{equation}
and
\begin{equation}\label{eq:Fgam}
  \mathcal{F}\triangleq-\int_{\Om_2^*}\Lk_2\,\Cb^{(2)}_\eta\dd x,
  \qquad
  \gamma_v\triangleq1-\int_{\Om_2^*}\Lk_2\,\Cb^{(2)}_{:,2}\dd x .
\end{equation}
We choose $\Lk$ so that $\dot w=c^*w+Gu$ for some scalars $c^*$ and $G$, in
three steps: for a scalar $c$, Step~1 replaces the integrals of $Y$ in
\eqref{eq:dotw} by $-c\,\Lk[Y]$, Step~2 cancels the terms in $Y$ of the
boundary terms \eqref{eq:bracket}, and Step~3 gives the equation that fixes
$c=c^*$. After Steps~1 and~2, $\Lk$ depends only on $c$, and we write
$\mathcal{F}(c)$ and $\gamma_v(c)$.

\paragraph*{Step 1: the integrals of $Y$.} For a scalar $c$, we ask that the
integrals of $Y$ in \eqref{eq:dotw} be equal to $-c\,\Lk[Y]$, that is,
$\Nb_i'+\Lk_i\Bb^{(i)}=c\,\Lk_i$ on $\Om_i^*$. Since $\Ab^{(i)}$ is invertible
($\det\Ab^{(1)}=\lambda^2$, $\det\Ab^{(2)}=\lambda_3\lambda_4\neq0$), this is the
linear ODE
\begin{equation}\label{eq:sylvester}
  \Nb_i'=\Nb_i\,\Xib_i(\cdot\,;c),\qquad
  \Xib_i\triangleq(\Ab^{(i)})^{-1}\bigl(c\,\mathbf{I}-\Bb^{(i)}\bigr),
\end{equation}
whose solutions are $\Nb_i(x)=\Nb_i(z_s^*)\Pib_i(x,z_s^*;c)$, with
$\partial_x\Pib_i=\Pib_i\Xib_i$ and $\Pib_i(z_s^*,z_s^*;c)=\mathbf{I}$. The
kernel is then $\Lk_i=\Nb_i(\Ab^{(i)})^{-1}$. On $\Om_1^*$ the coefficients are
constant and
\begin{equation}\label{eq:Pi1}
  \Pib_1(x,z_s^*;c)=
  \begin{bmatrix}
   e^{\frac{c}{\lambda}(x-z_s^*)} &
   \frac{c+\gamma}{\gamma\lambda}\bigl(e^{\frac{c}{\lambda}(x-z_s^*)}
   -e^{\frac{c+\gamma}{\lambda}(x-z_s^*)}\bigr)\\[4pt]
   0 & e^{\frac{c+\gamma}{\lambda}(x-z_s^*)}
  \end{bmatrix};
\end{equation}
on $\Om_2^*$, $\Pib_2$ has no closed form, and we compute it numerically for
the illustrative example of \Cref{sec:num}.

\paragraph*{Step 2: the boundary terms.} Let $N_{i,j}$ be the entries of
$\Nb_i$, $N_{i,j}^{\pm}\triangleq N_{i,j}(z_s^{*\pm})$ and
$K_{int,4}\triangleq-k^2$. By \eqref{eq:bracket}, \eqref{eq:bc} and
\eqref{eq:kratio}, with $Y_1,Y_2$ taken at $z_s^{*-}$ and $Y_4$ at
$z_s^{*+}$,
\begin{multline}\label{eq:bdry}
  \gamma_v\Go\Yo+\sum_{i=1,2}\bigl[\Nb_iY^{(i)}\bigr]_{\partial\Om_i^*}
  =\bigl(N_{2,1}+N_{2,2}\bigr)(z_b)\,Y_3(z_b)
  +\bigl(\gamma_vG_{o,1}+N_{1,1}^{-}-N_{2,1}^{+}K_{int,1}\bigr)Y_1\\
  +\bigl(\gamma_vG_{o,2}+N_{1,2}^{-}-N_{2,1}^{+}K_{int,2}\bigr)Y_2
  +\bigl(\gamma_vG_{o,4}-N_{2,1}^{+}K_{int,4}-N_{2,2}^{+}\bigr)Y_4\\
  +\bigl(-N_{1,2}(0)\,\epsin+N_{2,2}(z_b)\,\Db^{(bot)}\bigr)u .
\end{multline}
The coefficients of $Y_3(z_b)$, $Y_4$, $Y_1$ and $Y_2$ in \eqref{eq:bdry}
vanish if
\begin{equation}\label{eq:matching}
\begin{aligned}
  N_{2,1}(z_b)+N_{2,2}(z_b)&=0,\\
  \gamma_vG_{o,4}-N_{2,1}^{+}K_{int,4}-N_{2,2}^{+}&=0,\\
  \gamma_vG_{o,1}+N_{1,1}^{-}-N_{2,1}^{+}K_{int,1}&=0,\\
  \gamma_vG_{o,2}+N_{1,2}^{-}-N_{2,1}^{+}K_{int,2}&=0.
\end{aligned}
\end{equation}
Since $\Nb_2(z_b)=\Nb_2(z_s^*)\Pib_2(z_b,z_s^*;c)$, the first equation reads
$N_{2,1}^{+}v_1+N_{2,2}^{+}v_2=0$, with
\begin{equation}\label{eq:vvec}
  (v_1,v_2)^\top\triangleq\Pib_2(z_b,z_s^*;c)\,(1,1)^\top,
\end{equation}
and with the second equation it gives
\begin{equation}\label{eq:Delta}
  N_{2,1}^{+}\,\Delta_z(c)=-\gamma_vG_{o,4}v_2,\qquad
  \Delta_z(c)\triangleq v_1(c)-K_{int,4}\,v_2(c).
\end{equation}
If $\Delta_z(c)\neq0$ (non-resonance condition), \eqref{eq:Delta} and
\eqref{eq:matching} give $N_{2,1}^{+}$, $N_{2,2}^{+}$, $N_{1,1}^{-}$ and
$N_{1,2}^{-}$, all proportional to $\gamma_v$. Hence $\Lk_i=\gamma_v\Lk_i^0$,
where $\Lk_i^0$ is the kernel obtained with $\gamma_v=1$, and the definition of
$\gamma_v$ in \eqref{eq:Fgam} gives
\begin{equation}\label{eq:gammav}
  \gamma_v(c)=\frac{1}{1+m_2(c)},\qquad
  m_2(c)\triangleq\int_{\Om_2^*}\Lk_2^0\,\Cb^{(2)}_{:,2}\dd x,
\end{equation}
provided $1+m_2(c)\neq0$. Since
$\Delta_z(c)=(1,k^2)\,\Pib_2(z_b,z_s^*;c)\,(1,1)^\top$, non-resonance means
that $c$ is not an eigenvalue of the compression layer closed by its two
reflections ($Y_3=-k^2Y_4$ at $z_s^*$, $Y_4=Y_3$ at $z_b$). In
finite-dimensional forwarding, this corresponds to the condition that $c$ is not an eigenvalue
of the drift matrix of the stable subsystem.

\paragraph*{Step 3: the equation for $c$.} After Steps~1 and~2,
\eqref{eq:dotw} reduces to
\begin{equation}\label{eq:F}
\begin{gathered}
  \dot w=\mathcal{F}(c)\,\eta-c\,\Lk[Y]+G(c)\,u,\\
  G(c)\triangleq-N_{1,2}(0;c)\,\epsin+N_{2,2}(z_b;c)\,\Db^{(bot)}
  -\sum_{i=1,2}\int_{\Om_i^*}\Lk_i(\cdot\,;c)\,\Db^{(i)}\dd x .
\end{gathered}
\end{equation}
With $\eta=w+\Lk[Y]$, this is
$\dot w=\mathcal{F}(c)\,w+\bigl(\mathcal{F}(c)-c\bigr)\Lk[Y]+G(c)\,u$, so the
integral of $Y$ disappears if and only if
\begin{equation}\label{eq:cstar}
  c^*=\mathcal{F}(c^*).
\end{equation}
By Step~2, the kernel, hence $\mathcal{F}(c)$, is defined for $c$ in the
open set
\begin{equation*}
  \mathcal{C}\triangleq\bigl\{c\in\mathbb{R}:\ \Delta_z(c)\neq0,\
  1+m_2(c)\neq0\bigr\},
\end{equation*}
and $\mathcal{F}$ is continuous on $\mathcal{C}$, since
$\Pib_i(\cdot,z_s^*;c)$ depends continuously on $c$. By the intermediate value
theorem, if $[c_-,c_+]\subset\mathcal{C}$ and
$\bigl(\mathcal{F}(c_-)-c_-\bigr)\bigl(\mathcal{F}(c_+)-c_+\bigr)<0$, then
\eqref{eq:cstar} has a root $c^*\in(c_-,c_+)$. 

\begin{lemma}[Kernel construction]\label{lem:kernel}
Assume the equilibrium is strictly subcritical on $\Om_2^*$
($\lambda_4<0$ on $\overline{\Om_2^*}$), and let $c^*$ satisfy
\eqref{eq:cstar} with $\Delta_z(c^*)\neq0$ and $1+m_2(c^*)\neq0$. Then the
kernels $\Lk_i$ of Steps~1 and~2 with $c=c^*$ are in
$W^{1,\infty}(\Om_i^*;\mathbb{R}^{1\times2})$ and, along every solution of
\eqref{eq:plant} of class $C^1$,
\begin{equation}\label{eq:dotw_closed}
  \dot w=c^*w+Gu,\qquad G\triangleq G(c^*).
\end{equation}
\end{lemma}

\begin{proof}
Since $\lambda>0$, $\lambda_3>0$ and $\lambda_4<0$ on the closed intervals,
$(\Ab^{(i)})^{-1}$ is bounded, and
\begin{equation*}
  \|\Xib_i\|_{L^\infty(\Om_i^*)}\le X_i\triangleq
  \bigl\|(\Ab^{(i)})^{-1}\bigr\|_{L^\infty}\bigl(|c^*|
  +\|\Bb^{(i)}\|_{L^\infty}\bigr).
\end{equation*}
Gr\"onwall's lemma, applied to $\partial_x\Pib_i=\Pib_i\Xib_i$ and to
$\partial_x\Pib_i^{-1}=-\Xib_i\Pib_i^{-1}$, gives
\begin{equation*}
  \max\bigl\{\|\Pib_i(x,z_s^*)\|,\|\Pib_i(x,z_s^*)^{-1}\|\bigr\}
  \le e^{X_i|\Om_i^*|},\qquad x\in\overline{\Om_i^*}.
\end{equation*}
Hence $\Nb_i=\Nb_i(z_s^*)\Pib_i\in W^{1,\infty}$ and
$\Lk_i=\Nb_i(\Ab^{(i)})^{-1}\in W^{1,\infty}$. The traces $\Nb_i(z_s^{*\pm})$
and $\gamma_v$ are given by Step~2, and \eqref{eq:dotw_closed} follows from
Steps~1 to~3 with $c=c^*$.
\end{proof}

\subsection{ISS in the variables \texorpdfstring{$(Y,w,u,\dot{u})$}{(Y,w,u,du/dt)}}
\label{ssec:issw}

By \Cref{lem:kernel}, $\dot w=c^*w+Gu$, and $Y$ does not appear in this
equation. We now write the inequality of \Cref{thm:ISS} with $w$ instead of
$\eta$, using $\eta=w+\Lk[Y]$. Let
\begin{equation}\label{eq:constants}
  J_\Lk\triangleq\sum_{i=1,2}\int_{\Om_i^*}\Lk_i\Wb_i^{-1}\Lk_i^\top\dd x,
  \qquad
  a_\eta\triangleq-c^*+\sqrt{J_\Lk I_\eta},
  \qquad a_o\triangleq\Sigma_bI_o .
\end{equation}
By the Cauchy--Schwarz inequality, and by \eqref{eq:Fgam} and
\eqref{eq:cstar},
\begin{equation*}
  \Lk[Y]^2\le J_\Lk\sum_{i=1,2}\int_{\Om_i^*}Y^{(i)\top}\Wb_iY^{(i)}\dd x
  \le J_\Lk\Vc(Y),\qquad
  \int_{\Om_2^*}\Lk_2\Cb^{(2)}_\eta\dd x=-\mathcal{F}(c^*)=-c^* ,
\end{equation*}
where the second inequality holds because $\Vc$ contains, in addition, the
nonnegative term in $(\partial_xY_2)^2$.

\begin{proposition}[ISS in the variables $(Y,w,u,\dot{u})$]
\label{prop:ISSw}
Let \Cref{ass:dissip} hold, let $c^*$ be a root of \eqref{eq:cstar} with
$\Delta_z(c^*)\neq0$ and $1+m_2(c^*)\neq0$, let $\Lk$ be the kernel of
\Cref{lem:kernel}, and let $\nu_2\in(0,\bar\alpha/3)$, $\vartheta>0$ and $\delta\in(0,1)$. With the
weights of \Cref{thm:ISS}, along every solution of \eqref{eq:plant} with the
regularity of \Cref{thm:ISS},
\begin{equation}\label{eq:ISSw}
  \dot\Vc(Y)\le-\alpha_w\Vc(Y)-\delta\,\Dpar(Y)
  +\sigma_w\,w^2+(\sigma_u+\sigma_{{\rm vol},u})u^2+\sigma_{\dot u}\dot u^2,
\end{equation}
where
\begin{equation}\label{eq:alphaw}
  \alpha_w\triangleq\bar\alpha-\nu_2-a_\eta-\frac{a_o}{1-\delta}-\vartheta,
  \qquad
  \sigma_w\triangleq\frac{I_\eta}{\vartheta}.
\end{equation}
\end{proposition}

\begin{proof}
We follow the proof of \Cref{thm:ISS} in \appiss{}, but apply Young's
inequality only to the terms in $u$. This gives
\begin{equation}\label{eq:Vraw}
  \dot\Vc\le-(\bar\alpha-\nu_2)\Vc-\Dpar
  +2\eta\,\Phi(Y)+2\Psi(Y)\,\Go\Yo
  +(\sigma_u+\sigma_{{\rm vol},u})u^2+\sigma_{\dot u}\dot u^2,
\end{equation}
with
$\Phi(Y)\triangleq\int_{\Om_2^*}Y^{(2)\top}\Pb\,\Cb^{(2)}_\eta\dd x$ and
$\Psi(Y)\triangleq\int_{\Om_2^*}Y^{(2)\top}\Pb\,\Cb^{(2)}_{:,2}\dd x$. By the
Cauchy--Schwarz inequality, $\Phi(Y)^2\le I_\eta\Vc(Y)$ and
$\Psi(Y)^2\le I_o\Vc(Y)$.

\emph{Term in $\eta$.} With $\eta=w+\Lk[Y]$,
$2\eta\Phi=2w\Phi+2\Lk[Y]\Phi$. Young's inequality gives
$2w\Phi\le\vartheta\Vc+\vartheta^{-1}I_\eta w^2$. For the second term, let
$a$ and $b$ be the Riesz representatives of $\Lk[\cdot]$ and $\Phi$ for the
inner product of $\Vc$; $b=(0,\Cb^{(2)}_\eta)$ is supported on $\Om_2^*$. Then $2\Lk[Y]\Phi(Y)=2\langle a,Y\rangle\langle
b,Y\rangle\le\bigl(\langle a,b\rangle+\|a\|\|b\|\bigr)\Vc(Y)$, the largest
eigenvalue of the rank-two operator $ab^*+ba^*$, with
$\langle a,b\rangle=\Lk[b]=\int_{\Om_2^*}\Lk_2\Cb^{(2)}_\eta\dd x=-c^*$,
$\|a\|^2\le J_\Lk$ and $\|b\|^2=I_\eta$. Hence
\begin{equation}\label{eq:rank2}
  2\eta\,\Phi(Y)\le\bigl(a_\eta+\vartheta\bigr)\Vc(Y)
  +\frac{I_\eta}{\vartheta}\,w^2 .
\end{equation}

\emph{Term in $\Go\Yo$.} By Young's inequality with parameter
$\varsigma=\Sigma_b/(1-\delta)$ and \eqref{eq:Sigmab},
\begin{equation}\label{eq:tracebound}
  2\Psi(Y)\,\Go\Yo\le\varsigma I_o\Vc+\frac{\Sigma_b}{\varsigma}\Dpar
  =\frac{a_o}{1-\delta}\Vc+(1-\delta)\Dpar .
\end{equation}
Adding \eqref{eq:rank2} and \eqref{eq:tracebound} to \eqref{eq:Vraw} gives
\eqref{eq:ISSw}.
\end{proof}

The closed loop is thus a cascade: $w$ follows a scalar
equation driven by $u$ only, and by \eqref{eq:ISSw} the layer is
input-to-state stable with respect to $w$ and $u$. With the feedback
$u=\mathfrak{K}w$ of \Cref{ssec:feedback}, $\dot u=\mathfrak{K}\dot w
=-\kappa\mathfrak{K}w$, and the terms in $u^2$, $\dot u^2$ and $w^2$ of
\eqref{eq:ISSw} become $b_w w^2$, with
\begin{equation}\label{eq:bw}
  b_w\triangleq\sigma_w+(\sigma_u+\sigma_{{\rm vol},u}+\sigma_{\dot u}\kappa^2)
  \mathfrak{K}^2 .
\end{equation}
The cascade is exponentially stable if $\alpha_w>0$. This holds under the
following condition.

\begin{assumption}[Small gain]\label{cond:sg}
There are weights $p_1,p_2,p_{2x},p_3,p_4>0$, rates $\mu_1,\mu_2>0$ and Young
parameters $\nu_i,\nu_b>0$ such that \Cref{ass:dissip} and the conclusion of
\Cref{thm:ISS} hold, and
\begin{equation}\label{eq:smallgain}
  a_\eta+a_o<\bar\alpha,
\end{equation}
where $a_\eta$ and $a_o$ are given by \eqref{eq:constants} and $\bar\alpha$
by \eqref{eq:alphabar}, all computed with these weights.
\end{assumption}

Indeed, let $\alpha_{sg}\triangleq\bar\alpha-a_\eta-a_o>0$,
$\nu_2=\vartheta=\min(\alpha_{sg},\bar\alpha)/4$, so that
$\nu_2<\bar\alpha/3$, and $\delta\in(0,1)$ such that
$\delta a_o\le(1-\delta)\alpha_{sg}/4$. Since
$a_o/(1-\delta)=a_o+\delta a_o/(1-\delta)$, \eqref{eq:alphaw} gives
\begin{equation}\label{eq:alphawpos}
  \alpha_w=\alpha_{sg}-\nu_2-\vartheta-\frac{\delta a_o}{1-\delta}
  \ge\frac{\alpha_{sg}}{4}>0 .
\end{equation}

\Cref{cond:sg} is a sufficient condition, expressed through the weights of the
Lyapunov functional \eqref{eq:V}. Items~(i) and~(ii) of \Cref{thm:fwd_GES} below do not need it;
item~(iii) does.

\subsection{The feedback and the closed loop}
\label{ssec:feedback}

Assume $G\neq0$ and, for a chosen rate $\kappa>0$, apply
\begin{equation}\label{eq:feedback}
  u(t)=\mathfrak{K}\,w(t),\qquad
  \mathfrak{K}\triangleq-\frac{c^*+\kappa}{G},
\end{equation}
which turns \eqref{eq:dotw_closed} into $\dot w=-\kappa w$. In the original
variables the law reads
\begin{equation}\label{eq:law}
  v_m(t)=v_m^*+\mathfrak{K}\Bigl(\bigl(z_s(t)-z_s^*\bigr)
  -\int_0^{z_s^*}\!\Lk_1Y^{(1)}-\int_{z_s^*}^{z_b}\!\Lk_2Y^{(2)}\Bigr).
\end{equation}
This is a full-state feedback: it requires $z_s(t)$ and the profiles
$\epss(t,\cdot)$, $f_s(t,\cdot)$ on both zones, mapped onto the reference
configuration by \eqref{eq:diffeo}.

\begin{assumption}[Design assumptions]\label{ass:fwd}
The equilibrium $y^*$ is admissible and strictly subcritical on $\Om_2^*$;
\Cref{ass:dissip} holds; and $c^*$ is a root of \eqref{eq:cstar} with
$\Delta_z(c^*)\neq0$, $1+m_2(c^*)\neq0$ and $G(c^*)\neq0$.
\end{assumption}

Under \Cref{ass:fwd}, the feedback \eqref{eq:feedback} is well defined. The
stability of the layer also requires \Cref{cond:sg}.

We now write the closed loop as an evolution equation. Since
$\Lk_i\in W^{1,\infty}(\Om_i^*)$,
\begin{equation*}
  w(Y,\eta)\triangleq\eta-\Lk[Y]
\end{equation*}
is a bounded linear functional on $\Hs\times\mathbb{R}$, and
$w(t)=w(Y(t),\eta(t))$ along a solution. With the feedback
\eqref{eq:feedback}, the inlet condition $Y_2(t,0)=\epsin u(t)$ of
\eqref{eq:bc} is a condition on the state, since $Y_2\in H^1(\Om_1^*)$ is
continuous. The state space is the closed hyperplane
\begin{equation}\label{eq:Hcl}
  \Hs_{cl}\triangleq\bigl\{(Y,\eta)\in\Hs\times\mathbb{R}:\,
  Y_2(0)=\epsin\,\mathfrak{K}\,w(Y,\eta)\bigr\},
\end{equation}
and the closed-loop operator $\mathcal{A}_{cl}$ is the right-hand side of
\eqref{eq:plant} with $u=\mathfrak{K}w$, on the domain
\begin{equation}\label{eq:DAcl}
\begin{aligned}
  D(\mathcal{A}_{cl})=\bigl\{(Y,\eta)\in\Hs_{cl}:\,&
  Y_1,Y_3,Y_4\in H^1,\;Y_2\in H^2,\;
  \text{\eqref{eq:bc} with } u=\mathfrak{K}w \text{ and \eqref{eq:Kint} hold},\\
  &\lambda\,\partial_xY_2(0)=\kappa\,\epsin\,\mathfrak{K}\,w(Y,\eta)
  \bigr\}.
\end{aligned}
\end{equation}
The last condition is the compatibility condition of order one at $x=0$. By
the second line of \eqref{eq:zone1} at $x=0$ and $Y_2(t,0)=\epsin u(t)$, we
have $\lambda\partial_xY_2(t,0)=-\epsin\dot u(t)$, and
$\dot u=\mathfrak{K}\dot w=-\kappa\mathfrak{K}w$.

\begin{lemma}[Forwarding identity]\label{lem:fwdid}
Let \Cref{ass:fwd} hold. For every $(Y,\eta)\in D(\mathcal{A}_{cl})$,
\begin{equation}\label{eq:fwdid}
  w\bigl(\mathcal{A}_{cl}(Y,\eta)\bigr)=-\kappa\,w(Y,\eta)
  \qquad\text{and}\qquad
  \mathcal{A}_{cl}(Y,\eta)\in\Hs_{cl} .
\end{equation}
\end{lemma}

\begin{proof}
The computation of \eqref{eq:dotw} and \Cref{lem:kernel} only uses the
regularity of $D(\mathcal{A}_{cl})$, \eqref{eq:bc} and \eqref{eq:Kint}. It
gives $w(\mathcal{A}_{cl}(Y,\eta))=c^*w(Y,\eta)+Gu$ with
$u=\mathfrak{K}w(Y,\eta)$, and $c^*+G\mathfrak{K}=-\kappa$ by
\eqref{eq:feedback}. Let
$(\tilde Y,\tilde\eta)\triangleq\mathcal{A}_{cl}(Y,\eta)$. By the second line
of \eqref{eq:zone1} at $x=0$, $Y_2(0)=\epsin u$, the last condition of
\eqref{eq:DAcl} and the first identity,
\begin{equation*}
  \tilde Y_2(0)=-\lambda\partial_xY_2(0)-\gamma\epsin u+\gamma\epsin u
  =-\kappa\,\epsin\mathfrak{K}\,w(Y,\eta)
  =\epsin\mathfrak{K}\,w(\tilde Y,\tilde\eta),
\end{equation*}
so $(\tilde Y,\tilde\eta)\in\Hs_{cl}$.
\end{proof}

\begin{theorem}[Closed-loop behaviour]\label{thm:fwd_GES}
Let \Cref{ass:fwd} hold and $(Y_0,\eta_0)\in D(\mathcal{A}_{cl})$.
\begin{enumerate}[label=(\roman*),leftmargin=2.2em,itemsep=2pt]
\item \emph{(Well-posedness.)} $\mathcal{A}_{cl}$ generates a
$C_0$-semigroup on $\Hs_{cl}$, and the closed loop \eqref{eq:plant},
\eqref{eq:feedback} has a unique strong solution
$(Y,\eta)\in C^1(\Rp;\Hs_{cl})$ with $(Y,\eta)(0)=(Y_0,\eta_0)$.

\item \emph{(Decay of $w$.)} For every $t\ge0$,
\begin{equation}\label{eq:wexact}
  w(t)=e^{-\kappa t}w(0),\qquad u(t)=\mathfrak{K}e^{-\kappa t}w(0),
\end{equation}
and
\begin{equation}\label{eq:etasplit}
  |\eta(t)|\le e^{-\kappa t}|w(0)|+\sqrt{J_\Lk\,\Vc(Y(t))} .
\end{equation}
This also holds for the mild solutions with $(Y_0,\eta_0)\in\Hs_{cl}$.

\item \emph{(Exponential stability.)} If, in addition, \Cref{cond:sg}
holds, there are $C,\beta>0$ such that, for every
$(Y_0,\eta_0)\in D(\mathcal{A}_{cl})$ and every $t\ge0$,
\begin{equation}\label{eq:GES}
  \|Y(t,\cdot)\|_\Hs^2+\eta(t)^2\le Ce^{-\beta t}
  \bigl(\|Y_0\|_\Hs^2+\eta_0^2\bigr).
\end{equation}
Since $D(\mathcal{A}_{cl})$ is dense in $\Hs_{cl}$, \eqref{eq:GES} also
holds for the mild solutions with $(Y_0,\eta_0)\in\Hs_{cl}$.
\end{enumerate}
\end{theorem}

The control $u$ is $C^\infty$ in time, and the law \eqref{eq:law} uses
neither a trace nor a spatial derivative of the state.

\begin{proof}
Items (i) and (iii) are based on the derivative along the closed loop of
\begin{equation}\label{eq:Ucal}
  \Uc(Y,\eta)\triangleq\Vc(Y)+\rho\,w^2,\qquad w=\eta-\Lk[Y],\quad\rho>0:
\end{equation}
for item (i), this derivative must be bounded by a multiple of $\Uc$, and for item (iii), it must be negative.

\paragraph*{Item (i).} By \Cref{lem:fwdid},
$\mathcal{A}_{cl}$ is an operator on the Hilbert space $\Hs_{cl}$, which we
equip with the inner product of $\Uc$, equivalent to the one of
$\Hs\times\mathbb{R}$ by \eqref{eq:normequiv}, $\Lk[Y]^2\le J_\Lk\Vc(Y)$ and
$\eta^2\le2w^2+2J_\Lk\Vc(Y)$. In the following we check the three hypotheses
of the Lumer--Phillips theorem \cite[Thm.~1.4.3]{pazy2012semigroups}.

\emph{Density.} $D(\mathcal{A}_{cl})$ is dense in $\Hs_{cl}$. Given
$(Y,\eta)\in\Hs_{cl}$, every modification $Z$ of $Y$ is accompanied by adding
$\Lk[Z]$ to $\eta$, which is small since $\Lk$ is bounded on $\Hs$, and leaves
$w$, hence the constraint \eqref{eq:Hcl}, unchanged. We first approximate
$Y_1$ in $L^2$ and $Y_2$ in $H^1$ by $H^1$ and $H^2$ functions with $Y_2(0)$
unchanged; the value of $\partial_xY_2(0)$ required by \eqref{eq:DAcl} is
then obtained by adding $\phi_r(x)\triangleq a\,x(1-x/r)^2$ on $[0,r]$ (and
$0$ elsewhere) to $Y_2$, which leaves $Y_2(0)$ unchanged, sets
$\partial_xY_2(0)$ to any prescribed value, and has norm $O(\sqrt r)$ in
$H^1$. We then approximate $Y_3$, $Y_4$ in $L^2$ by $H^1$ functions whose
traces satisfy \eqref{eq:Kint} and \eqref{eq:bc} with the traces of the new
$Y_1$, $Y_2$, which is possible since these conditions only concern values
at the endpoints.

\emph{Dissipativity.} Let $x=(Y,\eta)\in D(\mathcal{A}_{cl})$ and $w=w(x)$.
By \eqref{eq:Ucal}, $2\langle\mathcal{A}_{cl}x,x\rangle_\Uc$ is the
derivative of $\Vc(Y)+\rho w^2$ in the direction $\mathcal{A}_{cl}x$.  The proof of \eqref{eq:ISSw} in \appiss{} is a pointwise computation that
only uses the equations and boundary conditions satisfied by $x\in D(\mathcal{A}_{cl})$,
with $\partial_tY$ replaced by the $Y$-component of $\mathcal{A}_{cl}x$; so
\eqref{eq:ISSw} holds as an inequality between quadratic forms on
$D(\mathcal{A}_{cl})$, with $u=\mathfrak{K}w$ and $\dot u=-\kappa u$ (last
condition of \eqref{eq:DAcl}). since $w(\mathcal{A}_{cl}x)=-\kappa w$ by
\eqref{eq:fwdid}, we have, with $b_w$ given by \eqref{eq:bw},
\begin{equation}\label{eq:quasidissip}
  2\langle\mathcal{A}_{cl}x,x\rangle_\Uc
  \le-\alpha_w\Vc(Y)-\delta\Dpar(Y)+b_w w^2-2\rho\kappa w^2
  \le2\omega\,\Uc(x),
\end{equation}
with $2\omega\triangleq\max\{0,-\alpha_w,b_w/\rho\}$. Hence $\mathcal{A}_{cl}-\omega\mathbf{I}$ is dissipative.

\emph{Range.} Let $s>\omega$ and $(f,g)\in\Hs_{cl}$, with
$f=(f_1,f_2,f_3,f_4)$ and $f^{(2)}=(f_3,f_4)^\top$. If $x=(Y,\eta)\in
D(\mathcal{A}_{cl})$ solves $sx-\mathcal{A}_{cl}x=(f,g)$, then, by
\eqref{eq:fwdid},
\begin{equation}\label{eq:uresolvent}
  (s+\kappa)\,w(x)=w(f,g)=g-\Lk[f],\qquad
  u=\mathfrak{K}\,w(x)=\frac{\mathfrak{K}\bigl(g-\Lk[f]\bigr)}{s+\kappa}.
\end{equation}

With this $u$ and $\chi\triangleq\Go\Yo$, the equation
$sx-\mathcal{A}_{cl}x=(f,g)$ reads, by \eqref{eq:plant},
\begin{subequations}\label{eq:res}
\begin{align}
  &\lambda\partial_xY_2=-(s+\gamma)Y_2+\gamma\epsin u+f_2,
  \qquad Y_2(0)=\epsin u,\label{eq:res1}\\
  &\lambda\partial_xY_1=-sY_1-\partial_xY_2+f_1,
  \qquad Y_1(0)=0,\label{eq:res2}\\
  &\Ab^{(2)}\partial_xY^{(2)}=(\Bb^{(2)}-s\mathbf{I})Y^{(2)}+\varphi,
  \qquad \varphi\triangleq\Cb^{(2)}_\eta\eta+\Cb^{(2)}_{:,2}\chi+\Db^{(2)}u
  +f^{(2)},\label{eq:res3}\\
  &Y_3(z_s^*)=\Kint\Yo,\quad Y_4(z_b)=Y_3(z_b)+\Db^{(bot)}u,
  \quad\chi=\Go\Yo,\quad s\eta-\chi=g.\label{eq:res4}
\end{align}
\end{subequations}
On $\Om_1^*$, \eqref{eq:res1} and \eqref{eq:res2} give
\begin{equation}\label{eq:resOm1}
\begin{aligned}
  Y_2(x)&=e^{-(s+\gamma)x/\lambda}\epsin u
  +\frac{1}{\lambda}\int_0^xe^{-(s+\gamma)(x-\xi)/\lambda}
  \bigl(\gamma\epsin u+f_2(\xi)\bigr)\dd\xi,\\
  Y_1(x)&=\frac{1}{\lambda}\int_0^xe^{-s(x-\xi)/\lambda}
  \bigl(f_1(\xi)-\partial_xY_2(\xi)\bigr)\dd\xi,
\end{aligned}
\end{equation}
so that $Y_2\in H^2$ and $Y_1\in H^1$, since $f_2\in H^1$ and $f_1\in L^2$.
With $Y_3^0\triangleq(\lambda^2Y_1+2\lambda Y_2)(z_s^*)/\lambda_3^2$ and
$\chi_0\triangleq G_{o,1}Y_1(z_s^*)+G_{o,2}Y_2(z_s^*)$, known from
\eqref{eq:resOm1}, \eqref{eq:Kint} and \eqref{eq:kinematics} turn the
conditions at $z_s^*$ in \eqref{eq:res4} into
\begin{equation}\label{eq:res5}
  Y_3(z_s^*)=Y_3^0-k^2Y_4(z_s^*),\qquad
  \chi=\chi_0+G_{o,4}Y_4(z_s^*),\qquad \eta=(g+\chi)/s .
\end{equation}
On $\Om_2^*$, let
\begin{equation*}
\begin{gathered}
  \Lambda_-\triangleq\min_{\overline{\Om_2^*}}\min(\lambda_3,|\lambda_4|),
  \quad
  \Lambda_+\triangleq\max_{\overline{\Om_2^*}}\max(\lambda_3,|\lambda_4|),
  \quad
  \bar b\triangleq\max_{\overline{\Om_2^*}}\max\bigl(|\Bb^{(2)}_{12}|,
  |\Bb^{(2)}_{21}|\bigr),\\
  \bar\ell\triangleq\exp\int_{z_s^*}^{z_b}\Bigl(\frac{|\Bb^{(2)}_{11}|}{\lambda_3}
  +\frac{|\Bb^{(2)}_{22}|}{|\lambda_4|}\Bigr)\dd\xi,\quad
  \bar c\triangleq|G_{o,4}|\max_{\overline{\Om_2^*},\,i}
  \bigl(|\Cb^{(2)}_{i,1}|+|\Cb^{(2)}_{i,2}|\bigr),\\
  \mathfrak{m}\triangleq\frac{\bar\ell\,\Lambda_+}{\Lambda_-},\qquad
  e_s\triangleq\bar\ell\,e^{-s(z_b-z_s^*)/\Lambda_+},
\end{gathered}
\end{equation*}
and
\begin{equation*}
  h_3(x)=\exp\int_{z_s^*}^x\frac{\Bb^{(2)}_{11}-s}{\lambda_3}\dd\xi,\qquad
  h_4(x)=\exp\int_{z_s^*}^x\frac{\Bb^{(2)}_{22}-s}{\lambda_4}\dd\xi .
\end{equation*}
For $z_s^*\le\xi\le x\le z_b$,
\begin{equation*}
  0<\frac{h_3(x)}{h_3(\xi)}\le\bar\ell\,e^{-s(x-\xi)/\Lambda_+},\qquad
  0<\frac{h_4(\xi)}{h_4(x)}\le\bar\ell\,e^{-s(x-\xi)/\Lambda_+},
\end{equation*}
hence, for every $x\in\overline{\Om_2^*}$,
\begin{equation}\label{eq:kernel}
  \int_{z_s^*}^x\frac{h_3(x)}{h_3(\xi)}\frac{\dd\xi}{\lambda_3}\le\frac{\mathfrak{m}}{s},
  \qquad
  \int_x^{z_b}\frac{h_4(x)}{h_4(\xi)}\frac{\dd\xi}{|\lambda_4|}\le\frac{\mathfrak{m}}{s},
  \qquad
  h_3(z_b)\le e_s,\qquad\frac{h_4(z_s^*)}{h_4(z_b)}\le e_s .
\end{equation}
\emph{Step 1: given inflow values.} For $q_3,q_4\in\mathbb{R}$, the solutions
of \eqref{eq:res3} with $Y_3(z_s^*)=q_3$ and $Y_4(z_b)=q_4$ are the fixed
points of
\begin{equation}\label{eq:fp}
\begin{aligned}
  Y_3(x)&=h_3(x)\,q_3+\int_{z_s^*}^x\frac{h_3(x)}{h_3(\xi)}
  \frac{(\Bb^{(2)}_{12}Y_4+\varphi_3)(\xi)}{\lambda_3(\xi)}\dd\xi,\\
  Y_4(x)&=\frac{h_4(x)}{h_4(z_b)}\,q_4-\int_x^{z_b}\frac{h_4(x)}{h_4(\xi)}
  \frac{(\Bb^{(2)}_{21}Y_3+\varphi_4)(\xi)}{\lambda_4(\xi)}\dd\xi .
\end{aligned}
\end{equation}
By \eqref{eq:kernel}, the right-hand side of \eqref{eq:fp} is a contraction
of $C^0(\overline{\Om_2^*})^2$, for the norm
$\|Y^{(2)}\|_\infty\triangleq\max(\|Y_3\|_\infty,\|Y_4\|_\infty)$, with constant
$\bar b\mathfrak{m}/s\le\frac{1}{2}$ as soon as $s\ge2\bar b\mathfrak{m}$; its fixed point is in $H^1$
by \eqref{eq:res3}. If moreover $\varphi$ is continuous, \eqref{eq:fp} and
\eqref{eq:kernel} give
\begin{gather}
  \|Y^{(2)}\|_\infty\le\bar\ell\,(|q_3|+|q_4|)+\frac{\mathfrak{m}}{s}\bigl(\bar b\,\|Y^{(2)}\|_\infty
  +\|\varphi\|_\infty\bigr)
  \le2\bar\ell\,(|q_3|+|q_4|)+\frac{2\mathfrak{m}}{s}\|\varphi\|_\infty,\label{eq:Nbound}\\
  |Y_4(z_s^*)|\le e_s|q_4|+\frac{\mathfrak{m}}{s}\bigl(\bar b\,\|Y^{(2)}\|_\infty
  +\|\varphi\|_\infty\bigr),\label{eq:outbound}\\
  |Y_3(z_b)|\le e_s|q_3|+\frac{\mathfrak{m}}{s}\bigl(\bar b\,\|Y^{(2)}\|_\infty
  +\|\varphi\|_\infty\bigr).\nonumber
\end{gather}
\emph{Step 2: the boundary values.} By \eqref{eq:res5} and the bottom
condition of \eqref{eq:res4}, $x$ solves \eqref{eq:res} if and only if
$(\varpi_1,\varpi_2)=(Y_4(z_s^*),Y_4(z_b))$ is a fixed point of the affine map
\begin{equation*}
  \Psi_s:\ (\varpi_1,\varpi_2)\longmapsto\bigl(Y_4(z_s^*),\,
  Y_3(z_b)+\Db^{(bot)}u\bigr),
\end{equation*}
where $Y^{(2)}$ solves \eqref{eq:fp} with $q_3=Y_3^0-k^2\varpi_1$,
$q_4=\varpi_2$, and $\varphi$ computed with $\chi=\chi_0+G_{o,4}\varpi_1$ and
$\eta=(g+\chi)/s$. The linear part of $\Psi_s$ is obtained with
$Y_3^0=\chi_0=g=u=0$ and $f^{(2)}=0$: then $q_3=-k^2\varpi_1$, $q_4=\varpi_2$
and, for $s\ge1$, $\|\varphi\|_\infty\le\bar c\,|\varpi_1|$ (the condition $s\ge1$ is used for the term
$\Cb^{(2)}_\eta\eta$ with $\eta=\chi/s$). With $k<1$ (since
$v_s^*>0$), $2\bar b\mathfrak{m}/s\le1$ and $|\varpi|\triangleq\max(|\varpi_1|,
|\varpi_2|)$, \eqref{eq:Nbound} gives
$\bar b\,\|Y^{(2)}\|_\infty+\|\varphi\|_\infty\le4\bar b\bar\ell\,|\varpi|+2\bar c\,|\varpi|$,
and \eqref{eq:outbound} gives, for the linear part,
\begin{equation*}
  |Y_4(z_s^*)|,\ |Y_3(z_b)|\ \le\ \Bigl(e_s
  +\frac{4\mathfrak{m}(\bar b\bar\ell+\bar c)}{s}\Bigr)|\varpi| .
\end{equation*}
Hence $\Psi_s$ is a contraction of constant at most $\frac{1}{2}$ on
$\mathbb{R}^2$ as soon as
\begin{equation}\label{eq:sstar}
  s\ge s_*\triangleq\max\Bigl\{1,\ 2\bar b\mathfrak{m},\ 16\mathfrak{m}(\bar b\bar\ell+\bar c),\
  \frac{\Lambda_+\ln(4\bar\ell)}{z_b-z_s^*}\Bigr\},
\end{equation}
since then $e_s\le\frac{1}{4}$ and $4\mathfrak{m}(\bar b\bar\ell+\bar c)/s\le\frac{1}{4}$.
For $s\ge s_*$, \eqref{eq:res}, with the input $u$ of
\eqref{eq:uresolvent}, therefore has a unique solution $x=(Y,\eta)$, with
$Y_1,Y_3,Y_4\in H^1$ and $Y_2\in H^2$. It lies in $D(\mathcal{A}_{cl})$.
First, $u=\mathfrak{K}w(x)$: the computation of \eqref{eq:dotw} and
\Cref{lem:kernel}, applied to \eqref{eq:res} with the input $u$, gives
$(s-c^*)\,w(x)-Gu=w(f,g)$; with \eqref{eq:uresolvent} and
$G\mathfrak{K}=-(c^*+\kappa)$,
\begin{equation*}
  (s-c^*)\bigl(\mathfrak{K}w(x)-u\bigr)
  =\mathfrak{K}w(f,g)+G\mathfrak{K}u-(s-c^*)u
  =(s+\kappa)u-(c^*+\kappa)u-(s-c^*)u=0,
\end{equation*}
so $u=\mathfrak{K}w(x)$ if $s\neq c^*$, and $Y_2(0)=\epsin\mathfrak{K}w(x)$.
Second, since $(f,g)\in\Hs_{cl}$, $f_2(0)=\epsin\mathfrak{K}w(f,g)
=\epsin(s+\kappa)u$, and \eqref{eq:res1} at $x=0$ gives
\begin{equation*}
  \lambda\partial_xY_2(0)=-s\epsin u+f_2(0)=\kappa\epsin u
  =\kappa\epsin\mathfrak{K}\,w(x),
\end{equation*}
which is the last condition of \eqref{eq:DAcl}. So, for
$s>\max(\omega,c^*)$ with $s\ge s_*$, $s-\mathcal{A}_{cl}$ is onto
$\Hs_{cl}$.

\emph{Conclusion.} By the Lumer--Phillips theorem, $\mathcal{A}_{cl}$
generates a $C_0$-semigroup on $\Hs_{cl}$, and for
$(Y_0,\eta_0)\in D(\mathcal{A}_{cl})$ the orbit $x(t)$ is the unique strong
solution. Along it, by \eqref{eq:fwdid},
\begin{equation*}
  \dot w(t)=w\bigl(\mathcal{A}_{cl}x(t)\bigr)=-\kappa\,w(t),\qquad
  u(t)=\mathfrak{K}\,e^{-\kappa t}w(0),
\end{equation*}
so $u$ is $C^\infty$, as \Cref{thm:ISS} requires.

\paragraph*{Item (ii).} By the conclusion of item (i), $\dot w=-\kappa w$,
which gives \eqref{eq:wexact}. For $(Y_0,\eta_0)\in\Hs_{cl}$, the mild
solution $x(t)$ satisfies $\int_0^tx(\sigma)\dd\sigma\in D(\mathcal{A}_{cl})$ and
$x(t)-x(0)=\mathcal{A}_{cl}\int_0^tx(\sigma)\dd\sigma$
\cite[Thm.~1.2.4]{pazy2012semigroups}; applying $w$ and \eqref{eq:fwdid}
gives $w(x(t))-w(x(0))=-\kappa\int_0^tw(x(\sigma))\dd\sigma$, hence
\eqref{eq:wexact}. Since $z_s-z_s^*=\eta=w+\Lk[Y]$ and
$\Lk[Y]^2\le J_\Lk\Vc(Y)$ (\Cref{ssec:issw}), we get \eqref{eq:etasplit}.

\paragraph*{Item (iii).} Along the closed loop, $\dot w=-\kappa w$ by item
(ii), so $\frac{d}{dt}(\rho w^2)=-2\rho\kappa w^2$. By \Cref{prop:ISSw} with
$u=\mathfrak{K}w$ and $\dot u=-\kappa\mathfrak{K}w$, and by \eqref{eq:bw},
\begin{equation}\label{eq:assembly}
  \dot\Uc\le-\alpha_w\Vc-\delta\,\Dpar-\bigl(2\rho\kappa-b_w\bigr)w^2 .
\end{equation}
We take $\nu_2$, $\vartheta$ and $\delta$ such that $\alpha_w>0$, which is
possible under \eqref{eq:smallgain} (for instance as in \eqref{eq:alphawpos}),
and then $\rho>b_w/(2\kappa)$
($b_w$ depends on $\nu_2$ and $\vartheta$, not on $\rho$). Since
$\Dpar\ge0$, \eqref{eq:assembly} gives
\begin{equation*}
  \dot\Uc\le-\beta\,\Uc,\qquad
  \beta\triangleq\min\Bigl(\alpha_w,\,2\kappa-\frac{b_w}{\rho}\Bigr)>0 .
\end{equation*}
By \eqref{eq:normequiv}, $\Lk[Y]^2\le J_\Lk\Vc(Y)$ and
$(a+b)^2\le2a^2+2b^2$,
\begin{equation*}
  \|Y\|_\Hs^2+\eta^2\le C_1\,\Uc(Y,\eta),\qquad
  \Uc(Y,\eta)\le C_2\bigl(\|Y\|_\Hs^2+\eta^2\bigr),
\end{equation*}
with $C_1\triangleq\max\bigl(c_1^{-1}+2J_\Lk,\,2/\rho\bigr)$ and
$C_2\triangleq\max\bigl((1+2\rho J_\Lk)c_2,\,2\rho\bigr)$. This gives
\eqref{eq:GES} with $C=C_1C_2$.
\end{proof}

\section{An illustrative example}
\label{sec:num}

The design of \Cref{sec:fwd} applies to \eqref{eq:model} with any closures
$\sigma_e$ and $\mathcal{R}$ satisfying the assumptions of \Cref{sec:model}.
We apply it to a laboratory-scale continuous settler with a bed of $2$~mm
compact polystyrene beads, chosen because every hypothesis of
\Cref{thm:fwd_GES}, including \Cref{cond:sg}, can be checked on
it.

\subsection{Setup}
\label{ssec:setup}

The settler, of depth $z_b=1$~m, is fed at the top with a suspension of compact (non-expanded) polystyrene beads, $\rho_s=1050$~kg/m$^3$, of diameter $2$~mm, in water
($\rho_f=1000$~kg/m$^3$, $g=9.81$~m/s$^2$), at the solid
fraction $\epsin=0.11$, and drained at the bottom by a pump whose velocity
$v_m$ is the control. We choose beads whose density is close to that of
water: the apparent gravity $K=g(\rho_s-\rho_f)/\rho_s$ is $0.4671$~m/s$^2$
here against $6.108$~m/s$^2$ for a mineral grain ($\rho_s=2650$~kg/m$^3$).
The relaxation rate $\gamma=15.18$~s$^{-1}$ is fixed so that the
free-settling velocity $K/\gamma$ equals the terminal velocity of an isolated
bead, $v_0=3.077\times10^{-2}$~m/s at a particle Reynolds number of $62$.
 Above the critical threshold
$\eps_c=0.55$ the beads are in contact and carry an interparticle effective stress; we take
the constitutive laws
\begin{equation}\label{eq:closures}
  \sigma_e(\epss)= \begin{cases}
    0, & \epss\le\eps_c,\\[2pt]
    \frac{\sigma_0}{\rho_s\eps_c^{n_s}}
  \bigl(\epss^{n_s}-\eps_c^{n_s}\bigr), & \epss>\eps_c,
  \end{cases}
  \qquad
  \mathcal{R}(\epss)=\begin{cases}
    \gamma, & \epss\le\eps_c,\\[2pt]
    \gamma\,\dfrac{\epss(1-\eps_c)}{(1-\epss)\eps_c}, & \epss>\eps_c,
  \end{cases}
\end{equation}
with $n_s=5$ and $\sigma_0=10^{5}$~Pa; $\sigma_e$ is a stress divided by
$\rho_s$, in m$^2$/s$^2$. The physically meaningful quantity is
not $\sigma_0$ but the compressive modulus it produces at the operating point,
$\rho_sc_s^2=\rho_s\sigma_e'(\epss^*)=0.909$~MPa, since only
$\sigma_e'(\epss^*)$ and $\sigma_e''(\epss^*)$ enter the linear design.
The target is a sediment bed of thickness $L\triangleq z_b-z_s^*=20$~cm.

This is not an industrial operating point: the underflow is $27.7$~m/h and
the bed is renewed in $26$~s.

\subsection{Target equilibrium and forwarding design}
\label{ssec:design}

The underflow $v_m^*$ follows from a shooting method: the compression profile is
integrated from \eqref{eq:steadyODE} with the interface trace given by
\eqref{eq:steadyRH}, and $v_m^*$ is adjusted so that $\epss^*(z_b)=f_s^*/v_m^*$.
The result is a uniform feed fraction above the bed, a jump
$\jump{\epss^*}>0$ at $z_s^*$ and a stratified compression profile below
(\Cref{fig:min}). The design then evaluates $\Go$, solves \eqref{eq:cstar}
by bisection and builds the kernel of \Cref{lem:kernel};
\Cref{tab:design} lists the outcome, and \Cref{ass:fwd} holds.

\begin{table}[t]
\centering\small
\begin{tabular}{@{}llrl@{}}
\toprule
Quantity & symbol & value & unit\\
\midrule
Underflow velocity (shooting) & $v_m^*$ & $7.6915\times10^{-3}$ & m/s\\
 Free-settling speed & $\lambda=v_m^*+K/\gamma$ & $3.8462\times10^{-2}$ & m/s\\ 
Interface jump & $\jump{\epss^*}$ & $0.44+1.5\times10^{-7}$ & ---\\
 Compression eigenvalues at $z_s^*$ & $(\lambda_3,\lambda_4)$ & $(29.432,-29.417)$ & m/s\\ 
 Compression eigenvalues at $z_b$ & $(\lambda_3,\lambda_4)$ & $(29.439,-29.423)$ & m/s\\ 
Interface reflection coefficient & $-k^2=-\lambda_4^2/\lambda_3^2$ & $-0.998955$ & ---\\
 Bound on $\mu_2$, \eqref{eq:mu2bar} & $\bar\mu_2$ & $1.3069\times10^{-3}$ & m$^{-1}$\\ 
 Free-settling transit time & $T_1=z_s^*/\lambda$ & $20.800$ & s\\ 
Interface pole & $c^*$ & $-5.1849043728\times10^{-6}$ & s$^{-1}$\\
Non-resonance function & $\Delta_z(c^*)$ & $1.99960$ & ---\\
 Velocity factor & $\gamma_v(c^*)$ & $0.99993$ & ---\\ 
Input gain & $G$ & $1.00007$ & ---\\
\bottomrule
\end{tabular}
\caption{Design quantities at the equilibrium of \Cref{ssec:setup}. The gain
$G$ has no unit because $u$ and $\dot w$ are both velocities.  The
pole $c^*$ is computed by bisection on \eqref{eq:cstar} with tolerance
$10^{-15}$, the kernel being integrated on a grid of $1200$ points of
$\Om_2^*$; the eleven digits given do not change when this grid is
refined.}
\label{tab:design}
\end{table}

The pole $c^*$ is negative: with a constant underflow, the interface
mode is stable, but its time constant $|c^*|^{-1}=1.93\times10^{5}$~s, about
$54$ hours, is far too slow, so the feedback is needed to set the rate. A
quasi-static mass balance explains the value: displacing the
interface by $\eta$ at fixed $v_m^*$, with the trace $\epss^*(z_s^{*+})$ fixed
by \eqref{eq:steadyRH}, lowers the bottom fraction by $\partial_x\epss^*\eta$,
so the outflow falls below the feed flux by $v_m^*\partial_x\epss^*\eta$, which
the bed gains at the rate $-\jump{\epss^*}\dot\eta$: hence
$\dot\eta\approx-v_m^*\partial_x\epss^*(z_s^{*+})\eta/\jump{\epss^*}
=-5.187\times10^{-6}\,\eta$~s$^{-1}$, within a relative $5\times10^{-4}$ of $c^*$. The
bed is indeed quasi-steady on the scale of the interface: with
$\mathrm{Sed}=2.6\times10^{-4}$ the compression waves cross it in $7$~ms.
The compression margin of the
equilibrium, $\epss^*(z_s^{*+})-\eps_c=1.5\times10^{-7}$ (\appsteady{}), is
small, for a structural reason: to first order, \eqref{eq:steadyRH} gives
$\epss^*(z_s^{*+})-\eps_c\approx\theta(1-\theta)\eps_c\lambda^2/L^+$ with
$\theta=\epsin/\eps_c$, so a stiff bed (large $L^+$), which
\Cref{ssec:feasibility} needs, has a thin margin. A perturbation that lowers the
compression-side trace by more than $1.5\times10^{-7}$ in
volume fraction leaves the admissible class. This does not affect the linear
results of \Cref{thm:fwd_GES}, but it limits what they imply for
\eqref{eq:model} (\Cref{sec:conclusion}).

We check $c^*$ with a computation that does not use the kernel. Replacing
$\partial_t$ by $s$ in \eqref{eq:plant}, \eqref{eq:bc} and eliminating the
distributed variables, in closed form on $\Om_1^*$ and by a Runge--Kutta
integration on $\Om_2^*$, gives a homogeneous linear system whose unknowns
are $Y_3(z_s^*)$, $Y_4(z_s^*)$, $\eta$, $\Go\Yo$ and $u$, and whose
determinant $D(s)$ is
the characteristic function of the open-loop generator. With $u=0$, the inflow
on $\Om_1^*$ is homogeneous, so $Y^{(1)}\equiv0$ and the system reduces to a
$3\times3$ system in $(Y_4(z_s^*),\eta,\Go\Yo)$. The zero of $D$ nearest the
origin is $-5.1849043728\times10^{-6}$, equal to $c^*$ in all eleven digits given, within the bisection tolerance
of \Cref{tab:design}. This computation uses the same equilibrium and
linearized coefficients as the design, but neither \eqref{eq:sylvester} nor
\eqref{eq:matching}.

\begin{figure}[t]
\centering
\includegraphics[width=0.55\textwidth]{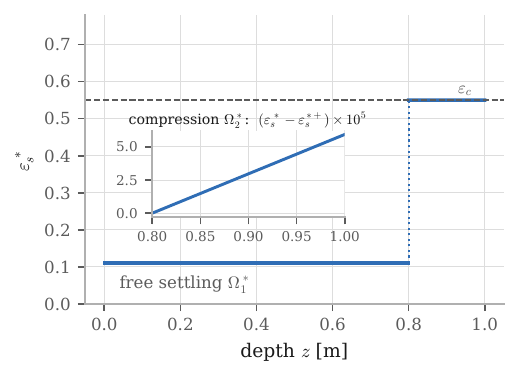}
\caption{Stationary admissible profile $\epss^*$. Inset: deviation
$\epss^*-\epss^*(z_s^{*+})$ on the compression zone $\Om_2^*$, multiplied by
$10^5$.}
\label{fig:min}
\end{figure}

\subsection{The small-gain condition on the example}
\label{ssec:feasibility}

To check \Cref{cond:sg}, we fix $p_3=1$ and take the other parameters of
\eqref{eq:V} and of \Cref{prop:ISSw} from \Cref{tab:weights}. On $\Om_1^*$,
$2\mu_1\lambda-\nu_1=\bar\alpha$ and
$2(\mu_1\lambda+\gamma)p_{2x}-\nu_1^{-1}p_1=\bar\alpha\,p_{2x}$ (the conditions of
\appiss{} with equality). The weights $p_1$, $p_2$ are large so that $J_\Lk$
reduces to its $\Om_2^*$ part and $G_{o,2}^2/d_2=1.3\times10^{-12}$ is
negligible in $\Sigma_b$; this increases the constant $C$ below, but does not affect the rate.
The pair $(\mu_2,r)$, $r=p_4/p_3$, is found numerically: the constants of
\Cref{cond:sg} are evaluated on a grid of pairs that satisfy
\eqref{eq:window}, which is narrow ($r\in(0.9989443,1.0000004)$ for
$\nu_i=10^{-8}$, $\nu_b=10^{-5}$, and $d_4=29.401428-29.370713$), and a pair
for which \eqref{eq:smallgain} holds is retained.

\begin{table}[t]
\centering\small
\begin{tabular}{@{}lll@{}}
\toprule
Parameter & value & main effect\\
\midrule
$p_3$ & $1$ & normalization\\
$r=p_4/p_3$ & $0.999989$ & $\bar\alpha$, $d_4$\\
$\mu_2$ & $0.980\,\bar\mu_2$ & $\bar\alpha$, $d_4$\\
$\nu_1$ & $0.03$~s$^{-1}$ & decay on $\Om_1^*$\\
$\mu_1$ & $(\bar\alpha+\nu_1)/(2\lambda)=1.37$~m$^{-1}$ & decay on $\Om_1^*$\\
$p_1e^{-2\mu_1z_s^*}=p_2e^{-2\mu_1z_s^*}$ & $10^{14}$ & $J_\Lk$, $d_1$, $d_2$, $\sigma_u$\\
$p_{2x}$ & $p_1/\bigl(\nu_1(2\mu_1\lambda+2\gamma-\bar\alpha)\bigr)=9.8\times10^{14}$ & $\sigma_{\dot u}$\\
$\nu_i$ & $10^{-8}$ & $d_1$, $d_2$, $d_4$\\
$\nu_b$ & $10^{-5}$ & $\sigma_u$\\
$\nu_2$ & $\bar\alpha/10=7.53\times10^{-3}$ & $\alpha_w$, $\sigma_u$, $\sigma_{{\rm vol},u}$\\
$\vartheta$ & $\alpha_{sg}/4=1.34\times10^{-2}$ & $\alpha_w$, $\sigma_w$\\
$\delta$ & $1/2$ & $\alpha_w$\\
\bottomrule
\end{tabular}
\caption{Parameters of the Lyapunov functional \eqref{eq:V} and of
\Cref{prop:ISSw} for the illustrative example. The rates $\bar\alpha$, $\alpha_{sg}$,
$\alpha_w$, $\nu_1$, $\nu_2$ and $\vartheta$ are in s$^{-1}$, $\mu_1$ and $\mu_2$
in m$^{-1}$; the weights and the Young parameters $\nu_i$, $\nu_b$, $\delta$
are dimensionless. All quantities are computed in double precision on a grid
of $1200$ points of $\Om_2^*$; interval arithmetic is not used.}
\label{tab:weights}
\end{table}

\noindent With these values,
\begin{equation}\label{eq:numsg}
  \bar\alpha=7.530\times10^{-2},\qquad
  a_\eta=-c^*+\sqrt{J_\Lk I_\eta}=1.9847\times10^{-2},\qquad
  a_o=\Sigma_bI_o=1.7069\times10^{-3},
\end{equation}
so that $a_\eta+a_o=2.1554\times10^{-2}$, $(a_\eta+a_o)/\bar\alpha=0.286$,
and \eqref{eq:smallgain} holds. The other constants are
$-c^*=5.185\times10^{-6}$, $J_\Lk=2.066$, $I_\eta=1.905\times10^{-4}$,
$I_o=2.934\times10^{-9}$, $d_1=d_2=3.846\times10^{12}$,
$d_4=3.072\times10^{-2}$, $\Sigma_b=5.818\times10^{5}$ and, for the terms in
$u$, $\sigma_u=1.1\times10^{17}$ (the bottom term of \eqref{eq:sigu}
contributes $1.0\times10^{3}$ of it), $\sigma_{\dot u}=3.1\times10^{14}$ and
$b_w=9.7\times10^{15}$ at $\kappa=0.3$~s$^{-1}$.

All three items of \Cref{thm:fwd_GES} apply. The proof of item~(iii) gives
the rate $\beta=\min(\alpha_w,2\kappa-b_w/\rho)$, with
$\alpha_w=3.11\times10^{-2}$~s$^{-1}$ for the values of \Cref{tab:weights}.
For $\rho\ge b_w/(2\kappa-\alpha_w)=1.7\times10^{16}$, this gives
$\beta=\alpha_w=3.11\times10^{-2}$~s$^{-1}$. Since
$\alpha_w<\bar\alpha-a_\eta-a_o$ for every choice of $\nu_2$, $\vartheta$
and $\delta$, and $\beta<2\kappa$, every choice of these parameters gives
\begin{equation}\label{eq:betacert}
  \beta<\min\bigl(\bar\alpha-a_\eta-a_o,\,2\kappa\bigr)
  =5.374\times10^{-2}~\text{s}^{-1},
\end{equation}
less than a fifth of the rate $\kappa$ assigned to $w$. The actual decay
rate of the closed loop lies between $\beta$ and $\kappa$, since $w$ decays
exactly at the rate $\kappa$; the characteristic determinant of
\Cref{ssec:design} with $u=\mathfrak{K}w$ would locate it.

The ratio $(a_\eta+a_o)/\bar\alpha$ depends mainly on the stiffness and on
the thickness of the bed. At the operating point $f_s^*\approx\epss^*v_m^*$
in the bed, so
$\Bb^{(2)}\approx-\tfrac{\gamma}{2}\bigl[\begin{smallmatrix}1&-1\\-1&1
\end{smallmatrix}\bigr]$ damps only the mode $Y_3-Y_4$, and $\bar\alpha$
comes from the transport term:
$\bar\alpha\approx2\mu_2|\lambda_4|\approx2v_m^*/L=0.077$~s$^{-1}$, against
the computed $0.075$~s$^{-1}$. This is twice the renewal rate $v_m^*/L$ of the bed; at fixed $v_m^*$ it
depends neither on $c_s$ nor on $\gamma$.
\begin{itemize}[leftmargin=1.5em,itemsep=1pt]
\item \emph{Stiffness.} $a_\eta$ decreases like $c_s^{-1}$ and $a_o$ like
$c_s^{-2}$, while $\bar\alpha$ does not depend on $c_s$. The ratio is $0.027$
at $\rho_sc_s^2=90.9$~MPa, $0.086$ at $9.09$~MPa, $0.286$ here
($0.909$~MPa) and $0.558$ at $0.273$~MPa; with weights of the form of
\Cref{tab:weights}, the condition fails below about $0.10$~MPa.
\item \emph{Thickness.} $\bar\alpha\approx2v_m^*/L$ decreases like $1/L$,
$a_o$ grows like $L$ and $a_\eta$ does not change. The ratio is $0.067$ at
$L=5$~cm, $0.137$ at $10$~cm, $0.286$ here ($20$~cm), $0.447$ at $30$~cm and
$0.802$ at $50$~cm; with the same weights, the condition fails above about
$60$~cm.
\end{itemize}

The constant $C$ of \Cref{thm:fwd_GES}(iii) is large. By the proof,
$C=C_1C_2\ge2\rho J_\Lk\,c_2/c_1$. The weights span nearly fifteen orders of
magnitude, so $c_2/c_1=9.8\times10^{14}$ in \eqref{eq:normequiv}, and
$\rho\ge1.7\times10^{16}$, so $C\ge6.9\times10^{31}$ and \eqref{eq:GES} gives
no information before $t=\ln C/\beta\approx2.4\times10^{3}$~s. Since
$\nu_2<\bar\alpha/3$, changing $\nu_2$ reduces $\sigma_u$, hence $\rho$, by a
factor $3.3$ at most, and $C$ stays above $2\times10^{31}$. The bound
\eqref{eq:betacert} is a statement about the rate only.

\subsection{Simulation of the closed loop}
\label{ssec:nonlinear}

\Cref{thm:fwd_GES} concerns the linearized loop. \Cref{fig:nl} shows
simulations of the \emph{nonlinear} system \eqref{eq:model}--\eqref{eq:BCbot}, solved on the
physical domain $(0,z_b)$ by a finite-volume scheme with Rusanov fluxes and a second-order SSP
Runge--Kutta step (CFL $0.4$) on
$800$ cells ($160$ in the bed), the interface being captured as a shock and
located by the crossing of $(\epsin+\eps_c)/2$. At each
step the state is mapped onto the reference configuration and
$u=\mathfrak{K}\bigl(\eta-\Lk[Y]\bigr)$ is applied through $v_m=v_m^*+u$, the
reference ($z_s^*$, $v_m^*$, $\epss^*$, $\Lk$, $\mathfrak{K}$) being the
equilibrium of the current set point and feed fraction.

\begin{figure}[t]
\centering
\includegraphics[width=0.98\textwidth]{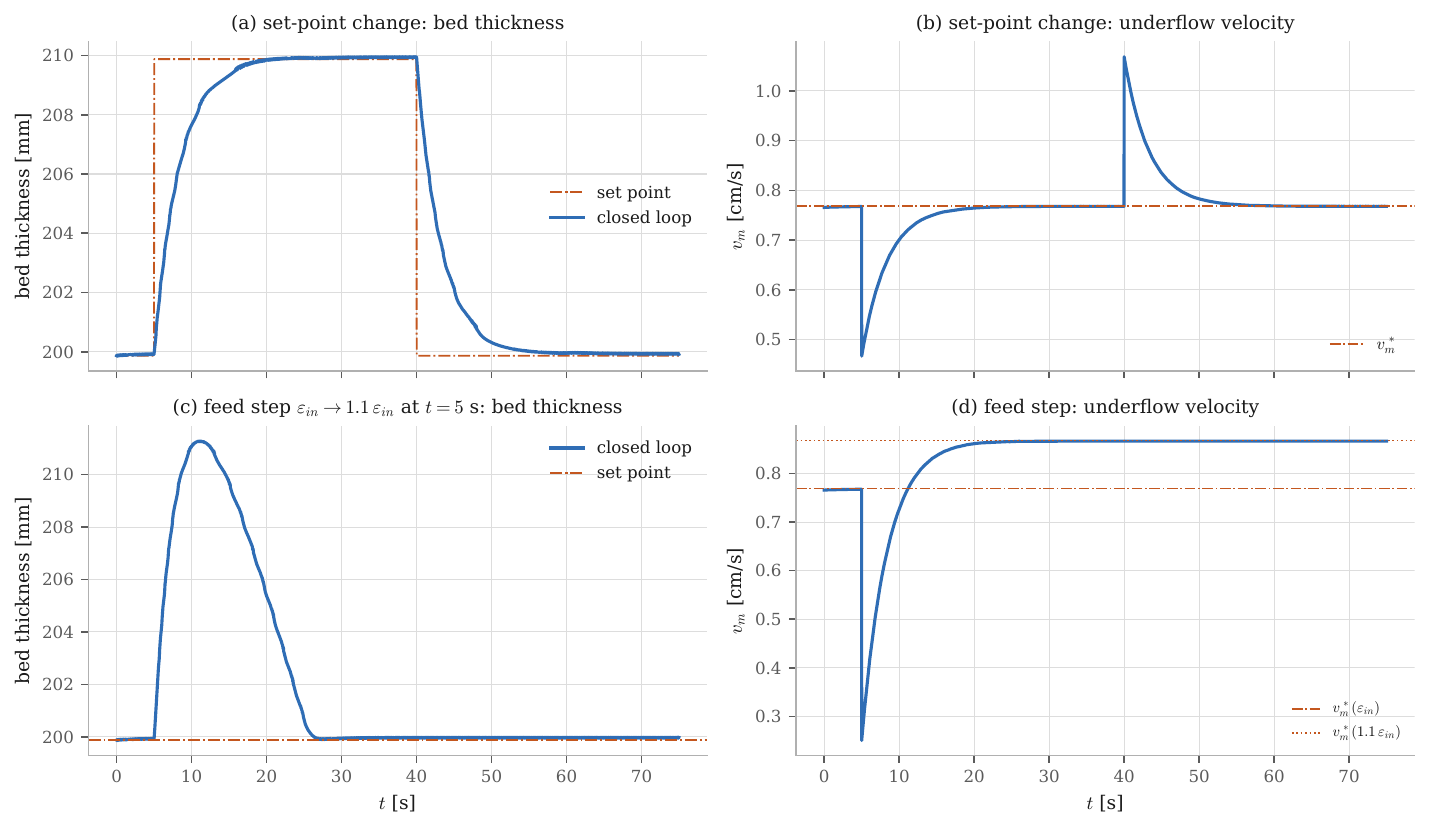}
\caption{Nonlinear closed loop, $\kappa=0.3$~s$^{-1}$. (a)--(b): the set
point of the bed thickness goes from $200$ to $210$~mm at $t=5$~s and back at
$t=40$~s. (c)--(d): the feed fraction goes from $\epsin$ to $1.1\,\epsin$ at
$t=5$~s, and the reference is switched to the equilibrium of the new feed
fraction; the new suspension reaches the bed at $t\approx25.8$~s.
(a), (c): bed thickness; (b), (d): underflow velocity.}
\label{fig:nl}
\end{figure}

In (a)--(b) the set point goes from $200$ to $210$~mm and back. The bed
is within $1$~mm of its set point after $8$~s, as $10\,e^{-\kappa t}$~mm
predicts, and the underflow stays between $0.467$ and $1.07$~cm/s. In
(c)--(d) the feed fraction rises by $10\%$, which \Cref{thm:fwd_GES} does not
cover; the reference is switched to the equilibrium of the new feed
fraction ($v_m^*=0.868$~cm/s). The suspension still in $\Om_1$ is a deficit
for the new reference, so the underflow first drops to $0.25$~cm/s and the
bed thickens to $211.3$~mm; it thins again and is back at its set point when
the new suspension reaches the bed, at $t\approx25.8$~s. The scheme adds a
numerical drift of $16\ \mu$m/s, which leaves a bias of at most $0.1$~mm,
since the loop drives $w$, not $\eta$, to zero: forwarding adds no integral
action on $\eta$. This simulation is not a proof: the scheme does not resolve the phase separation of
\Cref{def:admissible}, whose margin at this equilibrium is $1.5\times10^{-7}$, and the layer just under
the captured front is below $\eps_c$. It suggests that the law
tolerates finite changes of set point and of feed fraction, and the nonlinear terms neglected by the
linearization. The
passage from the linearized statement to \eqref{eq:model} is the main open
question, together with the saturation $v_m\ge0$, which the law does not
enforce.

\section{Conclusion}
\label{sec:conclusion}

Isolating the interface by the Rankine--Hugoniot relations and mapping the
subdomains onto fixed intervals turns the control of a moving shock into a
fixed-boundary problem whose linearization has the form for which forwarding
was designed.

We now state what is certified; this concerns only the linearized
loop. Under the design assumptions alone (\Cref{ass:fwd}), the feedback makes
the variable $w=\eta-\Lk[Y]$ decay exactly at the chosen rate $\kappa$, and
the interface error $\eta$ is bounded by this decaying term plus a term
controlled by the energy of the distributed layer
(\Cref{thm:fwd_GES}(ii)). The interface pole $c^*$ used by the design is
reproduced to the
solver tolerance by a computation that uses none of the kernel equations,
and to a relative $5\times10^{-4}$ by a quasi-static mass balance of the bed. The exponential stability of the whole linearized loop
(\Cref{thm:fwd_GES}(iii)) requires in addition the small-gain condition of
\Cref{cond:sg}, which \Cref{ssec:feasibility} examines: the certificate holds on the illustrative example, but it depends on
the stiffness of the sediment and on the thickness of the bed, and, for
weights of the form used in the example, is lost below a compressive modulus
of about $0.10$~MPa or above a bed of about $60$~cm. It does not apply to
activated sludge, which is several orders of magnitude softer and settles in
deeper units; the example is a stiff laboratory bed because the
sufficient condition can be satisfied for such a bed. With this stiffness, the
equilibrium has a compression margin of $1.5\times10^{-7}$ (\Cref{ssec:design}):
the admissible class is thin exactly where the certificate holds, so the
linear result certifies the rate of a loop whose nonlinear counterpart
leaves the class of \Cref{thm:wp} under perturbations of order $10^{-7}$ of
the compression-side trace. The captured-shock simulation of
\Cref{ssec:nonlinear}, whose layer under the front is outside that class,
nevertheless behaves as the linear theory predicts; whether the kinematic
admissibility of \Cref{def:admissible} is the right nonlinear framework, or
whether a weaker one (a contact layer below the front) is needed, remains an open question.
The conservatism of the certificate, a factor ten between $\beta$ and
$\kappa$ and a transient constant above $10^{31}$, has a clear cause:
the bed is quasi-steady on the scale of the interface (\Cref{ssec:design}),
and a quadratic estimate made at each time accounts for the reflections of the
fast waves, but not for the damping that the layer accumulates along its characteristics.
This work could be extended in three directions: a
less conservative certificate, by a singular-perturbation reduction of the bed,
by frequency-domain techniques \cite{amirat2026insights}
or by a backstepping transformation distributing the damping
\cite{vazquez2026backstepping}; an
observer, since the law needs the profiles and only the interface level is
measured; and a design that takes the saturation $v_m\ge0$ into account.

\bigskip
\noindent\textbf{Code availability.} The code reproducing \Cref{sec:num} is at\\
\href{https://github.com/Vincent69410/sedimentation-forwarding}%
{\texttt{github.com/Vincent69410/sedimentation-forwarding}}.

\medskip
\noindent\textbf{Use of AI assistance.} The authors used Claude (Anthropic)
in the preparation of this manuscript: to carry out the numerical experiments of
\Cref{sec:num} and the cross-checks reported there, to draft and
revise the text, and to perform internal consistency reviews. The authors have
verified the statements, proofs and numerical results reported here and take
full responsibility for the content.

\appendix

\section{Construction of the stationary profile}
\label[appendix]{app:steady}

We construct $(v_m^*,\epss^*)$ from $z_s^*$, as announced in
\Cref{ssec:steady}, in three steps: the interface value, the compression
profile, and the underflow velocity. Throughout, $z_s^*\in(0,z_b)$ and
$\epsin\in(0,\eps_c)$ are given, and $v_m^*>0$ is the unknown, with
$\lambda=v_m^*+K/\gamma$ and $f_s^*=\epsin\lambda$.

\paragraph*{Step 1: the interface value.} Write $\eps^+=\epss^*(z_s^{*+})$. The
stationary jump \eqref{eq:steadyRH} is $\Sigma(\eps^+)=0$ for
\begin{equation}\label{eq:app_Sigma}
  \Sigma(\eps)\triangleq\sigma_e(\eps)
  -(f_s^*)^2\Bigl(\frac{1}{\epsin}-\frac{1}{\eps}\Bigr),
  \qquad \eps\in[\eps_c,1).
\end{equation}
Two properties of $\Sigma$ give the result. First, $\sigma_e(\eps_c)=0$ and
$\epsin<\eps_c$ give
\begin{equation}\label{eq:app_Sigc}
  \Sigma(\eps_c)=-(f_s^*)^2\Bigl(\frac{1}{\epsin}-\frac{1}{\eps_c}\Bigr)<0 .
\end{equation}
Second, with $v_s(\eps)=f_s^*/\eps$ and $c_s(\eps)=\sqrt{\sigma_e'(\eps)}$,
\begin{equation}\label{eq:app_Sigprime}
  \Sigma'(\eps)=\sigma_e'(\eps)-\frac{(f_s^*)^2}{\eps^2}
  =c_s(\eps)^2-v_s(\eps)^2 ,
\end{equation}
so $\Sigma$ increases where the stationary flow is subcritical. Let
$\eps^+$ be the smallest root of $\Sigma$ in $(\eps_c,1)$, which exists as soon
as $\Sigma$ takes a positive value there; by \eqref{eq:app_Sigc},
$\Sigma'(\eps^+)\ge0$, and
\begin{equation}\label{eq:app_transversal}
  \Sigma'(\eps^+)>0
  \iff v_s(\eps^+)<c_s(\eps^+)
  \iff \text{\eqref{eq:subcritical} holds at }z=z_s^{*},
\end{equation}
so the subcriticality required of an admissible profile is the transversality
of that root. A transversal root is simple, so $\eps^+$ is the only root of
$\Sigma$ in a neighbourhood of it.

\paragraph*{Step 2: the compression profile.} On $\Om_2^*$ the flux $f_s^*$ is
constant, so $\partial_\eps F(\eps,f_s^*)=\Sigma'(\eps)$ and
\eqref{eq:steadyODE} is the scalar ODE
\begin{equation}\label{eq:app_odeprofile}
  \Sigma'(\epss^*)\,\frac{d\epss^*}{dz}=\Gamma(\epss^*),
  \qquad
  \Gamma(\eps)\triangleq K\eps-\mathcal{R}(\eps)\bigl(f_s^*-\eps v_m^*\bigr),
\end{equation}
with initial value $\epss^*(z_s^{*+})=\eps^+$. On the set where
$\eps\in(\eps_c,1)$ and $\Sigma'(\eps)>0$, the right-hand side
$\Gamma/\Sigma'$ is $C^1$, since $\sigma_e\in C^2$ and $\mathcal{R}\in C^1$
there; \eqref{eq:app_odeprofile} therefore has a unique maximal solution, and
that solution extends to all of $[z_s^*,z_b]$ provided it stays in a
compact subset of $\{\eps\in(\eps_c,1):\Sigma'(\eps)>0\}$, that is, in
$(\eps_c,1)$ and subcritical. The second condition is
\eqref{eq:subcritical} on $[z_s^*,z_b]$, which is \Cref{def:admissible}(ii)
with $\dot z_s=0$.

\paragraph*{Step 3: the underflow velocity.} Steps 1 and 2 depend on $v_m^*$,
through $\lambda$, $f_s^*$ and $\Gamma$. The bottom condition
\eqref{eq:BCbot}, at stationarity, is $f_s^*=\epss^*(z_b)v_m^*$, that is
\begin{equation}\label{eq:app_shoot}
  R(v_m^*)\triangleq\epss^*(z_b;v_m^*)
  -\epsin\Bigl(1+\frac{K}{\gamma v_m^*}\Bigr)=0 .
\end{equation}
By continuous dependence of \eqref{eq:app_odeprofile} on its parameters and
simplicity of the root $\eps^+$, the map $R$ is continuous on the open set of
$v_m^*$ for which Steps~1 and~2 succeed, and a zero of $R$ at which
\eqref{eq:subcritical} holds on $[z_s^*,z_b]$ is a stationary admissible
profile with interface at $z_s^*$; conversely, every such profile whose
interface trace is the smallest root of $\Sigma$ in $(\eps_c,1)$ is a zero of
$R$, and for \eqref{eq:closures} this root is the only one, since $\Sigma'$
is increasing. Numerically, $v_m^*$ is
computed by bisection on $R$; each evaluation of $R$ solves
\eqref{eq:app_Sigma} by bisection and \eqref{eq:app_odeprofile} by an
adaptive Runge--Kutta scheme.

At the operating point of \Cref{sec:num}, $v_m^*=7.6915\times10^{-3}$~m/s,
$\eps^+=0.55000015$ and
\begin{equation*}
  \Sigma'(\eps^+)=c_s^2-v_s^2=8.6580\times10^{2}-5.9172\times10^{-5}
  =8.6580\times10^{2}\ \mathrm{m^2/s^2},\qquad
  \frac{v_s}{c_s}=2.61\times10^{-4},
\end{equation*}
so the profile is strongly subcritical. At the interface,
$d\epss^*/dz=2.97\times10^{-4}$~m$^{-1}$; over $L=20$~cm, this gives a
relative variation of $1.1\times10^{-4}$ over the bed (inset of \Cref{fig:min}).

\section{Proof of the dissipation inequality}
\label{app:iss}

On $\Om_1^*$, since $u$ does not depend on $x$, \eqref{eq:zone1} gives
\begin{equation}\label{eq:app_zone1}
\begin{gathered}
  \partial_tY_1+\lambda\partial_xY_1=-\partial_xY_2,\qquad
  \partial_tY_2+\lambda\partial_xY_2=-\gamma Y_2+\gamma\epsin u,\\
  \partial_t(\partial_xY_2)+\lambda\partial_x(\partial_xY_2)=-\gamma\,\partial_xY_2 .
\end{gathered}
\end{equation}
Differentiating \eqref{eq:V} and integrating by parts in $x$ gives
$\dot\Vc=\mathcal{I}^{(1)}+\mathcal{I}^{(2)}+\mathcal{B}_0+\mathcal{B}_{z_s^*}
+\mathcal{B}_{z_b}$, where $\mathcal{I}^{(i)}$ are the volume terms on
$\Om_i^*$ and $\mathcal{B}_0$, $\mathcal{B}_{z_s^*}$, $\mathcal{B}_{z_b}$ the
boundary terms at $x=0$, $z_s^*$, $z_b$. We bound them in turn.

\paragraph*{Volume, free-settling zone.} By \eqref{eq:app_zone1},
\begin{multline}\label{eq:app_Ivol1}
  \mathcal{I}^{(1)}=\int_0^{z_s^*}e^{-2\mu_1x}\Bigl[
  -2\mu_1\lambda\bigl(p_1Y_1^2+p_2Y_2^2+p_{2x}(\partial_xY_2)^2\bigr)
  -2p_1Y_1\partial_xY_2\\
  +2p_2\gamma\epsin Y_2u
  -2\gamma\bigl(p_2Y_2^2+p_{2x}(\partial_xY_2)^2\bigr)\Bigr]\dd x .
\end{multline}
Young's inequality, with a parameter $\nu_1>0$ on the Jordan term and with
$\nu_2$ on the source term, gives
\begin{equation*}
  -2p_1Y_1\partial_xY_2\le\nu_1p_1Y_1^2+\nu_1^{-1}p_1(\partial_xY_2)^2,
  \qquad
  2p_2\gamma\epsin Y_2u\le\nu_2p_2Y_2^2+\nu_2^{-1}p_2\gamma^2\epsin^2u^2,
\end{equation*}
hence
\begin{multline}\label{eq:app_I1}
  \mathcal{I}^{(1)}\le-\int_0^{z_s^*}e^{-2\mu_1x}\Bigl[
  p_1(2\mu_1\lambda-\nu_1)Y_1^2+p_2(2\mu_1\lambda+2\gamma-\nu_2)Y_2^2\\
  +\bigl(2(\mu_1\lambda+\gamma)p_{2x}-\nu_1^{-1}p_1\bigr)(\partial_xY_2)^2
  \Bigr]\dd x
  +\nu_2^{-1}\frac{1-e^{-2\mu_1z_s^*}}{2\mu_1}p_2\gamma^2\epsin^2u^2 .
\end{multline}

\paragraph*{Volume, compression zone.} By \eqref{eq:zone2},
\eqref{eq:closedcoupling} and \eqref{eq:Q},
\begin{equation}\label{eq:app_Ivol2}
  \mathcal{I}^{(2)}
  =-\int_{z_s^*}^{z_b}Y^{(2)\top}\Qb\,Y^{(2)}\dd x
  +2\int_{z_s^*}^{z_b}Y^{(2)\top}\Pb\bigl(\Cb^{(2)}_\eta\eta
  +\Cb^{(2)}_{:,2}\Go\Yo+\Db^{(2)}u\bigr)\dd x .
\end{equation}
Young's inequality with parameter $\nu_2$ on each of the three terms of the
second integral gives, for instance,
$2Y^{(2)\top}\Pb\,\Cb^{(2)}_\eta\eta\le\nu_2Y^{(2)\top}\Pb\,Y^{(2)}
+\nu_2^{-1}\bigl(\Cb^{(2)\top}_\eta\Pb\,\Cb^{(2)}_\eta\bigr)\eta^2$. With
\eqref{eq:alphabar}, \eqref{eq:sigu} and \eqref{eq:Ieo},
\begin{equation}\label{eq:app_I2}
  \mathcal{I}^{(2)}\le-(\bar\alpha-3\nu_2)\int_{z_s^*}^{z_b}Y^{(2)\top}\Pb\,
  Y^{(2)}\dd x+\sigma_\eta\eta^2+\sigma_o(\Go\Yo)^2+\sigma_{{\rm vol},u}u^2 .
\end{equation}

\paragraph*{Top boundary.} By \eqref{eq:bc}, $Y_1(t,0)=0$ and
$Y_2(t,0)=\epsin u(t)$, and the second equation of \eqref{eq:app_zone1} at
$x=0$ gives $\epsin\dot u+\lambda\partial_xY_2(t,0)=0$. Hence
\begin{equation}\label{eq:app_B0}
  \mathcal{B}_0=\lambda\bigl[p_1Y_1^2+p_2Y_2^2+p_{2x}(\partial_xY_2)^2
  \bigr](0)
  =\lambda p_2\epsin^2u^2+\frac{p_{2x}\epsin^2}{\lambda}\dot u^2 .
\end{equation}

\paragraph*{Bottom boundary.} The boundary terms of \Cref{ssec:lyap} at
$z_b$, with \eqref{eq:youngb}, give
\begin{equation}\label{eq:app_Bzb}
  \mathcal{B}_{z_b}\le\bigl[-p_3e^{-2\mu_2(z_b-z_s^*)}\lambda_3(z_b)
  +(1+\nu_b)p_4|\lambda_4(z_b)|\bigr]Y_3(z_b)^2
  +(1+\nu_b^{-1})p_4|\lambda_4(z_b)|\bigl(\Db^{(bot)}\bigr)^2u^2,
\end{equation}
and the bracket is negative by the right inequality of \eqref{eq:window}.

\paragraph*{Interface.} With all traces at $z_s^*$, the boundary terms are
$-\lambda e^{-2\mu_1z_s^*}[p_1Y_1^2+p_2Y_2^2+p_{2x}(\partial_xY_2)^2]
+p_3\lambda_3Y_3^2-p_4e^{-2\mu_2(z_b-z_s^*)}|\lambda_4|Y_4^2$.
We drop the term in $(\partial_xY_2)^2$ and bound $Y_3^2$ by \eqref{eq:youngi}.
By \eqref{eq:Dpar}--\eqref{eq:d14}, this gives
\begin{equation}\label{eq:app_Bzs}
  \mathcal{B}_{z_s^*}\le-\Dpar(Y).
\end{equation}

\paragraph*{Choice of the weights.} Fix $\nu_1>0$. Choose $\mu_1$ such that
$2\mu_1\lambda-\nu_1\ge\bar\alpha$ and $2\mu_1\lambda+2\gamma\ge\frac43\bar\alpha$,
so that $2\mu_1\lambda+2\gamma-\nu_2\ge\bar\alpha$ for every
$\nu_2\in(0,\bar\alpha/3)$;
then $p_1,p_2$ such that $d_1>0$ and $d_2>0$; then $p_{2x}$ such that
$2(\mu_1\lambda+\gamma)p_{2x}-\nu_1^{-1}p_1\ge\bar\alpha\,p_{2x}$, which is
possible since $2(\mu_1\lambda+\gamma)>\bar\alpha$. Then \eqref{eq:app_I1}
gives
\begin{equation*}
  \mathcal{I}^{(1)}\le-\bar\alpha\int_0^{z_s^*}e^{-2\mu_1x}
  \bigl[p_1Y_1^2+p_2Y_2^2+p_{2x}(\partial_xY_2)^2\bigr]\dd x
  +\nu_2^{-1}\frac{1-e^{-2\mu_1z_s^*}}{2\mu_1}p_2\gamma^2\epsin^2u^2 .
\end{equation*}
Adding this bound, \eqref{eq:app_I2}, \eqref{eq:app_B0}, \eqref{eq:app_Bzb}
and \eqref{eq:app_Bzs}, and using \eqref{eq:sigu}, gives \eqref{eq:ISS}.
Finally, $d_4>0$ is the left inequality of \eqref{eq:window}. \qed

\section{Proof of the well-posedness theorem}
\label[appendix]{app:wp}

Throughout this appendix, $\|\cdot\|$ is the supremum norm on the set
indicated, $\|\zeta\|_t\triangleq\|\zeta\|_{C^0([0,t])}$ for a function of
time, $\jump{g}=g^+-g^-$ is the jump across the interface, and
$c_s\triangleq\sqrt{\sigma_e'(\epss)}$ is computed with the current state
(in \Cref{sec:fixed}, it was computed with the stationary profile).

The difficulty is that the interface $z_s$ is an unknown. Its velocity
$\dot z_s=\jump{f_s}/\jump{\epss}$ depends on the traces of the solution at
the interface, and these traces depend on $z_s$, which delimits the two zones
$\Om_i(t)$. The proof separates the two dependences: we first solve the
problem for a prescribed interface velocity $\zeta$, and then look for the
velocity $\zeta$ for which the Rankine--Hugoniot condition holds. It is
organized in six steps, one per subsection.
\begin{enumerate}[label=\arabic*.,leftmargin=2em,itemsep=1pt]
\item In \Cref{ssec:app_fixed}, each zone $\Om_i(t)$ is mapped affinely onto
its initial position $\Om_{i,0}$. The interface becomes the fixed point
$x=z_s(0)$, and $\dot z_s$ only appears in the coefficients
(\Cref{lem:app_fixed}). The margin function $\widehat{\Nc}$ computed in these
variables blows up if and only if $\Nc$ does (\Cref{lem:Nequiv}), so the
theorem can be proved with $\widehat{\Nc}$.
\item In \Cref{ssec:app_bc}, we check that the boundary value problem on
each zone is well posed: at each boundary, the number of conditions equals the
number of incoming characteristics. In the Riemann invariants $R_1,R_2$, the
compression zone becomes two transport equations, and the Hugoniot
condition $H=0$ can be solved explicitly for the incoming invariant,
$R_2^+=h(R_1^+,y^-)$, with derivatives bounded in terms of $M$
(\Cref{lem:app_h,lem:app_bc}).
\item In \Cref{ssec:app_frozen}, the interface velocity is replaced by a given
function $\zeta$, which gives the frozen problem. The free-settling zone
receives no data from the interface and is solved first, along its
characteristics: the velocity equation is closed, its first $x$-derivative
satisfies a Riccati equation and its second a linear one, and $\epss$ and
$\partial_x\epss$ follow (\Cref{lem:app_char}). The compression zone is then
solved by \cite[Ch.~2]{li1985boundary}, with $y^-$ as interface data.
\item In \Cref{ssec:app_bounds}, we show that, if $\widehat{\Nc}(0)\le M$,
the frozen solution exists on an interval $[0,\tau_0(M)]$ on which its
margins stay above $\frac1{2M}$ (\Cref{lem:app_frozen}), and that its traces
at the interface depend on $\zeta$ in a Lipschitz way, with a constant
proportional to $t$ (\Cref{lem:app_lip}). All the constants depend only on
$M$ and on the data.
\item In \Cref{ssec:app_local}, the map $\zeta\mapsto\mathcal{T}(\zeta)$,
the Rankine--Hugoniot quotient of the traces of the frozen solution, is shown
to be a contraction for $\tau\le\tau_*(M)$. Its fixed point gives the local
solution, which is unique (\Cref{prop:app_local}).
\item In \Cref{ssec:app_cont}, we use that $\tau_*$ depends only on $M$: the
solution can be restarted from any time $t_0$, on an interval of length
$\tau_*(M)$ where $M$ bounds the margin function at $t_0$, so, if
$\widehat{\Nc}$ stayed bounded, the solution would extend beyond
$T_{\max}$ (\Cref{prop:app_cont}). Finally, the state stays bounded while the
structural margins do, and $\partial_x^2v_s$ and $\partial_x\epss$ stay
bounded while $\partial_xv_s$ does (\Cref{lem:app_ctrl}); the unbounded term
of $\widehat{\Nc}$ is therefore one of those listed in (i)--(iv).
\end{enumerate}

\subsection{Fixed-domain formulation}\label{ssec:app_fixed}

The interface moves, so the two zones change with time. In
\Cref{sssec:app_mapped}, we map each zone onto its position at $t=0$ and
write the problem and the speeds in these variables. Since \Cref{thm:wp} is
stated with the margin function of the physical variables, we compare it, in
\Cref{sssec:app_margin}, with the margin function of the mapped variables, and
show that it is enough to prove the theorem with the latter.

\subsubsection{Mapped equations}\label{sssec:app_mapped}

Let $\Om_{1,0}=(0,z_s(0))$, $\Om_{2,0}=(z_s(0),z_b)$ and, for $z_s\in C^1$,
\begin{equation}\label{eq:app_rho}
  \rho_1(t)\triangleq\frac{z_s(t)}{z_s(0)},\qquad
  \rho_2(t)\triangleq\frac{z_b-z_s(t)}{z_b-z_s(0)},\qquad
  \Phi(t,z)\triangleq\begin{cases}
    z/\rho_1(t), & z\in\Om_1(t),\\
    z_b-(z_b-z)/\rho_2(t), & z\in\Om_2(t).
  \end{cases}
\end{equation}
For each $t$, $\Phi(t,\cdot)$ maps $\Om_i(t)$ onto $\Om_{i,0}$, and $\Phi$
is $C^1$ in $t$. Let $\hat y^{(i)}(t,x)\triangleq y(t,z)$ for
$x=\Phi(t,z)\in\Om_{i,0}$, and $v_s^{(i)}\triangleq f_s^{(i)}/\epss^{(i)}$.

\begin{lemma}[Fixed-domain formulation]\label{lem:app_fixed}
Let $z_s\in C^1([0,T);(0,z_b))$. Then $(y,z_s)$ is an admissible interface
solution on $[0,T)$ if and only if
$\hat y^{(i)}\in C^1([0,T)\times\overline{\Om_{i,0}})$,
$\partial_x^2v_s^{(1)}\in C^0$, the inequalities of \Cref{def:admissible}
hold, and
\begin{equation}\label{eq:app_FBVP}
\left\{\begin{aligned}
 &\partial_t\hat y^{(i)}+\hat A^{(i)}\partial_x\hat y^{(i)}=\hat S^{(i)},
   \quad \hat A^{(i)}\triangleq\rho_i^{-1}A(\hat y^{(i)})-g_i\mathbf{I},
   \quad \hat S^{(i)}\triangleq S_c^{(i)}(\hat y^{(i)},v_m),
   \quad x\in\Om_{i,0},\\
 &\epss^{(1)}(t,0)=\epsin,\qquad
   f_s^{(1)}(t,0)=\epsin\Bigl(v_m+\tfrac{K}{\gamma}\Bigr),\qquad
  f_s^{(2)}(t,z_b)=\epss^{(2)}(t,z_b)\,v_m,\\
 &H=0\ \text{ at }x=z_s(0),\qquad
  \dot z_s=\mathfrak{s}\triangleq\jump{f_s}/\jump{\epss}\ \text{ at }x=z_s(0),
\end{aligned}\right.
\end{equation}
with $g_1\triangleq x\,\dot z_s/z_s$ and
$g_2\triangleq(z_b-x)\,\dot z_s/(z_b-z_s)$, and with
$H=\jump{f_s}^2-\jump{\epss}\jump{F}$ as in \Cref{def:compat}. In the rest of
this appendix, $H$ and $\mathfrak{s}$ are seen as functions of the two traces
$(y^+,y^-)$ at the interface.
\end{lemma}

\begin{proof}
By the chain rule,
\begin{equation}\label{eq:app_diffeo}
  \partial_z=\rho_i^{-1}\partial_x,\qquad
  \partial_t|_z=\partial_t|_x-g_i\,\partial_x,
\end{equation}
so \eqref{eq:model} becomes the first line of \eqref{eq:app_FBVP}, and
\eqref{eq:BCtop}--\eqref{eq:BCbot} the second one; $\hat A^{(i)}$ and
$\hat S^{(i)}$ are the $A_c^{(i)}$ and $S_c^{(i)}$ of \eqref{eq:Ac} with
$z_s^*$ replaced by $z_s(0)$. A function that is $C^1$ on each side of the
$C^1$ curve $z_s$ is a weak solution of \eqref{eq:model} if and only if it is a
classical solution on each side and satisfies \eqref{eq:RH} across the curve
\cite[Ch.~1]{godlewski_raviart_2021}; since $\jump{\epss}>0$, \eqref{eq:RH} is
equivalent to the third line. Since $\Phi(t,\cdot)$ is an affine bijection of
each zone, $C^1$ in $t$, the regularity is the same in both variables.
\end{proof}

By \eqref{eq:jac} and $x=z/\rho_1$, $z_b-x=(z_b-z)/\rho_2$, the eigenvalues
of $\hat A^{(i)}$ are
\begin{align}
  \hat\lambda^{(1)}&=\rho_1^{-1}v_s-g_1
   =\rho_1^{-1}\Bigl(v_s-\frac{z}{z_s}\dot z_s\Bigr),
   \label{eq:app_lam1}\\
  \hat\lambda^{(2)}_{1,2}&=\rho_2^{-1}(v_s\mp c_s)-g_2
   =\rho_2^{-1}\bigl(v_{\rm rel}\mp c_s\bigr),
  &&v_{\rm rel}\triangleq v_s-\frac{z_b-z}{z_b-z_s}\dot z_s .
  \label{eq:app_lam2}
\end{align}
On $\Om_{1,0}$, $\sigma_e'=0$, so $\hat\lambda^{(1)}$ is a double eigenvalue
and
\begin{equation*}
  \hat A^{(1)}-\hat\lambda^{(1)}\mathbf{I}=\rho_1^{-1}\bigl(A-v_s\mathbf{I}\bigr)
  =\rho_1^{-1}\begin{pmatrix}-v_s&1\\-v_s^2&v_s\end{pmatrix}\neq0:
\end{equation*}
$\hat A^{(1)}$ has the single eigenvector $(1,v_s)^\top$ (Jordan block).
On $\Om_{2,0}$, since $-\hat\lambda^{(2)}_1=\rho_2^{-1}(c_s-v_{\rm rel})$,
\begin{equation}\label{eq:app_minlam}
  \min\bigl\{-\hat\lambda^{(2)}_1,\hat\lambda^{(2)}_2\bigr\}
  =\rho_2^{-1}\bigl(c_s-|v_{\rm rel}|\bigr).
\end{equation}
At $t=0$, $\rho_i=1$ and \Cref{ass:wp}(ii) reads
$\hat\lambda^{(1)}(0,\cdot)>0$ and
$\hat\lambda^{(2)}_1(0,\cdot)<0<\hat\lambda^{(2)}_2(0,\cdot)$;
$\hat\lambda^{(2)}_1$ and $\hat\lambda^{(2)}_2$ are the mapped $\lambda_4$ and
$\lambda_3$ of \Cref{sec:model}.

\subsubsection{Mapped margin function}\label{sssec:app_margin}

\Cref{thm:wp} is stated with the margin function $\Nc$ of the physical
variables. Since the proof works in the mapped variables, we use the
counterpart of $\Nc$ in these variables,
\begin{equation}\label{eq:app_margin}
\begin{aligned}
  \widehat{\Nc}_0(t)&\triangleq
  \sum_{i=1}^{2}\bigl\|\hat y^{(i)}(t,\cdot)\bigr\|_{C^0(\overline{\Om_{i,0}})}
  +\frac{1}{\jump{\epss}}+\frac{1}{\hat m_{\rm kin}}+\frac{1}{m_{\rm phase}}\\
  &\qquad+\frac{1}{\mathrm{dist}(z_s,\partial\Om)}
  +\frac{1}{\inf_{\Om}\epss}+\frac{1}{1-\sup_{\Om}\epss},\\
  \widehat{\Nc}(t)&\triangleq\widehat{\Nc}_0(t)
  +\sum_{i=1}^{2}\bigl\|\partial_x\hat y^{(i)}(t,\cdot)\bigr\|_{C^0(\overline{\Om_{i,0}})}
  +\bigl\|\partial_x^2v_s^{(1)}(t,\cdot)\bigr\|_{C^0(\overline{\Om_{1,0}})},
\end{aligned}
\end{equation}
with $\hat m_{\rm kin}\triangleq\min\bigl\{\inf_{\Om_{1,0}}\hat\lambda^{(1)},\
\inf_{\Om_{2,0}}\min\{-\hat\lambda^{(2)}_1,\hat\lambda^{(2)}_2\}\bigr\}$, so that
$\hat m_{\rm kin}>0$ encodes the three inequalities
$\hat\lambda^{(1)}>0$, $\hat\lambda^{(2)}_1<0<\hat\lambda^{(2)}_2$.
The other terms are those of \eqref{eq:margin}, unchanged because
$\Phi(t,\cdot)$ is a bijection. Since $\jump{\epss}\le1$,
$\widehat{\Nc}\ge1$. For a base time $t_0$, $\widehat{\Nc}^{[t_0]}$ denotes
\eqref{eq:app_margin} computed with the map \eqref{eq:app_rho} based at
$t_0$, that is, with $z_s(t_0)$ in place of $z_s(0)$.

\begin{lemma}[$\widehat{\Nc}$ and $\Nc$ blow up together]\label{lem:Nequiv}
Let $d(t)\triangleq\mathrm{dist}(z_s(t),\partial\Om)$ and $d_0\triangleq d(0)$.
Along an admissible interface solution, for every $t$,
\begin{gather}
  \widehat{\Nc}(t)\le\Bigl(\frac{z_b}{d_0}\Bigr)^2\Nc(t),\qquad
  \Nc(t)\le(1+z_b)^2\bigl(1+\widehat{\Nc}(t)\bigr)^3,\label{eq:Nequiv}\\
  \bigl\|\partial_x\hat y(t,\cdot)\bigr\|\le\frac{z_b}{d_0}\,
  \bigl\|\partial_zy(t,\cdot)\bigr\|,\qquad
  \hat m_{\rm kin}(t)\ge\frac{d_0}{z_b}\,m_{\rm kin}(t).\label{eq:Nequiv2}
\end{gather}
\end{lemma}

\begin{proof}
From $d\le z_s\le z_b-d$ and $d_0\le z_s(0)\le z_b-d_0$,
\begin{equation}\label{eq:app_rhobound}
  \frac{d}{z_b}\le\rho_i(t)\le\frac{z_b}{d_0},\qquad i=1,2 .
\end{equation}
Let $m_{{\rm kin},1}$ and $m_{{\rm kin},2}$ be the two terms of $m_{\rm kin}$
in \eqref{eq:margins}, so that $m_{\rm kin}=\min_im_{{\rm kin},i}$. By
\eqref{eq:app_diffeo}, \eqref{eq:app_lam1} and \eqref{eq:app_minlam},
\begin{equation}\label{eq:app_scal}
  \partial_x\hat y^{(i)}=\rho_i\,\partial_zy,\qquad
  \partial_x^2v_s^{(1)}=\rho_1^2\,\partial_z^2v_s,\qquad
  \hat m_{\rm kin}=\min_{i=1,2}\rho_i^{-1}m_{{\rm kin},i}.
\end{equation}
With \eqref{eq:app_rhobound}, \eqref{eq:app_scal} gives \eqref{eq:Nequiv2} and
\begin{equation}\label{eq:app_back}
  \|\partial_zy\|\le\frac{z_b}{d}\,\|\partial_x\hat y\|,\qquad
  \|\partial_z^2v_s\|\le\frac{z_b^2}{d^2}\,\|\partial_x^2v_s^{(1)}\|,\qquad
  \frac1{m_{\rm kin}}\le\frac{z_b}{d}\,\frac1{\hat m_{\rm kin}} .
\end{equation}
The terms of $\widehat{\Nc}$ and $\Nc$ coincide except the gradients, the
second derivative and the kinematic margin, which satisfy, by
\eqref{eq:Nequiv2}, \eqref{eq:app_scal} and \eqref{eq:app_rhobound},
$\|\partial_x\hat y\|\le\frac{z_b}{d_0}\|\partial_zy\|$,
$\|\partial_x^2v_s^{(1)}\|\le\frac{z_b^2}{d_0^2}\|\partial_z^2v_s\|$ and
$\hat m_{\rm kin}^{-1}\le\frac{z_b}{d_0}m_{\rm kin}^{-1}$. Since
$d_0\le z_b/2$, $z_b/d_0\le z_b^2/d_0^2$, which gives the first inequality of
\eqref{eq:Nequiv}. Conversely, $1/d$ is a term of $\widehat{\Nc}$, so
$z_b/d\le z_b\widehat{\Nc}$, and \eqref{eq:app_back} gives
\begin{equation*}
  \Nc\le\widehat{\Nc}+z_b\widehat{\Nc}\cdot\widehat{\Nc}
  +z_b\widehat{\Nc}\cdot\widehat{\Nc}+z_b^2\widehat{\Nc}^{\,2}\cdot\widehat{\Nc}
  =\widehat{\Nc}+2z_b\widehat{\Nc}^{\,2}+z_b^2\widehat{\Nc}^{\,3}
  \le(1+z_b)^2\bigl(1+\widehat{\Nc}\bigr)^3 .
\end{equation*}
\end{proof}

By \eqref{eq:Nequiv}, $\limsup_{t\to T}\Nc(t)=\infty$ if and only if
$\limsup_{t\to T}\widehat{\Nc}(t)=\infty$, and by \eqref{eq:Nequiv2},
\begin{equation}\label{eq:app_reduce}
  \bigl\|\partial_x\hat y(t_n)\bigr\|\to\infty\ \Rightarrow\
  \bigl\|\partial_zy(t_n)\bigr\|\to\infty,\qquad
  \hat m_{\rm kin}(t_n)\to0\ \Rightarrow\ m_{\rm kin}(t_n)\to0 .
\end{equation}
It is therefore enough to prove \Cref{thm:wp} with $\widehat{\Nc}$,
$\hat m_{\rm kin}$ and $\partial_x\hat y$ in place of $\Nc$, $m_{\rm kin}$ and
$\partial_zy$.

\subsection{Boundary conditions}\label{ssec:app_bc}

In \Cref{sssec:app_h}, we solve the interface condition $H=0$ explicitly
for the incoming Riemann invariant. In \Cref{sssec:app_bccount}, we check
that the boundary value problem on each zone is well posed, and write the
conditions at $x=z_b$ and $x=z_s(0)^+$ as explicit expressions of the
incoming Riemann invariant in terms of known or outgoing quantities.

On $\Om_{2,0}$ we use the Riemann invariants
\begin{equation}\label{eq:app_RI}
  R_1=v_s-\Psi(\epss),\qquad R_2=v_s+\Psi(\epss),\qquad
  \Psi(\eps)\triangleq\int_{\eps_c}^{\eps}\frac{\sqrt{\sigma_e'(\eta)}}{\eta}\dd\eta ,
\end{equation}
and write $R_i^+\triangleq R_i(t,z_s(0)^+)$ and
$y^-\triangleq y^{(1)}(t,z_s(0)^-)$.

\subsubsection{The interface condition}\label{sssec:app_h}

The interface condition $H=0$ is a single equation for the traces. We
solve it explicitly for the incoming invariant $R_2^+$, as a function of
$a\triangleq(R_1^+,y^-)=(R_1^+,\epss^-,f_s^-)\in\mathbb{R}^3$, and bound the
derivatives of this function in terms of $M$.

For a vector $\xi$, a scalar function $g$ and a matrix $P$, we use the norms
\begin{equation*}
  |\xi|_\infty\triangleq\max_j|\xi_j|,\qquad
  |\xi|_1\triangleq\sum_j|\xi_j|,\qquad
  |\nabla g|_\infty\triangleq\max_j|\partial_jg|,\qquad
  \|P\|\triangleq\max_i\sum_j|P_{ij}|,
\end{equation*}
so that $|P\xi|_\infty\le\|P\|\,|\xi|_\infty$. For
$0<\epss^-<\eps_c<\epss^+<1$, let
\begin{equation}\label{eq:app_JG}
  \mathcal{J}(\epss^+,\epss^-)\triangleq
  \sqrt{\sigma_e(\epss^+)\Bigl(\frac1{\epss^-}-\frac1{\epss^+}\Bigr)},\qquad
  \mathcal{G}_{\epss^-}(\epss^+)\triangleq\Psi(\epss^+)+\mathcal{J}(\epss^+,\epss^-).
\end{equation}
By \eqref{eq:sigmae}, $\sigma_e$ is increasing on $[\eps_c,1)$ with
$\sigma_e(\eps_c)=0$, and $\Psi'=c_s/\epss>0$, so $\mathcal{J}$ is nondecreasing in
$\epss^+$, $\mathcal{G}_{\epss^-}$ is increasing with
$\mathcal{G}_{\epss^-}'\ge\Psi'>0$, and $\mathcal{G}_{\epss^-}(\eps_c)=0$. Let
$\mathcal{G}_{\epss^-}(1^-)\in(0,+\infty]$ be its limit at $1$, $v_s^-\triangleq
f_s^-/\epss^-$,
\begin{equation}\label{eq:app_Dh}
  D_h\triangleq\Bigl\{a=(R_1^+,\epss^-,f_s^-):\ 0<\epss^-<\eps_c,\
  0<v_s^--R_1^+<\mathcal{G}_{\epss^-}(1^-)\Bigr\},
\end{equation}
Since $\sigma_e$ is increasing,
$\mathcal{G}_{\epss^-}(1^-)=\Psi(1^-)+\sqrt{\sigma_e(1^-)(1/\epss^--1)}$ is a
continuous function of $\epss^-\in(0,\eps_c)$ with values in $(0,+\infty]$,
so $D_h$ is open. For $a\in D_h$, let
\begin{equation}\label{eq:app_hexpl}
  \epss^+(a)\triangleq\mathcal{G}_{\epss^-}^{-1}\bigl(v_s^--R_1^+\bigr),\qquad
  h(a)\triangleq R_1^++2\Psi\bigl(\epss^+(a)\bigr),\qquad
  y(a)\triangleq\bigl(\mathfrak{r}^{-1}(R_1^+,h(a)),\,y^-\bigr),
\end{equation}
where $\mathfrak{r}:(\epss,f_s)\mapsto(R_1,R_2)$, so that $y(a)$ are the
traces obtained from $a$. For $M\ge1$, let $A_M$ be the set of the $a\in D_h$
whose traces $y(a)=(y^+,y^-)$ satisfy
\begin{equation}\label{eq:app_AM}
  |f_s^\pm|\le2M,\qquad
  \tfrac1{2M}\le\epss^-\le\eps_c-\tfrac1{2M},\qquad
  \eps_c+\tfrac1{2M}\le\epss^+\le1-\tfrac1{2M},
\end{equation}
and let $\Theta_h(M)\triangleq40M^5$.

\begin{lemma}[Interface function]\label{lem:app_h}
\begin{enumerate}[label=(\roman*),leftmargin=2em,itemsep=1pt]
\item The function $h$ is $C^1$ on the open set $D_h$. Let $y^\pm$ be traces
with $0<\epss^-<\eps_c<\epss^+<1$, and $a=(R_1^+,y^-)$. Then
\begin{equation*}
  H(y^+,y^-)=0\ \text{and}\ v_s^->\mathfrak{s}
  \quad\Longleftrightarrow\quad
  a\in D_h\ \text{and}\ R_2^+=h(a),
\end{equation*}
and, in this case, with $\mathcal{J}=\mathcal{J}(\epss^+,\epss^-)$,
\begin{equation}\label{eq:app_hug}
  v_s^+=v_s^--\mathcal{J},\qquad
  v_s^--\mathfrak{s}=\frac{\epss^+\mathcal{J}}{\jump{\epss}},\qquad
  v_s^++c_s^+-\mathfrak{s}=\frac{\epss^-\mathcal{J}}{\jump{\epss}}+c_s^+>0 .
\end{equation}
\item Let $M\ge1$. For $a\in A_M$, $|\nabla h(a)|_\infty\le\Theta_h(M)$, and,
for $a,a'\in A_M$, $|h(a)-h(a')|\le\Theta_h(M)\,|a-a'|_1$.
\end{enumerate}
\end{lemma}

\begin{proof}
For (i), we write the Rankine--Hugoniot relations in the frame of the shock;
they reduce to an explicit relation between $v_s^+$ and $v_s^-$, which, in
the Riemann invariants, is the scalar equation
$\mathcal{G}_{\epss^-}(\epss^+)=v_s^--R_1^+$. For (ii), we differentiate this
equation.

Since $\jump{\epss}>0$, $H=0$ is equivalent to $\jump{f_s}=\mathfrak{s}\jump{\epss}$
and $\jump{F}=\mathfrak{s}\jump{f_s}$. Let $w^\pm\triangleq v_s^\pm-\mathfrak{s}$.
Since $f_s=\epss v_s$, $F=\epss v_s^2+\sigma_e(\epss)$ by \eqref{eq:FSP} and
$\sigma_e(\epss^-)=0$, these two relations read
\begin{equation*}
  \epss^+w^+=\epss^-w^-\triangleq j,\qquad
  \epss^+(w^+)^2+\sigma_e(\epss^+)=\epss^-(w^-)^2 .
\end{equation*}
Substituting $w^\pm=j/\epss^\pm$ in the second relation gives
$j^2(\frac1{\epss^-}-\frac1{\epss^+})=\sigma_e(\epss^+)$, that is,
$|j|=\epss^+\epss^-\mathcal{J}/\jump{\epss}$, and $v_s^->\mathfrak{s}$ means $j>0$.
Then $v_s^+-v_s^-=w^+-w^-=-j\jump{\epss}/(\epss^+\epss^-)=-\mathcal{J}$,
$v_s^--\mathfrak{s}=j/\epss^-$ and $v_s^+-\mathfrak{s}=j/\epss^+$, which give
\eqref{eq:app_hug}. Conversely, if $v_s^+=v_s^--\mathcal{J}$, let
$j\triangleq\epss^+\epss^-\mathcal{J}/\jump{\epss}>0$ and
$s'\triangleq v_s^--j/\epss^-$. Then $v_s^\pm-s'=j/\epss^\pm$, so the two
relations hold with $s'$ in place of $\mathfrak{s}$, and the first one gives
$s'=\jump{f_s}/\jump{\epss}=\mathfrak{s}$. Hence
$H=0$ and $v_s^->\mathfrak{s}$ if and only if $v_s^+=v_s^--\mathcal{J}$. Since
$v_s^+=R_1^++\Psi(\epss^+)$ by \eqref{eq:app_RI}, this is
$\mathcal{G}_{\epss^-}(\epss^+)=v_s^--R_1^+$. As $\mathcal{G}_{\epss^-}$ is an
increasing bijection of $(\eps_c,1)$ onto $(0,\mathcal{G}_{\epss^-}(1^-))$, this
holds if and only if $a\in D_h$ and $\epss^+=\epss^+(a)$, and then
$R_2^+=R_1^++2\Psi(\epss^+)=h(a)$. Finally, $\sigma_e\in C^2([\eps_c,1))$ is
positive on $(\eps_c,1)$, so $\mathcal{J}$ and $\Psi$ are $C^2$ there, and, since
$\mathcal{G}_{\epss^-}'>0$, the implicit function theorem shows that
$a\mapsto\epss^+(a)$, hence $h$, is $C^1$ on $D_h$.

For (ii), let $a\in A_M$ and $\nu\triangleq\Psi'/\mathcal{G}_{\epss^-}'\in(0,1]$ at
$\epss^+(a)$. Differentiating $\mathcal{G}_{\epss^-}(\epss^+)=v_s^--R_1^+$ and
$h=R_1^++2\Psi(\epss^+)$, in the variables $(R_1^+,\epss^-,v_s^-)$,
\begin{equation*}
  \partial_{R_1^+}h=1-2\nu,\qquad
  \partial_{v_s^-}h=2\nu,\qquad
  \partial_{\epss^-}h=-2\nu\,\partial_{\epss^-}\mathcal{J},\qquad
  \partial_{\epss^-}\mathcal{J}=-\frac{\epss^+\mathcal{J}}{2\epss^-\jump{\epss}} .
\end{equation*}
By \eqref{eq:app_AM}, $\epss^-\ge\frac1{2M}$, $\jump{\epss}\ge\frac1M$ and
$|v_s^\pm|=|f_s^\pm|/\epss^\pm\le4M^2$, so $\mathcal{J}=v_s^--v_s^+\le8M^2$ and
$|\partial_{\epss^-}\mathcal{J}|\le8M^4$. Since $v_s^-=f_s^-/\epss^-$, with
$|\partial_{f_s^-}v_s^-|\le2M$ and $|\partial_{\epss^-}v_s^-|\le8M^3$, we get,
in the variables $a=(R_1^+,\epss^-,f_s^-)$,
\begin{equation*}
  |\partial_{R_1^+}h|\le1,\qquad
  |\partial_{f_s^-}h|\le4M,\qquad
  |\partial_{\epss^-}h|\le16M^4+16M^3,
\end{equation*}
so $|\nabla h|_\infty\le32M^4\le\Theta_h$. The set $A_M$ need not be convex,
so we prove the Lipschitz bound directly. Let $a^k\in A_M$, $k=1,2$, with
$\epss^{+,k}\triangleq\epss^+(a^k)$ and $\mathsf g^k\triangleq v_s^{-,k}-R_1^{+,k}$. Since
$\mathcal{J}(\cdot,\epss^{-,1})$ is nondecreasing,
\begin{equation*}
  \bigl|\Psi(\epss^{+,1})-\Psi(\epss^{+,2})\bigr|
  \le\bigl|\mathcal{G}_{\epss^{-,1}}(\epss^{+,1})-\mathcal{G}_{\epss^{-,1}}(\epss^{+,2})\bigr|
  =\bigl|\mathsf g^1-\mathsf g^2-\mathcal{J}(\epss^{+,2},\epss^{-,1})+\mathcal{J}(\epss^{+,2},\epss^{-,2})\bigr| .
\end{equation*}
For $\epss^-$ between $\epss^{-,1}$ and $\epss^{-,2}$, the identity
$\mathcal{J}(\epss^+,\epss^-)^2\,\epss^+\epss^-/(\epss^+-\epss^-)=\sigma_e(\epss^+)$
and \eqref{eq:app_AM} give
\begin{equation*}
  \mathcal{J}(\epss^{+,2},\epss^-)^2
  =\mathcal{J}(\epss^{+,2},\epss^{-,2})^2\,
  \frac{\epss^{-,2}}{\epss^-}\,\frac{\epss^{+,2}-\epss^-}{\epss^{+,2}-\epss^{-,2}}
  \le(8M^2)^2\cdot2M\cdot M,
\end{equation*}
so, since $\epss^-\ge\frac1{2M}$ and $\epss^{+,2}-\epss^-\ge\frac1M$,
$|\partial_{\epss^-}\mathcal{J}(\epss^{+,2},\epss^-)|
=\frac{\epss^{+,2}\mathcal{J}}{2\epss^-(\epss^{+,2}-\epss^-)}\le8\sqrt2M^3\cdot M^2$, and
$|\mathcal{J}(\epss^{+,2},\epss^{-,1})-\mathcal{J}(\epss^{+,2},\epss^{-,2})|\le8\sqrt2M^5|\epss^{-,1}-\epss^{-,2}|$.
Since $\{\frac1{2M}\le\epss^-,\ |f_s^-|\le2M\}$ is convex,
$|\mathsf g^1-\mathsf g^2|\le|R_1^{+,1}-R_1^{+,2}|+8M^3|\epss^{-,1}-\epss^{-,2}|+2M|f_s^{-,1}-f_s^{-,2}|$.
Hence
\begin{equation*}
\begin{aligned}
  |h(a^1)-h(a^2)|
  &\le3|R_1^{+,1}-R_1^{+,2}|+4M|f_s^{-,1}-f_s^{-,2}|
  +\bigl(16M^3+16\sqrt2M^5\bigr)|\epss^{-,1}-\epss^{-,2}|\\
  &\le40M^5|a^1-a^2|_1 .
\end{aligned}
\end{equation*}
\end{proof}

The last identity of \eqref{eq:app_hug} shows that, on the branch selected
by $v_s^->\mathfrak{s}$, the interface is automatically transverse to the
fast compression wave: $v_s^++c_s^+-\mathfrak{s}>c_s^+$.

\subsubsection{Number and form of the boundary conditions}\label{sssec:app_bccount}

Let $B\triangleq f_s-\epss v_m$.

\begin{lemma}[Boundary conditions]\label{lem:app_bc}
Along an admissible solution, the speeds at the four boundaries and the
conditions imposed there are
\begin{equation}\label{eq:app_count}
\begin{array}{llll}
  x=0: & \hat\lambda^{(1)}>0, & \text{2 incoming}, &
     \text{conditions \eqref{eq:BCtop}};\\
  x=z_s(0)^-: & \hat\lambda^{(1)}>0, & \text{0 incoming}, &
     \text{no condition};\\
  x=z_s(0)^+: & \hat\lambda^{(2)}_1<0<\hat\lambda^{(2)}_2, & \text{1 incoming}, &
     H=0;\\
  x=z_b: & \hat\lambda^{(2)}_1<0<\hat\lambda^{(2)}_2, & \text{1 incoming}, &
     B=0,
\end{array}
\end{equation}
and each condition is transverse to the incoming eigenvectors. On
$\Om_{2,0}$, the equations read
\begin{equation}\label{eq:app_Rsys}
  \partial_tR_i+\hat\lambda^{(2)}_i\partial_xR_i=S_{\rm R}\triangleq
  K-\mathcal{R}(\epss)(v_s-v_m),\qquad i=1,2 .
\end{equation}
Along an admissible solution, the conditions $B=0$ and $H=0$ are
equivalent to
\begin{equation}\label{eq:app_bdata}
  R_1(t,z_b)=2v_m(t)-R_2(t,z_b),\qquad
    R_2\bigl(t,z_s(0)^+\bigr)=h\bigl(R_1(t,z_s(0)^+),\,y^{(1)}(t,z_s(0)^-)\bigr),
\end{equation}
with $h$ the function \eqref{eq:app_hexpl} of \Cref{lem:app_h}.
\end{lemma}

\begin{proof}
We first check the signs of the speeds and the transversality of the
conditions, then write the equations on $\Om_{2,0}$ in Riemann invariants,
and finally solve the conditions at $x=z_b$ and $x=z_s(0)^+$ for the incoming
invariant.

The signs of the speeds in \eqref{eq:app_count} are the inequalities of
\Cref{def:admissible}. Transversality is the uniform Kreiss--Lopatinskii
condition of \cite[Ch.~2 and 5]{li1985boundary} and
\cite[Ch.~4]{benzoni2007multi}. At $x=0$, \eqref{eq:BCtop} prescribes
$y^{(1)}(t,0)$, so the Jacobian of the condition with respect to $y^{(1)}$ is
$\mathbf{I}_2$, transverse to the whole incoming space $\mathbb{R}^2$,
although $\hat A^{(1)}$ is not diagonalizable. On $\Om_{2,0}$, by
\eqref{eq:jac}, the eigenvectors of $\hat\lambda^{(2)}_1$ and
$\hat\lambda^{(2)}_2$ are
\begin{equation}\label{eq:app_eigvec}
  \mathbf{r}_1=(1,\,v_s-c_s)^\top,\qquad \mathbf{r}_2=(1,\,v_s+c_s)^\top .
\end{equation}
At $x=z_b$, $\nabla_yB=(-v_m,1)$ and, on $B=0$ where $v_s=v_m$,
\begin{equation}\label{eq:app_KLbot}
  \nabla_yB\cdot\mathbf{r}_1=v_s-c_s-v_m=-c_s\neq0 .
\end{equation}
At $x=z_s(0)^+$, by \eqref{eq:FSP},
$\nabla_{y^+}H=\bigl(-\jump{F}-\jump{\epss}(c_s^2-v_s^2),\
2\jump{f_s}-2\jump{\epss}v_s\bigr)$, and on $H=0$,
$\jump{f_s}=\mathfrak{s}\jump{\epss}$ and $\jump{F}=\mathfrak{s}^2\jump{\epss}$, so
\begin{equation}\label{eq:app_KLint}
  \nabla_{y^+}H\cdot\mathbf{r}_2
  =\jump{\epss}\bigl(-\mathfrak{s}^2-c_s^2+v_s^2+2(\mathfrak{s}-v_s)(v_s+c_s)\bigr)
  =-\jump{\epss}\bigl(v_s^{+}+c_s^{+}-\mathfrak{s}\bigr)^2 .
\end{equation}
Along an admissible solution $\mathfrak{s}=\dot z_s$, and at $x=z_s(0)^+$,
where $z=z_s$ and $v_{\rm rel}=v_s^+-\dot z_s$, \eqref{eq:app_lam2} gives
\begin{equation}\label{eq:app_KLpos}
  v_s^++c_s^+-\mathfrak{s}=\rho_2\,\hat\lambda^{(2)}_2\big|_{x=z_s(0)^+}
  \ge\rho_2\,\hat m_{\rm kin}>0,
  \qquad
  \bigl|\nabla_{y^+}H\cdot\mathbf{r}_2\bigr|\ge\jump{\epss}\,\rho_2^2\,\hat m_{\rm kin}^2 .
\end{equation}

On $\Om_{2,0}$, in the primitive variables, \eqref{eq:model} reads
$\partial_t\epss+v_s\partial_z\epss+\epss\partial_zv_s=0$ and
$\partial_tv_s+v_s\partial_zv_s+\frac{c_s^2}{\epss}\partial_z\epss=S_{\rm R}$.
Since $\Psi'=c_s/\eps$,
\begin{equation*}
  \partial_tR_2=\partial_tv_s+\frac{c_s}{\epss}\partial_t\epss
  =S_{\rm R}-(v_s+c_s)\Bigl(\partial_zv_s+\frac{c_s}{\epss}\partial_z\epss\Bigr)
  =S_{\rm R}-(v_s+c_s)\partial_zR_2,
\end{equation*}
and in the same way $\partial_tR_1+(v_s-c_s)\partial_zR_1=S_{\rm R}$; with
\eqref{eq:app_diffeo} and \eqref{eq:app_lam2}, this is \eqref{eq:app_Rsys}.
Since $\Psi'>0$ and $\sigma_e\in C^2$, $\mathfrak{r}:(\epss,f_s)\mapsto(R_1,R_2)$
is a $C^2$-diffeomorphism of $(\eps_c,1)\times\mathbb{R}$ onto its image, with
inverse $v_s=(R_1+R_2)/2$, $\epss=\Psi^{-1}((R_2-R_1)/2)$, and
\begin{equation}\label{eq:app_dyR}
  \frac{\partial y}{\partial R_2}=\frac{\epss}{2c_s}\,\mathbf{r}_2,\qquad
  \frac{\partial y}{\partial R_1}=-\frac{\epss}{2c_s}\,\mathbf{r}_1 .
\end{equation}
At $x=z_b$, $B=0$ reads $v_s=v_m$, that is, $R_1=2v_m-R_2$. At
$x=z_s(0)^+$, along an admissible solution, $\epss^-<\eps_c<\epss^+$ and,
by \eqref{eq:app_lam1}, $v_s^--\dot z_s=\rho_1\hat\lambda^{(1)}(t,z_s(0)^-)>0$
with $\dot z_s=\mathfrak{s}$; by \Cref{lem:app_h}(i), $H=0$ is then
equivalent to $R_2^+=h(R_1^+,y^-)$.
\end{proof}

\subsection{The frozen problem}\label{ssec:app_frozen}

We replace the unknown interface velocity by a given function $\zeta$. The
interface is then a known curve, and the problem on each zone can be solved
by characteristics; this subsection sets the notation and derives the
equations along the characteristics.

For $\tau,\delta,M'>0$ let
\begin{equation}\label{eq:app_K}
  \mathcal{Z}_{\tau,\delta,M'}=\Bigl\{\zeta\in C^0([0,\tau]):\
  \zeta(0)=\dot z_{s,0},\
  \|\zeta-\dot z_{s,0}\|_\tau\le\delta,\
  \mathrm{Lip}(\zeta)\le M'\Bigr\},
\end{equation}
a closed subset of $C^0([0,\tau])$. For $\zeta\in\mathcal{Z}_{\tau,\delta,M'}$,
let $z_\zeta(t)\triangleq z_s(0)+\int_0^t\zeta$. The frozen problem
is the first two lines of \eqref{eq:app_FBVP} on $\Om_{1,0}$ and
\eqref{eq:app_Rsys}--\eqref{eq:app_bdata} on $\Om_{2,0}$, with
$(z_s,\dot z_s)$ replaced by $(z_\zeta,\zeta)$. In the interface condition
$R_2^+=h(R_1^+,y^-)$ of \eqref{eq:app_bdata}, $h$ is the explicit function
of \Cref{lem:app_h}, defined on the open set $D_h$. At $t=0$,
\Cref{ass:wp}(iii) gives $H=0$ and \Cref{ass:wp}(ii) gives
$v_s^--\dot z_{s,0}=\hat\lambda^{(1)}(0,z_s(0)^-)>0$, so, by
\Cref{lem:app_h}(i), $(R_1^+,y^-)(0)\in D_h$ and
$R_2^+(0)=h\bigl((R_1^+,y^-)(0)\bigr)$. Near these traces, $H$ and
$R_2^+-h(R_1^+,y^-)$ vanish on the same set and have nonzero derivatives
with respect to $R_2^+$ (by \eqref{eq:app_KLint} and \eqref{eq:app_hug}), so
$H=\varphi\cdot(R_2^+-h(R_1^+,y^-))$ with $\varphi$ continuous and nonzero
(Hadamard's lemma), and \eqref{eq:compat_int} is also the first-order
compatibility condition of $R_2^+=h(R_1^+,y^-)$. The frozen problem is considered as long as
$(R_1^+,y^-)(t)\in D_h$. We denote its solution by $(y^{(1)}_\zeta,y^{(2)}_\zeta)$, and let
\begin{equation}\label{eq:app_T}
  \mathcal{T}(\zeta)\triangleq\mathfrak{s}\bigl(y^{(1)}_\zeta,y^{(2)}_\zeta\bigr)
  =\frac{f_s^{(2)}-f_s^{(1)}}{\epss^{(2)}-\epss^{(1)}}\Big|_{x=z_s(0)} .
\end{equation}
If $\zeta=\mathcal{T}(\zeta)$, then $\dot z_\zeta=\mathfrak{s}$, and
$(y^{(1)}_\zeta,y^{(2)}_\zeta,z_\zeta)$ solves \eqref{eq:app_FBVP}, since
$R_2^+=h(R_1^+,y^-)$ gives $H=0$ at the interface (\Cref{lem:app_h}(i)) and
$B=0$ is equivalent to the condition at $x=z_b$ (\Cref{lem:app_bc}). Recall that $d_0=\min\{z_s(0),z_b-z_s(0)\}$. If
$\tau(|\dot z_{s,0}|+\delta)\le d_0/2$, then $|z_\zeta-z_s(0)|\le d_0/2$, and
\begin{equation}\label{eq:app_thetaomega}
\begin{gathered}
  \theta_\zeta\triangleq\frac{z_s(0)}{z_\zeta},\quad
  \theta^{(2)}_\zeta\triangleq\frac{z_b-z_s(0)}{z_b-z_\zeta},\qquad
  \theta_\zeta,\ \theta^{(2)}_\zeta\in\bigl[\tfrac23,2\bigr],\\
  \omega_\zeta\triangleq\frac{\zeta}{z_\zeta},\quad
  \omega^{(2)}_\zeta\triangleq\frac{\zeta}{z_b-z_\zeta},\qquad
  |\omega_\zeta|,|\omega^{(2)}_\zeta|\le\frac{2(|\dot z_{s,0}|+\delta)}{d_0}.
\end{gathered}
\end{equation}
In these variables $g_1=x\,\omega_\zeta$, $g_2=(z_b-x)\,\omega^{(2)}_\zeta$,
$\rho_1^{-1}=\theta_\zeta$ and $\rho_2^{-1}=\theta^{(2)}_\zeta$. Let
$\Gamma_{in}\triangleq(\{0\}\times[0,z_s(0)])\cup([0,\tau]\times\{0\})$ be the
inflow boundary of $\Om_{1,0}$, which we parametrize by
$\alpha\in[-\tau,z_s(0)]$: $\alpha\ge0$ stands for the point $(0,\alpha)$ and
$\alpha<0$ for $(-\alpha,0)$. For a characteristic $X$, denote by a bar the values along $X$. As long as the speeds have the signs of \eqref{eq:app_count}, a
characteristic followed backward in time reaches $\Gamma_{in}$ on
$\Om_{1,0}$ and, on $\Om_{2,0}$, either the line $t=0$ or the boundary at
which its family is incoming ($x=z_b$ for $R_1$, $x=z_s(0)^+$ for $R_2$). We
call this point the starting point of $X$ and denote its time by $t_\alpha$, where
$\alpha$ labels the starting point (on $\Om_{1,0}$, $\alpha$ is the parameter of $\Gamma_{in}$ defined above). At the
starting point, the solution is given by the initial data or by the boundary
condition.

\begin{lemma}[Characteristic equations]\label{lem:app_char}
Let $(y^{(1)}_\zeta,y^{(2)}_\zeta)$ be a solution of the frozen problem with
$y^{(i)}_\zeta\in C^1$ and $\partial_xv_s^{(1)}\in C^1$. On $\Om_{1,0}$,
\begin{equation}\label{eq:app_prim}
  \partial_tv_s+\hat\lambda^{(1)}\partial_xv_s=K-\gamma(v_s-v_m),\qquad
  \partial_t\epss+\hat\lambda^{(1)}\partial_x\epss
  =-\theta_\zeta\epss\,\partial_xv_s,\qquad
  \hat\lambda^{(1)}=\theta_\zeta v_s-x\,\omega_\zeta ,
\end{equation}
and, along $\partial_tX=\hat\lambda^{(1)}(t,X)$, with
$p\triangleq\partial_xv_s$, $r\triangleq\partial_x^2v_s$ and
$q\triangleq\partial_x\epss$,
\begin{align}
  \frac{\dd}{\dd t}\bar v_s&=K-\gamma(\bar v_s-v_m),&
  \frac{\dd}{\dd t}\bar p&=-\theta_\zeta\bar p^{\,2}+(\omega_\zeta-\gamma)\bar p,&
  \frac{\dd}{\dd t}\bar r&=-3\theta_\zeta\bar p\,\bar r+(2\omega_\zeta-\gamma)\bar r,
  \label{eq:app_riccati}\\
  \frac{\dd}{\dd t}\bar\epss&=-\theta_\zeta\bar\epss\,\bar p,&
  \frac{\dd}{\dd t}\bar q&=-2\theta_\zeta\bar p\,\bar q+\omega_\zeta\bar q
  -\theta_\zeta\bar\epss\,\bar r .&&
  \label{eq:app_epsq}
\end{align}
At $x=0$, with $p_0(t)\triangleq p(t,0)$ and
$\hat\lambda^{(1)}(t,0)=\theta_\zeta(v_m+K/\gamma)$,
\begin{gather}
  v_s=v_m+\frac{K}{\gamma},\qquad
  p_0=-\frac{\dot v_m}{\hat\lambda^{(1)}(t,0)},\qquad
  \epss=\epsin,\qquad
  q=-\frac{\dot\epsin+\theta_\zeta\epsin p_0}{\hat\lambda^{(1)}(t,0)},
  \label{eq:app_in0}\\
  r=-\frac{\dot p_0+(\theta_\zeta p_0-\omega_\zeta+\gamma)p_0}
  {\hat\lambda^{(1)}(t,0)},\qquad
  \dot p_0=-\frac{\ddot v_m+p_0\,\partial_t\hat\lambda^{(1)}(t,0)}
  {\hat\lambda^{(1)}(t,0)} .
  \label{eq:app_r0}
\end{gather}
On $\Om_{2,0}$, along $\partial_tX=\hat\lambda^{(2)}_i(t,X)$, the gradients
$P_i\triangleq\partial_xR_i$ satisfy
\begin{equation}\label{eq:app_dxR}
  \frac{\dd}{\dd t}\bar P_i
  =-\Bigl(\partial_x\hat\lambda^{(2)}_i+\sum_{j=1,2}\partial_{R_j}\hat\lambda^{(2)}_i\,\bar P_j\Bigr)\bar P_i
  +\sum_{j=1,2}\partial_{R_j}S_{\rm R}\,\bar P_j,\qquad
  \partial_tR_i=S_{\rm R}-\hat\lambda^{(2)}_iP_i .
\end{equation}
\end{lemma}

\begin{proof}
On $\Om_{1,0}$, $\sigma_e=0$ and $\mathcal{R}=\gamma$, so the first line of
\eqref{eq:app_FBVP} in the variables $(v_s,\epss)$ is \eqref{eq:app_prim}.
Since $\partial_x\hat\lambda^{(1)}=\theta_\zeta p-\omega_\zeta$,
differentiating \eqref{eq:app_prim} once and twice in $x$ gives
\eqref{eq:app_riccati}--\eqref{eq:app_epsq}. At $x=0$, $g_1=0$, and
\eqref{eq:BCtop} gives $v_s$ and $\epss$; the first equation of
\eqref{eq:app_prim} at $x=0$, where its right-hand side vanishes, gives $p_0$,
and the second one gives $q$. Differentiating
$\hat\lambda^{(1)}(t,0)\,p_0=-\dot v_m$ in time gives $\dot p_0$, and the
equation for $p$ at $x=0$ gives $r$. On $\Om_{2,0}$, differentiating
\eqref{eq:app_Rsys} in $x$ gives \eqref{eq:app_dxR}.
\end{proof}

The equation for the gradient $q$ of $\epss$ contains $r=\partial_x^2v_s$, which is why
\Cref{def:admissible} and \Cref{ass:wp} ask for $\partial_x^2v_s\in C^0$. For
$\zeta\equiv0$ ($\theta_\zeta=1$, $\omega_\zeta=0$),
\eqref{eq:app_riccati}--\eqref{eq:app_epsq} integrate to
$\bar p(t)=\gamma\bar p(0)e^{-\gamma t}/(\gamma+\bar p(0)(1-e^{-\gamma t}))$ and to
\eqref{eq:derivloss}. By \Cref{def:compat}, the values of
$(v_s,p,r,\epss,q)$ given by $y_0$ on $\{0\}\times[0,z_s(0)]$ and by
\eqref{eq:app_in0}--\eqref{eq:app_r0} on $[0,\tau]\times\{0\}$ coincide at
the corner $(0,0)$.

\subsection{Existence and bounds for the frozen problem}\label{ssec:app_bounds}

We show that the frozen problem has a solution on a time interval
$[0,\tau_0(M)]$ whose length depends only on a bound $M$ of the margin
function at $t=0$ (\Cref{lem:app_frozen}), and that its traces depend on
$\zeta$ in a Lipschitz way (\Cref{lem:app_lip}).

The proof of \Cref{lem:app_frozen} uses a continuity argument: we assume
that some bounds hold up to the first time $t^\dagger$ at which one of them
fails, we show that the solution then satisfies stronger bounds, and we
conclude that $t^\dagger$ cannot come before the end of the time interval.
Here, at $t=0$, every margin is at least $\frac1M$ and the gradients are at most
$M$. As long as the margins stay above $\frac1{2M}$ and the gradients below
$2\widetilde M$, the coefficients of the equations and the right-hand sides of
the ordinary differential equations satisfied by the gradients along the
characteristics are bounded. The gradients then stay close to their values at
the starting points, and the solution changes at a bounded rate. On a short enough
time interval, the solution therefore loses less than half of its margins, and
its gradients stay below $\widetilde M$. To make this
quantitative, we first fix the set in which the solution is to stay, and then
define three groups of constants: bounds on the coefficients, bounds on the
gradients at the starting points, and a bound on the rate of change of the solution, which gives
$\tau_0(M)$.

The subsection is organized accordingly. \Cref{sssec:app_set} defines the
set in which the frozen solution is to stay, \Cref{sssec:app_const} the three
groups of constants, \Cref{sssec:app_exist} contains the continuity argument
(\Cref{lem:app_frozen}), and \Cref{sssec:app_lipsec} proves the Lipschitz
dependence on $\zeta$ (\Cref{lem:app_lip}).

\subsubsection{The set in which the frozen solution stays}\label{sssec:app_set}

Fix $M\ge\widehat{\Nc}(0)$ and let $\delta\triangleq1/(8M)$. By
\eqref{eq:app_margin}, $d_0\ge\frac1M$ and
$|\dot z_{s,0}|\le|\jump{f_{s,0}}|/\jump{\eps_{s,0}}\le M^2$. Let
$\mathcal{S}_M$ be the set of triples $(\hat y,z_\zeta,\zeta)$, with
$\hat y=(\hat y^{(1)},\hat y^{(2)})$ continuous on $\overline{\Om_{1,0}}$ and
$\overline{\Om_{2,0}}$, $z_\zeta\in(0,z_b)$ and $\zeta\in\mathbb{R}$, such
that, with $\rho_i$, $\hat m_{\rm kin}$ and $\mathrm{dist}$ computed with
$(z_\zeta,\zeta)$,
\begin{equation}\label{eq:app_SM}
\begin{gathered}
    \|\hat y^{(1)}\|_{C^0}+\|\hat y^{(2)}\|_{C^0}\le M,\quad
  \jump{\epss}\ge\tfrac1M,\quad
  \tfrac1M\le\epss\le1-\tfrac1M,\\
  \min\{\hat m_{\rm kin},m_{\rm phase},\mathrm{dist}\}\ge\tfrac1M,\quad
  \tfrac1{Mz_b}\le\rho_i\le Mz_b,\quad|\zeta|\le8M^2+1 ,
\end{gathered}
\end{equation}
and $\mathcal{S}_M^{G}$ the subset where, in addition,
$\|\partial_x\hat y^{(1)}\|+\|\partial_x\hat y^{(2)}\|+\|\partial_x^2v_s^{(1)}\|\le G$.
Then $(\hat y(0),z_s(0),\dot z_{s,0})\in\mathcal{S}_M^{M}$, and the
interface traces of every element of $\mathcal{S}_{2M}$ satisfy
\eqref{eq:app_AM}: if they are of the form $y(a)$ with $a\in D_h$, then
$a\in A_M$. The frozen
solutions will be shown to stay in $\mathcal{S}_{2M}^{2\widetilde M}$, where
the margins are half as large as at $t=0$. All the constants below depend
only on $M$ and on $K$, $\gamma$, $z_b$, the closures, $\|v_m\|_{C^2}$ and
$\|\epsin\|_{C^1}$.

\subsubsection{Constants}\label{sssec:app_const}

We start with the coefficients of the equations, which we bound together
with their derivatives. On $\mathcal{S}_{2M}$,
$\epss\in[\frac1{2M},1-\frac1{2M}]$, $|f_s|\le2M$, $|v_s|\le4M^2$ and
$\rho_i^{-1}\le2Mz_b$. On $[\eps_c,1-\frac1{2M}]$, $\sigma_e''$ and
$\mathcal{R}'$ are continuous and $\sigma_e'\ge\min_{[\eps_c,1-1/(2M)]}\sigma_e'>0$
by \eqref{eq:sigmae}. Hence \eqref{eq:jac} and \eqref{eq:FSP} make
\begin{equation}\label{eq:app_a1}
\begin{aligned}
  a_1(M)\triangleq\sup\Bigl(&\sum_{i=1,2}
  \bigl(\|\hat A^{(i)}\|+\|\hat S^{(i)}\|+\|\partial_{(x,y)}\hat A^{(i)}\|
  +\|\partial_{(x,y)}\hat S^{(i)}\|\bigr)
  +\sum_{\hat\lambda}\|\partial_{(x,y)}\hat\lambda\|\\
  &+\sum_{i=1,2}\|\partial_{(x,R,\theta,\omega)}\hat\lambda^{(2)}_i\|
  +|S_{\rm R}|+\|\partial_RS_{\rm R}\|
  +\|\partial\mathfrak{r}\|+\|\partial\mathfrak{r}^{-1}\|\Bigr)
\end{aligned}
\end{equation}
finite, the supremum being over $\mathcal{S}_{2M}$ and $x\in[0,z_b]$, and the
second sum over the three speeds. Operator norms are taken for the max norm
on vectors. The terms in $R$ are taken over the parallelogram
\begin{equation*}
  \Pi_M\triangleq\Bigl\{R:\ \eps_c+\tfrac1{2M}\le\Psi^{-1}\bigl(\tfrac{R_2-R_1}2\bigr)\le1-\tfrac1{2M},\
  |R_1+R_2|\le8M^2\Bigr\},
\end{equation*}
which is convex and contains the values of $R$ on $\Om_{2,0}$ for
$\mathcal{S}_{2M}$, and $(\theta,\omega)$ ranges over
$[\frac1{2Mz_b},2Mz_b]\times[-2M(32M^2+1),2M(32M^2+1)]$; so $a_1$ is also a
Lipschitz constant of $\hat\lambda^{(2)}_i$ in $(x,R,\theta,\omega)$, and of
$S_{\rm R}$ and $\mathfrak{r}^{-1}$ in $R$. Let
\begin{equation*}
  \mathsf b(M)\triangleq\bigl(1+2M(1+4M^2)\bigr)\bigl(1+a_1(M)\bigr).
\end{equation*}
The gradients are estimated along the characteristics in the variables
$(p,q)$ on $\Om_{1,0}$ and $P=\partial_xR$ on $\Om_{2,0}$, whereas
$\widehat{\Nc}$ involves $\partial_x\hat y$. On $\mathcal{S}_{2M}$, $\mathsf
b(M)$ bounds the norms of the four linear maps that pass from one to the
other,
\begin{equation}\label{eq:app_conv}
\begin{gathered}
  (p,q)\mapsto\partial_x\hat y^{(1)}=(q,\,\epss p+v_sq),\qquad
  \partial_x\hat y^{(1)}\mapsto(p,q)=\Bigl(\frac{\partial_xf_s-v_s\partial_x\epss}{\epss},\,\partial_x\epss\Bigr),\\
  P\mapsto\partial_x\hat y^{(2)}=\partial\mathfrak{r}^{-1}P,\qquad
  \partial_x\hat y^{(2)}\mapsto P=\partial\mathfrak{r}\,\partial_x\hat y^{(2)} .
\end{gathered}
\end{equation}

The second group bounds the gradients at the starting points. These values are
given by the initial data on the line $t=0$, by
\eqref{eq:app_in0}--\eqref{eq:app_r0} at $x=0$, and, at $x=z_b$ and
$x=z_s(0)^+$, by differentiating the boundary conditions
\eqref{eq:app_bdata} in time. In the last case they involve the time
derivatives of the outgoing traces, which are bounded by $a_2$ below. The
constant $G_0$ bounds all the values at the starting points (this is checked in the proof of
\Cref{lem:app_frozen}), and $\widetilde M$ bounds the resulting gradients of
$\hat y$, with a safety factor. Let $C_r(M)$ be the supremum, over $\mathcal{S}_{2M}$ and
$|\zeta-\dot z_{s,0}|\le\delta$, of the right-hand side of \eqref{eq:app_r0};
there, $\partial_t\hat\lambda^{(1)}(t,0)=\dot\theta_\zeta(v_m+K/\gamma)+\theta_\zeta\dot v_m$
with $\dot\theta_\zeta=-z_s(0)\zeta/z_\zeta^2$, so this right-hand side
involves $\zeta$, $z_\zeta$ and $v_m,\dot v_m,\ddot v_m$, but not
$\dot\zeta$. Let
\begin{gather*}
  M_0(M)\triangleq\mathsf b(M)\,M,\qquad
  a_2(M)\triangleq(1+a_1)\,a_1\bigl(1+2\mathsf b(M)(M_0+1)\bigr),\\
  G_0(M)\triangleq\max\Bigl\{M_0,\ 2M\|v_m\|_{C^1},\ 2M\bigl(\|\epsin\|_{C^1}+4M\|v_m\|_{C^1}\bigr),\
  C_r,\\ 2M\bigl(a_1+2\|v_m\|_{C^1}+a_2\bigr),\ 2M\bigl(a_1+3\Theta_h\,a_2\bigr)\Bigr\},\\
  \widetilde M(M)\triangleq6\,\mathsf b(M)\bigl(G_0(M)+1\bigr)\ \ge12M .
\end{gather*}

The last group bounds the rate at which the solution changes. The speeds depend on $\zeta$, whose only time regularity is
$\mathrm{Lip}(\zeta)\le M'$, and the constants must not depend on $M'$ (this
is used for uniqueness in \Cref{prop:app_local}). We therefore separate the
part of the speeds that depends on $\zeta-\dot z_{s,0}$, which is small since
$|\zeta-\dot z_{s,0}|\le\delta$, from a part $\hat\lambda^\flat$ that is $C^1$
in time, with a time derivative bounded independently of $M'$:
\begin{equation*}
  \hat\lambda^{(1)}=\hat\lambda^{(1)\flat}-\frac{x(\zeta-\dot z_{s,0})}{z_\zeta},\qquad
  \hat\lambda^{(2)}_{1,2}=\hat\lambda^{(2)\flat}_{1,2}-\frac{(z_b-x)(\zeta-\dot z_{s,0})}{z_b-z_\zeta},
\end{equation*}
with $\hat\lambda^{(1)\flat}\triangleq\theta_\zeta v_s-x\dot z_{s,0}/z_\zeta$ and
$\hat\lambda^{(2)\flat}_{1,2}\triangleq\theta^{(2)}_\zeta(v_s\mp c_s)-(z_b-x)\dot z_{s,0}/(z_b-z_\zeta)$.
Let $\mathsf c(M)$ be the supremum, over $\mathcal{S}^{2\widetilde M}_{2M}$
and $x\in[0,z_b]$, of
\begin{equation*}
  |\partial_t\hat y|,\ |\partial_tR|,\ |\zeta|,\ |\dot\rho_i|,\
  |\partial_t\hat\lambda^{\flat}|,
\end{equation*}
and of the right-hand sides of \eqref{eq:app_riccati}--\eqref{eq:app_epsq}
and \eqref{eq:app_dxR}, with $\partial_t\hat y=-\hat A\partial_x\hat y+\hat S$;
it is finite by \eqref{eq:app_a1} and \eqref{eq:app_conv}. Hence, in a time $t$, $\hat y$, $R$, $z_\zeta$, $\rho_i$ and
$\hat\lambda^\flat$ change by at most $\mathsf c\,t$ at fixed $x$, and
$p,q,r,P_i$ by at most $\mathsf c\,t$ along a characteristic. The time
$\tau_1$ below makes this change at most $\frac1{4\widetilde M}$, a small
fraction of the margins. In $\tau_0$, we add the
condition $\tau\le1$, used in the Gr\"onwall estimates, and two conditions:
the first keeps the interface away from the boundaries, and the second makes
the backward characteristics through the outgoing traces reach $t=0$, see \eqref{eq:app_tr} below. Let
\begin{equation}\label{eq:app_tau0}
  \tau_1\triangleq\frac{1}{4\widetilde M\,\mathsf c},\qquad
  \tau_0(M)\triangleq\min\Bigl\{1,\ \frac{1}{2M(M^2+1)},\ \frac{1}{2Ma_1},\ \tau_1\Bigr\}.
\end{equation}
Since $d_0\ge\frac1M$ and $|\dot z_{s,0}|\le M^2$, $\tau\le\tau_0(M)$ gives
$\tau(|\dot z_{s,0}|+\delta)\le d_0/2$, hence \eqref{eq:app_thetaomega}, and,
as long as the frozen solution stays in $\mathcal{S}_{2M}$, where
$\sup|\hat\lambda|\le a_1(M)$,
\begin{equation}\label{eq:app_tr}
  \tau\sup|\hat\lambda|\le\frac{d_0}2<\min\{z_s(0),z_b-z_s(0)\}:
\end{equation}
the backward characteristics of the outgoing waves from
$(t,z_s(0)^\pm)$ and $(t,z_b)$, $t\le\tau$, reach $t=0$ without meeting a
boundary.

\subsubsection{Existence and bounds}\label{sssec:app_exist}

The next lemma states that the frozen problem is well posed on
$[0,\tau_0(M)]$ and that its solution stays in
$\mathcal{S}^{2\widetilde M}_{2M}$.

\begin{lemma}[Frozen problem]\label{lem:app_frozen}
Let $M\ge\widehat{\Nc}(0)$, $\delta=\frac1{8M}$, $\tau\le\tau_0(M)$, $M'>0$
and $\zeta\in\mathcal{Z}_{\tau,\delta,M'}$. Then:
\begin{enumerate}[label=(\roman*),leftmargin=2em,itemsep=1pt]
\item the frozen problem has a unique solution on $[0,\tau]$, with
$y^{(i)}_\zeta\in C^1$ and $\partial_xv_s^{(1)}\in C^1$; its arguments
$(R_1^+,y^-)(t)$ stay in $A_M$, and, for $t\le\tau$,
\begin{equation}\label{eq:app_boot}
\begin{gathered}
  \|\hat y\|\le M+\tfrac12,\qquad
  \text{each margin of \eqref{eq:app_SM}}\ge\tfrac1M-\tfrac1{4M}-\tfrac1{24M}
  >\tfrac1{2M},\\
  \|\partial_x\hat y^{(1)}\|+\|\partial_x\hat y^{(2)}\|+\|\partial_x^2v_s^{(1)}\|
  <\widetilde M,\qquad
  \rho_i\in\bigl[\tfrac34,\tfrac54\bigr],
\end{gathered}
\end{equation}
so that, with $|\zeta|\le M^2+\delta$, $(\hat y_\zeta(t),z_\zeta(t),\zeta(t))\in\mathcal{S}^{2\widetilde M}_{2M}$;
\item $|\partial_ty_\zeta(t,z_s(0)^\pm)|\le a_1(M)\bigl(1+2\widetilde M(M)\bigr)$ for $t\le\tau$.
\end{enumerate}
\end{lemma}

\begin{proof}
For (i), we construct the
solution by characteristics and show that it exists as long as its $C^1$ norm
stays bounded; we then bound the gradients at the starting points; and we conclude
by a continuity argument, which shows that the solution cannot leave
$\mathcal{S}^{2\widetilde M}_{2M}$ before $t=\tau$. Part (ii) follows from
(i).

We first construct the solution. On $\Om_{1,0}$, let $X(t;\alpha)$ be the characteristic that starts from the
point of $\Gamma_{in}$ with parameter $\alpha$.
The unknowns $(X,\bar v_s,\bar p,\bar r,\bar\epss,\bar q)$ solve the closed
system $\partial_tX=\theta_\zeta\bar v_s-\omega_\zeta X$,
\eqref{eq:app_riccati}--\eqref{eq:app_epsq}, with data on $\Gamma_{in}$,
continuous by \Cref{def:compat}; this solution exists as long as $\bar p$
stays bounded, and is defined for $\alpha\in[-t,z_s(0)]$ even after $X$ has
left $\overline{\Om_{1,0}}$. Moreover
\begin{equation}\label{eq:app_dalphaX}
  \frac{\dd}{\dd t}\partial_\alpha X=(\theta_\zeta\bar p-\omega_\zeta)\,\partial_\alpha X,\qquad
  \partial_\alpha X(t_\alpha;\alpha)\in\bigl\{1,\ \hat\lambda^{(1)}(t_\alpha,0)\bigr\},
\end{equation}
so $\partial_\alpha X>0$ while $\hat\lambda^{(1)}>0$: the characteristics
do not cross, move to the right, leave only through $x=z_s(0)$, and cover
$[0,t]\times\overline{\Om_{1,0}}$. Composing with the inverse of
$\alpha\mapsto X(t;\alpha)$ gives $y^{(1)}_\zeta=(\bar\epss,\bar\epss\bar v_s)$,
$\partial_xv_s=\bar p$, $\partial_x^2v_s=\bar r$, $\partial_x\epss=\bar q$, so
$y^{(1)}_\zeta\in C^1$ and $\partial_xv_s^{(1)}\in C^1$. On $\Om_{2,0}$, the
coefficients of \eqref{eq:app_Rsys} are $C^1$ in $(x,R)$ and continuous in
$t$, and the boundary data \eqref{eq:app_bdata}, in which $y^-$ is given by
$\Om_{1,0}$, are $C^1$ in $t$ and $C^1$-compatible at the corners
(\Cref{def:compat}). By \cite[Ch.~2]{li1985boundary}, there is a unique
$C^1$ solution on a time interval depending only on the $C^1$ norms of the
data and of the coefficients and on a lower bound of $|\hat\lambda^{(2)}_i|$
at the boundaries, which can be continued as long as its $C^1$ norm remains
bounded, $|\hat\lambda^{(2)}_i|\ge\frac1{2M}$ at the boundaries and
$(R_1^+,y^-)$ stays in a compact subset of $D_h$. The coefficients contain
$\zeta$, so they are Lipschitz in $t$ with a constant depending on $M'$;
but the construction of \cite{li1985boundary} is by characteristics and uses
\eqref{eq:app_dxR}, which differentiates the coefficients in $x$ and $R$
only, so the existence time and the bounds do not depend on $M'$.

Next, we bound the gradients at the starting points. Let $t^\dagger$ be the first
time at which $(\hat y_\zeta,z_\zeta,\zeta)$ leaves
$\mathcal{S}^{2\widetilde M}_{2M}$; the bounds
below hold for $t<t^\dagger$. On the initial
line, $\widehat{\Nc}(0)\le M$ bounds $\partial_x\hat y$ and
$\partial_x^2v_s^{(1)}$ by $M$, so, by \eqref{eq:app_conv},
$|p|,|q|,|r|,|P_1|,|P_2|\le M_0$. At $x=0$, by
\eqref{eq:app_in0}--\eqref{eq:app_r0}, with
$\hat\lambda^{(1)}(t,0)\ge\hat m_{\rm kin}\ge\frac1{2M}$, $\theta_\zeta\le2$
and $\epsin\le1$,
\begin{equation*}
  |p|\le2M\|v_m\|_{C^1},\qquad
  |q|\le2M\bigl(\|\epsin\|_{C^1}+4M\|v_m\|_{C^1}\bigr),\qquad
  |r|\le C_r .
\end{equation*}
By \eqref{eq:app_tr}, the outgoing traces at $x=z_s(0)^\pm$ and $x=z_b$
come along characteristics from the initial line: both components of
$y^{(1)}$ at $x=z_s(0)^-$, $R_1$ at $x=z_s(0)^+$ and $R_2$ at $x=z_b$. As
long as the values of $p,q,P$ on these characteristics stay below $M_0+1$,
\eqref{eq:app_conv} and $\partial_t\hat y^{(1)}=-\hat A^{(1)}\partial_x\hat y^{(1)}+\hat S^{(1)}$
give $|\partial_ty^-|\le a_1(1+2\mathsf b(M_0+1))\le a_2$, and the transport
form $\partial_tR_i=S_{\rm R}-\hat\lambda^{(2)}_iP_i$ of \eqref{eq:app_dxR},
which involves only the outgoing $P_i$, gives
$|\partial_tR_1(t,z_s(0)^+)|,|\partial_tR_2(t,z_b)|\le a_1(2+M_0)\le a_2$.
For the incoming components, \eqref{eq:app_dxR} gives
$P_i=(S_{\rm R}-\partial_tR_i)/\hat\lambda^{(2)}_i$ with
$|\hat\lambda^{(2)}_i|\ge\frac1{2M}$, and \eqref{eq:app_bdata} gives
$\partial_tR_1(t,z_b)=2\dot v_m-\partial_tR_2(t,z_b)$ and
$\partial_tR_2^+=\nabla h\cdot\frac{\dd}{\dd t}(R_1^+,y^-)$, so
\begin{equation*}
  |P_1(t,z_b)|\le2M\bigl(a_1+2\|v_m\|_{C^1}+a_2\bigr),\qquad
  |P_2(t,z_s(0)^+)|\le2M\bigl(a_1+3\Theta_h\,a_2\bigr).
\end{equation*}
All these values at the starting points are bounded by $G_0$.

We now show by a continuity argument that the solution exists on
$[0,\tau]$. Let $T_e\le\tau$ be the existence time of the frozen solution. Before
$t^\dagger$, the bounds used to define the constants above hold; we show that the solution then
satisfies the stronger bounds \eqref{eq:app_boot}, which lie strictly inside
$\mathcal{S}^{2\widetilde M}_{2M}$, so that $t^\dagger\ge T_e$.

The speeds contain a term in $\zeta-\dot z_{s,0}$, whose time derivative is not bounded independently of $M'$, but which is small: since $|\zeta-\dot z_{s,0}|\le\delta$,
$x\le z_s(0)\le2z_\zeta$ on $\Om_{1,0}$ and
$z_b-x\le z_b-z_s(0)\le2(z_b-z_\zeta)$ on $\Om_{2,0}$,
\begin{equation*}
  \Bigl|\frac{x(\zeta-\dot z_{s,0})}{z_\zeta}\Bigr|,\
  \Bigl|\frac{(z_b-x)(\zeta-\dot z_{s,0})}{z_b-z_\zeta}\Bigr|
  \le2\delta=\frac1{4M}.
\end{equation*}
All the other quantities change at a rate at most $\mathsf c$, so, for
$t<\min\{t^\dagger,T_e\}$, at fixed $x$ and along characteristics,
\begin{gather*}
  |\hat y(t,x)-\hat y(0,x)|,\ |\rho_i(t)-1|,\ |z_\zeta(t)-z_s(0)|\le\mathsf c\,t\le\tfrac1{4\widetilde M},\\
  |\hat\lambda(t,x)-\hat\lambda(0,x)|\le\tfrac1{4\widetilde M}+\tfrac1{4M},\qquad
  |\bar g(t)-\bar g(t_\alpha)|\le\mathsf c\,t\le\tfrac14\quad\text{for }g\in\{p,q,r,P_1,P_2\},
\end{gather*}
We compare with $t=0$, where the state is in $\mathcal{S}^{M}_{M}$, and with
the starting points, where the gradients are bounded by $G_0$. Since
$\frac1{4\widetilde M}\le\frac1{48M}$, the state moves by at most $\frac12$,
each margin loses at most $\frac1{4M}+\frac1{24M}$ (the margin $\jump{\epss}$
loses $\frac2{4\widetilde M}$ because it involves two traces), and, by
\eqref{eq:app_conv}, the gradients are bounded by $\mathsf b(|p|+|q|)$,
$\mathsf b(|P_1|+|P_2|)$ and $|r|$, with $|p|,|q|,|r|,|P_i|\le G_0+\frac14$.
This gives \eqref{eq:app_boot} for $t<\min\{t^\dagger,T_e\}$, with
$\|\partial_x\hat y^{(1)}\|+\|\partial_x\hat y^{(2)}\|+\|\partial_x^2v_s^{(1)}\|
\le5\mathsf b(G_0+\frac14)<\widetilde M$. Since $Mz_b\ge2$
($z_b\ge2\,\mathrm{dist}\ge\frac2M$), $[\frac34,\frac54]\subset(\frac1{2Mz_b},2Mz_b)$,
so \eqref{eq:app_boot} is strictly inside $\mathcal{S}^{2\widetilde M}_{2M}$;
hence $t^\dagger\ge T_e$. Moreover, $(R_1^+,y^-)\in A_M$, so, by \eqref{eq:app_AM},
$|f_s^-|\le2M$, $\epss^-\in[\frac1{2M},\eps_c-\frac1{2M}]$ and
$v_s^--R_1^+=\mathcal{G}_{\epss^-}(\epss^+)\in[\mathcal{G}_{\epss^-}(\eps_c+\frac1{2M}),\mathcal{G}_{\epss^-}(1-\frac1{2M})]$;
these conditions define a compact set, contained in $D_h$ since
$\mathcal{G}_{\epss^-}(1-\frac1{2M})<\mathcal{G}_{\epss^-}(1^-)$, so
$(R_1^+,y^-)$ stays in a compact subset of $D_h$.
Then \eqref{eq:app_boot} bounds $\bar p$ and the $C^1$ norm on $[0,T_e)$, so,
by the construction above, the solution continues beyond $T_e$ unless
$T_e=\tau$: the frozen solution exists on $[0,\tau]$ and satisfies
\eqref{eq:app_boot}. Uniqueness follows from the uniqueness of the
characteristics on $\Om_{1,0}$ and from \cite[Ch.~2]{li1985boundary} on
$\Om_{2,0}$. The right-hand sides of \eqref{eq:app_riccati}--\eqref{eq:app_epsq}
do not involve $X$, so $|\bar g(t)-\bar g(t_\alpha)|\le\mathsf c\,t$, hence
$|\bar p|,|\bar q|,|\bar r|\le G_0+\frac14$ and $\frac1{2M}\le\bar\epss\le1$,
also hold on the extensions of the characteristics of $\Om_{1,0}$ beyond
$x=z_s(0)$, for $\alpha\in[-t,z_s(0)]$.

Finally, (ii) follows from \eqref{eq:app_boot}: $|\partial_t\hat y|\le\|\hat A\|\,|\partial_x\hat y|+\|\hat S\|\le a_1(1+2\widetilde M)$.
\end{proof}

\subsubsection{Lipschitz dependence on the interface velocity}\label{sssec:app_lipsec}

The contraction argument of \Cref{ssec:app_local} requires that the traces
of the frozen solution at the interface depend on $\zeta$ in a Lipschitz way,
with a constant that tends to $0$ with $t$. This is the content of the next
lemma.

\begin{lemma}[Lipschitz dependence on $\zeta$]\label{lem:app_lip}
Under the assumptions of \Cref{lem:app_frozen}, let
$\zeta_1,\zeta_2\in\mathcal{Z}_{\tau,\delta,M'}$,
$\tilde\zeta\triangleq\zeta_1-\zeta_2$, and
$\Delta y^{(i)}(t)\triangleq\|y^{(i)}_{\zeta_1}-y^{(i)}_{\zeta_2}\|_{C^0([0,t]\times\overline{\Om_{i,0}})}$.
For $t\le\tau$,
\begin{equation}\label{eq:app_lip}
  \Delta y^{(1)}(t)\le L_1\int_0^t\|\tilde\zeta\|_s\dd s,\qquad
  \Delta y^{(2)}(t)\le L_2\Bigl(\int_0^t\|\tilde\zeta\|_s\dd s+\Delta y^{(1)}(t)\Bigr),
\end{equation}
where the constants $L_1(M)$ and $L_2(M)$ depend only on $M$ and on the
data; their explicit values, \eqref{eq:app_lip1} and \eqref{eq:app_lip2}, are
obtained in the proof.
\end{lemma}

\begin{proof}
We compare the two frozen solutions at the same point $(t,x)$ by following
their characteristics back to their starting points. The difference has three
sources: the coefficients $\theta_\zeta$ and $\omega_\zeta$ differ; the
characteristics through $(t,x)$ differ, and so do their starting points; and, on
$\Om_{2,0}$, the data at the interface differ, because they are computed
through $h$ from the traces of $\Om_{1,0}$. We first bound the difference of
the coefficients, then treat $\Om_{1,0}$, which does not depend on
$\Om_{2,0}$, and finally $\Om_{2,0}$.

Index by $k=1,2$ the quantities computed with $\zeta_k$. By
\Cref{lem:app_frozen}, along both frozen solutions,
$|\bar v_s|\le4M^2$, $\epss\le1$, $\hat m_{\rm kin}\ge\frac1{2M}$ and, by
\eqref{eq:app_conv},
\begin{equation*}
  |p|,|q|,|r|,|P_1|,|P_2|\le G_1(M)\triangleq2\mathsf b(M)\widetilde M(M).
\end{equation*}

The coefficients depend on $\zeta$ through $\zeta$ itself and through
$z_\zeta$. By \eqref{eq:app_thetaomega},
$|\omega_\zeta|,|\omega^{(2)}_\zeta|\le\omega_M\triangleq2M(M^2+1)$. For
$t\le\tau\le1$, $|z_{\zeta_1}-z_{\zeta_2}|(t)\le\int_0^t|\tilde\zeta|\le t\|\tilde\zeta\|_t$
and $z_{\zeta_k}\ge\frac{z_s(0)}2$, so
\begin{gather*}
  |\theta_1-\theta_2|=\frac{z_s(0)\,|z_{\zeta_1}-z_{\zeta_2}|}{z_{\zeta_1}z_{\zeta_2}}
  \le\frac{4t\|\tilde\zeta\|_t}{z_s(0)},\\
  |\omega_1-\omega_2|\le\frac{|\tilde\zeta|}{z_{\zeta_1}}
  +\frac{|\zeta_2|\,|z_{\zeta_1}-z_{\zeta_2}|}{z_{\zeta_1}z_{\zeta_2}}
  \le\frac{2+2\omega_Mt}{z_s(0)}\|\tilde\zeta\|_t,
\end{gather*}
and the same on $\Om_{2,0}$ with $z_b-z_s(0)$ in place of $z_s(0)$. Since
$d_0\ge\frac1M$, on both zones,
\begin{equation}\label{eq:app_coeff}
  |\theta_1-\theta_2|(t)+|\omega_1-\omega_2|(t)\le C_\theta\|\tilde\zeta\|_t,\qquad
  C_\theta\triangleq M(6+2\omega_M) .
\end{equation}

On $\Om_{1,0}$, let $\alpha_k$ be the starting point on $\Gamma_{in}$ of the
characteristic of $\zeta_k$ through $(t,x)$, and
$\bar W\triangleq(\bar v_s,\bar p,\bar\epss)$. Then
\begin{equation}\label{eq:app_split}
  W_{\zeta_1}(t,x)-W_{\zeta_2}(t,x)
  =\underbrace{\bar W_{\zeta_1}(t;\alpha_1)-\bar W_{\zeta_2}(t;\alpha_1)}_{\rm(I)}
  +\underbrace{\bar W_{\zeta_2}(t;\alpha_1)-\bar W_{\zeta_2}(t;\alpha_2)}_{\rm(II)} ,
\end{equation}
where (I) is the effect of $\zeta$ along a characteristic with a fixed
starting point, and (II) the effect of the displacement of the starting point. We
bound the two terms one after the other.

For (I), we note first that the equation for $\bar v_s$ and its initial value do not
involve $\zeta$, so $\bar v_{s,1}=\bar v_{s,2}$ for the same starting point. Let
$\Delta p\triangleq\bar p_1-\bar p_2$ and
$\Delta\epss\triangleq\bar\epss_{,1}-\bar\epss_{,2}$. At the starting point,
$\Delta\epss=0$, $\Delta p=0$ on the initial line and, at $x=0$, by
\eqref{eq:app_in0} with $\theta_\zeta^{-1}=z_\zeta/z_s(0)$ and $v_m\ge0$,
\begin{equation*}
  |\Delta p(t_\alpha)|=\frac{|\dot v_m|\,|z_{\zeta_1}-z_{\zeta_2}|(t_\alpha)}{z_s(0)(v_m+K/\gamma)}
  \le\frac{\gamma M\|v_m\|_{C^1}}{K}\int_0^t\|\tilde\zeta\|_s\dd s .
\end{equation*}
By \eqref{eq:app_riccati}--\eqref{eq:app_epsq},
\begin{equation*}
\begin{aligned}
  \frac{\dd}{\dd t}\Delta p&=\bigl(\omega_1-\gamma-\theta_1(\bar p_1+\bar p_2)\bigr)\Delta p
  -(\theta_1-\theta_2)\bar p_2^{\,2}+(\omega_1-\omega_2)\bar p_2,\\
  \frac{\dd}{\dd t}\Delta\epss&=-\theta_1\bar p_2\,\Delta\epss-\theta_1\bar\epss_{,1}\,\Delta p
  -(\theta_1-\theta_2)\bar\epss_{,2}\bar p_2 .
\end{aligned}
\end{equation*}
The coefficients of $\Delta p$ and $\Delta\epss$ are bounded by
$\kappa(M)\triangleq\omega_M+\gamma+4G_1$, so, by \eqref{eq:app_coeff} and
Gr\"onwall's lemma, for $t\le1$,
\begin{gather*}
  |\Delta p(t)|\le C_p\int_0^t\|\tilde\zeta\|_s\dd s,\qquad
  C_p\triangleq e^{\kappa}\Bigl(\frac{\gamma M\|v_m\|_{C^1}}{K}+C_\theta(G_1^2+G_1)\Bigr),\\
  |\Delta\epss(t)|\le e^{\kappa}\bigl(2C_p+C_\theta G_1\bigr)\int_0^t\|\tilde\zeta\|_s\dd s,
\end{gather*}
and, since $\Delta\bar v_s=0$,
\begin{equation}\label{eq:app_I}
  |{\rm(I)}|\le C_{\rm I}\int_0^t\|\tilde\zeta\|_s\dd s,\qquad
  C_{\rm I}\triangleq C_p+e^{\kappa}\bigl(2C_p+C_\theta G_1\bigr).
\end{equation}

For (II), we bound the distance between the starting points $\alpha_1$ and
$\alpha_2$ by the distance at time $t$ between the two characteristics
starting from $\alpha_1$. The functions $\alpha\mapsto X(t;\alpha)$ and
$\alpha\mapsto\bar W(t;\alpha)$ are continuous (by the compatibility at the
corner, \Cref{def:compat}) and $C^1$ on each side of $\alpha=0$, so the mean
value inequalities below hold with the one-sided bounds. Since $\bar v_s$ is the same for both, these characteristics
satisfy
$\frac{\dd}{\dd t}(X_1-X_2)=-\omega_2(X_1-X_2)+(\theta_1-\theta_2)\bar v_s-(\omega_1-\omega_2)X_1$,
hence, since $0\le X_1\le z_s(0)<z_b$,
$|X_1-X_2|(t;\alpha_1)\le e^{\omega_M}(4M^2+z_b)C_\theta\int_0^t\|\tilde\zeta\|_s\dd s$.
By \eqref{eq:app_dalphaX}, with $\hat\lambda^{(1)}(t_\alpha,0)\in[\frac1{2M},a_1]$
and $|\theta_\zeta\bar p-\omega_\zeta|\le\Lambda_1\triangleq2G_1+\omega_M$,
\begin{equation}\label{eq:app_jacflow}
  \frac1{2M}\,e^{-\Lambda_1t}\le\partial_\alpha X(t;\alpha)\le\mu_X\,e^{\Lambda_1t},\qquad
  \mu_X\triangleq\max\{1,a_1\}.
\end{equation}
Since $X_2(t;\alpha_2)=x=X_1(t;\alpha_1)$,
\begin{equation}\label{eq:app_footprint}
\begin{gathered}
  |\alpha_1-\alpha_2|\le2Me^{\Lambda_1}\,|X_2(t;\alpha_1)-X_2(t;\alpha_2)|
  \le C_\alpha\int_0^t\|\tilde\zeta\|_s\dd s,\\
  C_\alpha\triangleq2Me^{\Lambda_1+\omega_M}(4M^2+z_b)C_\theta,
\end{gathered}
\end{equation}
and, since $\partial_\alpha\bar W=(\bar p,\bar r,\bar q)\,\partial_\alpha X$,
\begin{equation}\label{eq:app_II}
  |{\rm(II)}|\le\sup|\partial_\alpha\bar W_{\zeta_2}|\,|\alpha_1-\alpha_2|
  \le3G_1\mu_Xe^{\Lambda_1}C_\alpha\int_0^t\|\tilde\zeta\|_s\dd s .
\end{equation}
Since $f_s=\epss v_s$, $\epss\le1$ and $|v_s|\le4M^2$, \eqref{eq:app_I} and
\eqref{eq:app_II} give the first inequality of \eqref{eq:app_lip}, with
\begin{equation}\label{eq:app_lip1}
  L_1\triangleq(1+4M^2)\bigl(C_{\rm I}+3G_1\mu_Xe^{\Lambda_1}C_\alpha\bigr).
\end{equation}

On $\Om_{2,0}$, the speeds depend on $R$, hence on $\zeta$, so the
displacement of the characteristics and the change of the values along them
are coupled and cannot be bounded separately as (I) and (II) were; moreover,
the interface data depend on $\Om_{1,0}$ through $h$. We derive instead an integral inequality for the distance
between the two solutions and conclude with Gr\"onwall's lemma. Let
$D(t)\triangleq\|R^{\zeta_1}-R^{\zeta_2}\|_{C^0([0,t]\times\overline{\Om_{2,0}})}$.
By \eqref{eq:app_a1}, $a_1$ is a Lipschitz constant of
$(x,R,\theta,\omega)\mapsto\hat\lambda^{(2)}_i$, $R\mapsto S_{\rm R}$ and
$R\mapsto\mathfrak{r}^{-1}(R)$, and $|S_{\rm R}|\le a_1$; moreover
$|P^{\zeta_2}|\le G_1$, $|\partial_tR^{\zeta_2}|\le\mathsf c$ at
$x\in\{z_s(0)^+,z_b\}$, $|\hat\lambda^{(2)}_i|\ge\frac1{2M}$, and
$|h(a)-h(a')|\le\Theta_h|a-a'|_1$ for $a,a'\in A_M$ (\Cref{lem:app_h}(ii)). Fix $(t,x)$ and $i$, and let $X_k$ be
the $i$-characteristic of $\zeta_k$ through $(t,x)$, with starting point
$(t_k,b_k)$. Since
$|R^{\zeta_1}(s,X_1)-R^{\zeta_2}(s,X_2)|\le D(s)+G_1|X_1-X_2|(s)$,
\begin{equation*}
  \Bigl|\frac{\dd}{\dd s}(X_1-X_2)\Bigr|
  \le a_1\bigl((1+G_1)|X_1-X_2|+D(s)+C_\theta\|\tilde\zeta\|_s\bigr),
\end{equation*}
and Gr\"onwall's lemma, applied backward in time from $X_1(t)=X_2(t)$, gives for $t\le1$
\begin{equation}\label{eq:app_Xdiff}
  |X_1-X_2|(s)\le C_X\int_0^t\bigl(D(\sigma)+\|\tilde\zeta\|_\sigma\bigr)\dd\sigma,\qquad
  C_X\triangleq a_1(1+C_\theta)e^{a_1(1+G_1)} .
\end{equation}
We now compare the data at the two starting points. Each characteristic moves
monotonically with speed at least $\frac1{2M}$, so the starting points are close
in time and space. With $\delta_{\rm in}\triangleq|R^{\zeta_1}-R^{\zeta_2}|(t_1,b_1)$,
the difference of the two boundary data at the same starting point, the data at
the two starting points differ by at most
\begin{equation*}
\begin{cases}
  G_1|X_1-X_2|(0), & t_1=t_2=0,\\
  \delta_{\rm in}+\mathsf c|t_1-t_2|,\quad |t_1-t_2|\le2M|X_1-X_2|, & b_1=b_2\ \text{lateral},\\
  \delta_{\rm in}+\mathsf ct_1+G_1|X_2(0)-b_1|,\quad
  \tfrac{t_1}{2M},\,|X_2(0)-b_1|\le|X_1-X_2|(t_1), & b_1\ \text{lateral},\ t_2=0,
\end{cases}
\end{equation*}
the last case using the continuity of the data at the corner
(\Cref{def:compat}); the case $t_1=0$, $b_2$ lateral, is symmetric.
Integrating \eqref{eq:app_Rsys} along $X_1$ and $X_2$, with
$|S_{\rm R}(R^{\zeta_1}(s,X_1))-S_{\rm R}(R^{\zeta_2}(s,X_2))|\le a_1(D(s)+G_1|X_1-X_2|(s))$,
\begin{equation}\label{eq:app_z2a}
\begin{gathered}
  \bigl|R_i^{\zeta_1}(t,x)-R_i^{\zeta_2}(t,x)\bigr|
  \le\delta_{\rm in}+C\int_0^t\bigl(D(s)+\|\tilde\zeta\|_s\bigr)\dd s,\\
  C\triangleq a_1+C_X\bigl(G_1+a_1G_1+2M(\mathsf c+a_1)\bigr),
\end{gathered}
\end{equation}
with $\delta_{\rm in}=0$ when both starting points are on the initial line. It
remains to bound $\delta_{\rm in}$ at the two lateral boundaries. By
\eqref{eq:app_bdata},
\begin{equation*}
  \delta_{\rm in}=|R_2^{\zeta_1}-R_2^{\zeta_2}|(t_1,z_b)\ \text{at}\ z_b,\qquad
  \delta_{\rm in}\le\Theta_h\bigl(|R_1^{\zeta_1}-R_1^{\zeta_2}|(t_1,z_s(0)^+)+2\Delta y^{(1)}(t_1)\bigr)\
  \text{at}\ z_s(0)^+,
\end{equation*}
and, by \eqref{eq:app_tr}, the outgoing traces on the right-hand sides
satisfy \eqref{eq:app_z2a} with $\delta_{\rm in}=0$. Hence
\begin{equation*}
  D(t)\le(2+\Theta_h)\,C\int_0^t\bigl(D(s)+\|\tilde\zeta\|_s\bigr)\dd s
  +2\Theta_h\,\Delta y^{(1)}(t),
\end{equation*}
and Gr\"onwall's lemma (with $\Delta y^{(1)}$ non-decreasing and $t\le1$), with
$\Delta y^{(2)}\le a_1D$, gives the second inequality of \eqref{eq:app_lip},
with
\begin{equation}\label{eq:app_lip2}
  L_2\triangleq a_1\max\{(2+\Theta_h)C,\,2\Theta_h\}e^{(2+\Theta_h)C}.
\end{equation}
\end{proof}

\subsection{Local existence and uniqueness}\label{ssec:app_local}

The interface velocity is recovered as a fixed point: the Rankine--Hugoniot
quotient of the traces of the frozen solution defines a map
$\zeta\mapsto\mathcal{T}(\zeta)$, which is a contraction on a short time
interval. The constants are the following: $\Theta_y\,t$ is the Lipschitz constant of the traces as functions of $\zeta$ (\Cref{lem:app_lip}), $\Theta_{RH}$ that of the
Rankine--Hugoniot quotient, $M'$ bounds the Lipschitz constant in time of
$\mathcal{T}(\zeta)$, and $\tau_*$ is small enough for $\mathcal{T}$ to map
$\mathcal{Z}_{\tau,\delta,M'}$ into itself and to be a contraction. Let
\begin{equation}\label{eq:app_taustar}
\begin{gathered}
  \Theta_y(M)\triangleq\max\bigl\{L_1,\,L_2(1+L_1)\bigr\},\qquad
  \Theta_{RH}(M)\triangleq4M(1+4M^2),\\
  M'(M)\triangleq\Theta_{RH}\,a_1(1+2\widetilde M),\qquad
  \tau_*(M)\triangleq\min\Bigl\{\tau_0,\ \frac{1}{8MM'},\ \frac{1}{2\Theta_{RH}\Theta_y}\Bigr\}.
\end{gathered}
\end{equation}

\begin{proposition}[Local existence and uniqueness]\label{prop:app_local}
Let $y_0$, $v_m$ and $\epsin$ satisfy the assumptions of \Cref{thm:wp}, and
$M\ge\widehat{\Nc}(0)$. There is an admissible interface solution on
$[0,\tau_*(M)]$, with $z_s\in C^2$, which satisfies \eqref{eq:app_boot}. Two
admissible solutions with the same data coincide on a neighbourhood of
$t=0$.
\end{proposition}

\begin{proof}
We write $\mathcal{T}$ as the composition of the map $\zeta\mapsto$ traces,
which is Lipschitz by \Cref{lem:app_lip}, and of the quotient $\mathfrak{s}$,
which is Lipschitz while $\jump{\epss}$ stays away from $0$. Let
$\tau\le\tau_*(M)$, $\delta=\frac1{8M}$ and $M'=M'(M)$, and
\begin{equation*}
  \mathcal{T}_1(\zeta)\triangleq\bigl(y^{(1)}_\zeta(\cdot,z_s(0)^{-}),\,y^{(2)}_\zeta(\cdot,z_s(0)^{+})\bigr),
\end{equation*}
with the max norm $\|Y\|$ on $Y=(y^-,y^+)\in\mathbb{R}^4$, so that
$\mathcal{T}=\mathfrak{s}\circ\mathcal{T}_1$. By \Cref{lem:app_lip} and
$\int_0^t\|\tilde\zeta\|_s\dd s\le t\|\tilde\zeta\|_t$,
\begin{equation}\label{eq:app_T1}
  \|\mathcal{T}_1(\zeta_1)-\mathcal{T}_1(\zeta_2)\|_t
  \le\Theta_y\,t\,\|\tilde\zeta\|_t .
\end{equation}
For $Y,Y'$ with $\jump{\epss},\jump{\epss}'\ge j_0>0$,
\begin{equation}\label{eq:app_LRH}
  \mathfrak{s}(Y)-\mathfrak{s}(Y')
  =\frac{\jump{f_s}-\jump{f_s}'}{\jump{\epss}'}
  +\mathfrak{s}(Y)\,\frac{\jump{\epss}'-\jump{\epss}}{\jump{\epss}'},\qquad
  |\mathfrak{s}(Y)-\mathfrak{s}(Y')|\le\frac{2(1+|\mathfrak{s}(Y)|)}{j_0}\|Y-Y'\|.
\end{equation}
By \eqref{eq:app_boot}, the traces of the frozen solutions satisfy
$\jump{\epss}\ge\frac1{2M}$ and $|\mathfrak{s}|\le\frac{2M}{1/(2M)}=4M^2$, so
\eqref{eq:app_LRH} holds with $j_0=\frac1{2M}$ and Lipschitz constant
$\Theta_{RH}$. With \Cref{lem:app_frozen}(ii),
\begin{equation}\label{eq:app_contr}
\begin{gathered}
  \|\mathcal{T}(\zeta_1)-\mathcal{T}(\zeta_2)\|_t\le
  \Theta_{RH}\Theta_y\,t\,\|\tilde\zeta\|_t,\qquad
  \mathcal{T}(\zeta)(0)=\dot z_{s,0},\\
  \mathrm{Lip}\bigl(\mathcal{T}(\zeta)\bigr)\le\Theta_{RH}
  \bigl\|\partial_ty_\zeta(\cdot,z_s(0)^{\pm})\bigr\|\le M',
\end{gathered}
\end{equation}
where the identity follows from \Cref{ass:wp}. Since
$\|\mathcal{T}(\zeta)-\dot z_{s,0}\|_\tau\le M'\tau\le\delta$, $\mathcal{T}$
maps $\mathcal{Z}_{\tau,\delta,M'}$ into itself, and, since
$\Theta_{RH}\Theta_y\tau\le\frac12$, it is a contraction. Its fixed point
$\zeta^*$ is a quotient of traces that are $C^1$ in time, so $\zeta^*\in C^1$
and $z_s\in C^2$; by \Cref{lem:app_fixed,lem:app_bc,lem:app_h} and
\eqref{eq:app_boot}, $(y^{(1)},y^{(2)},z_s)$ is an admissible solution.

It remains to prove uniqueness. By \Cref{lem:RH}, the interface velocity of an
admissible solution is the quotient $\mathfrak{s}$ of $C^1$ traces, and, by
\Cref{lem:app_bc}, the solution satisfies \eqref{eq:app_bdata}.
For $\tau$ small, its velocity belongs to $\mathcal{Z}_{\tau,\delta,M''}$
for some $M''\ge M'$ and is a fixed point of $\mathcal{T}$. Since $M'$ and
$\Theta_y$ do not depend on $M''$,
$\mathcal{T}(\mathcal{Z}_{\tau,\delta,M''})\subseteq\mathcal{Z}_{\tau,\delta,M'}
\subseteq\mathcal{Z}_{\tau,\delta,M''}$, with the same contraction factor,
so this fixed point is $\zeta^*$.
\end{proof}

\subsection{Maximal solution and breakdowns}\label{ssec:app_cont}

Since the length of the existence interval depends only on a bound of the
margin function at its initial time, the local solution can be continued as
long as this function stays bounded; this is shown in
\Cref{sssec:app_crit}. In \Cref{sssec:app_break}, we identify which of its
terms can become unbounded, which completes the proof of \Cref{thm:wp}.

\subsubsection{Continuation criterion}\label{sssec:app_crit}

\begin{proposition}[Continuation criterion]\label{prop:app_cont}
Under the assumptions of \Cref{thm:wp}, there is a unique maximal admissible
interface solution on $[0,T_{\max})$, and
\begin{equation}\label{eq:app_crit}
  T_{\max}<\infty\quad\Longrightarrow\quad
  \limsup_{t\to T_{\max}^{-}}\widehat{\Nc}(t)=\limsup_{t\to T_{\max}^{-}}\Nc(t)=+\infty .
\end{equation}
\end{proposition}

\begin{proof}
We first show that the solution can be restarted from any time; with the
local uniqueness of \Cref{prop:app_local}, this gives a unique maximal
solution. Since the length $\tau_*$ of the restarted interval depends only on
a bound of the margin function at the restart time, the solution can then be
extended beyond $T_{\max}$ if this function stays bounded.

At any time $t_0$ of an admissible solution, the state $y(t_0,\cdot)$
satisfies \Cref{ass:wp}, and the inputs are $C^2$-compatible with it
(\Cref{def:compat}): the boundary conditions and $H=0$ hold identically on
$[0,t_0]$, and the solution is $C^1$ with $\partial_xv_s^{(1)}\in C^1$, so their time
derivatives hold at $t_0$. \Cref{prop:app_local} therefore applies from
$t_0$, with the map \eqref{eq:app_rho} based at $t_0$ and any
$M\ge\widehat{\Nc}^{[t_0]}(t_0)$.

Two admissible solutions coincide on their common interval: otherwise, let
$t_1$ be the infimum of the times where they differ; they coincide at $t_1$
by continuity, and the uniqueness part of \Cref{prop:app_local}, started at
$t_1$, gives a contradiction. Hence there is a unique maximal solution, on
$[0,T_{\max})$ with $T_{\max}\ge\tau_*(\widehat{\Nc}(0))>0$.

Assume $T_{\max}<\infty$ and $\widehat{\Nc}\le M$ on $[0,T_{\max})$. For
$t_0\in[0,T_{\max})$, the map based at $t_0$ has $\rho_i(t_0)=1$, so the
mapped and physical margin functions coincide at $t_0$, and
\Cref{lem:Nequiv} (base time $0$) gives
\begin{equation}\label{eq:app_restart}
  \widehat{\Nc}^{[t_0]}(t_0)=\Nc(t_0)
  \le(1+z_b)^2\bigl(1+\widehat{\Nc}(t_0)\bigr)^3\le(1+z_b)^2(1+M)^3\triangleq M_1 .
\end{equation}
With $t_0\triangleq\max\{0,\,T_{\max}-\frac12\tau_*(M_1)\}$,
\Cref{prop:app_local} with $M_1$ extends the solution to
$t_0+\tau_*(M_1)>T_{\max}$, a contradiction. Hence $\widehat{\Nc}$ is
unbounded on $[0,T_{\max})$; since it is finite and continuous there,
$\limsup_{t\to T_{\max}^-}\widehat{\Nc}(t)=+\infty$, and \Cref{lem:Nequiv}
gives the same result for $\Nc$.
\end{proof}

\subsubsection{The four breakdowns}\label{sssec:app_break}

It remains to show that the blow-up of $\widehat{\Nc}$ is one of the
breakdowns (i)--(iv) of \Cref{thm:wp}. We split $\widehat{\Nc}$ into the
six structural terms of $\widehat{\Nc}_0$,
\begin{equation}\label{eq:app_Nstr}
  \widehat{\Nc}_{\rm str}\triangleq\frac{1}{\jump{\epss}}+\frac{1}{\hat m_{\rm kin}}
  +\frac{1}{m_{\rm phase}}+\frac{1}{\mathrm{dist}(z_s,\partial\Om)}
  +\frac{1}{\inf_\Om\epss}+\frac{1}{1-\sup_\Om\epss},
\end{equation}
the $C^0$ norm of the state, and the derivatives. Let
$M_s\triangleq\sup_{[0,T_{\max})}\widehat{\Nc}_{\rm str}\in(0,\infty]$. The
next lemma shows that, while the structural terms stay bounded, the state
stays bounded, and that $\partial_x^2v_s$ and $\partial_x\epss$ on
$\Om_{1,0}$ stay bounded as long as $p=\partial_xv_s$ does.

\begin{lemma}[State and higher derivatives]\label{lem:app_ctrl}
Let $(y,z_s)$ be the maximal solution of \Cref{prop:app_cont}, with $T_{\max}<\infty$ and $M_s<\infty$. Then:
\begin{enumerate}[label=(\alph*),leftmargin=2em,itemsep=1pt]
\item $\sup_{[0,T_{\max})}\|\hat y\|<\infty$;
\item if $\mathsf\Pi_p\triangleq\sup_{[0,T_{\max})}\|p\|<\infty$, then
$\|\partial_x^2v_s^{(1)}\|$, $\|\partial_x\epss^{(1)}\|$ and
$\|\partial_xf_s^{(1)}\|$ are bounded on $[0,T_{\max})$.
\end{enumerate}
\end{lemma}

\begin{proof}
For (a), we bound $v_s$ on each zone; this bounds $\hat y=(\epss,\epss v_s)$
since $\epss\le1$. On $\Om_{1,0}$, the first equation of \eqref{eq:app_riccati} is
linear, so along any characteristic
\begin{gather*}
  \bar v_s(t)=e^{-\gamma(t-t_\alpha)}\bar v_s(t_\alpha)
  +\int_{t_\alpha}^t\gamma e^{-\gamma(t-s)}\Bigl(v_m(s)+\frac{K}{\gamma}\Bigr)\dd s,\\
  |\bar v_s(t)|\le\max\Bigl\{\|v_{s,0}\|,\ \|v_m\|+\frac{K}{\gamma}\Bigr\}.
\end{gather*}
On $\Om_{2,0}$, let $\Psi_s\triangleq\Psi(1-1/M_s)$, $\mathcal{R}_s\triangleq\sup_{[0,1-1/M_s]}\mathcal{R}$
and $U_R(t)\triangleq\sup_{[0,t]}\max_i\|R_i\|$. Since
$0\le\Psi(\epss)\le\Psi_s$ and $|v_s|\le\max_i|R_i|$, \eqref{eq:app_Rsys}
gives $|\frac{\dd}{\dd t}\bar R_i|\le\Upsilon(t)\triangleq K+\mathcal{R}_s(U_R(t)+\|v_m\|)$,
and \eqref{eq:app_bdata} gives
\begin{gather*}
  |R_1(t,z_b)|=|v_m-\Psi(\epss)|\le\|v_m\|+\Psi_s,\\
  |R_2(t,z_s(0)^+)|=|R_1(t,z_s(0)^+)+2\Psi(\epss^+)|\le|R_1(t,z_s(0)^+)|+2\Psi_s .
\end{gather*}
Following an $R_2$-characteristic back to the interface and then, from that
point, the $R_1$-characteristic back to $t=0$ or to $z_b$,
\begin{equation}\label{eq:app_C0lin}
\begin{gathered}
  U_R(t)\le U_R(0)+\|v_m\|+3\Psi_s+\int_0^t\Upsilon(s)\dd s,\\
  U_R(t)\le\Bigl(U_R(0)+\|v_m\|+3\Psi_s+(K+\mathcal{R}_s\|v_m\|)t\Bigr)e^{\mathcal{R}_st},
\end{gathered}
\end{equation}
the second by Gr\"onwall's lemma. Since $|v_s|\le U_R$ on $\Om_{2,0}$ and
$T_{\max}<\infty$, this proves (a).

For (b), we use that the equations for $r$ and $q$ along the characteristics
are linear. On $\Om_{1,0}$, \eqref{eq:app_riccati}--\eqref{eq:app_epsq} are linear in
$(\bar r,\bar q)$ with coefficients bounded by
$\kappa_T\triangleq3\theta_T\mathsf\Pi_p+2\omega_T+\gamma$, where
$\theta_T\triangleq z_bM_s\ge\rho_1^{-1}$ and
$\omega_T\triangleq\sup|\mathfrak{s}|M_s\ge|\dot z_s/z_s|$ (finite by (a),
since $|\mathfrak{s}|\le2\|f_s\|M_s$), so
\begin{equation}\label{eq:app_rq}
  |\bar r(t)|\le|\bar r(t_\alpha)|e^{\kappa_Tt},\qquad
  |\bar q(t)|\le\bigl(|\bar q(t_\alpha)|+\theta_T\|\epss\|\,\|r\|\,t\bigr)e^{\kappa_Tt},
\end{equation}
with values at the starting points bounded in terms of $M_s$, $y_0$ and the inputs by
\eqref{eq:app_in0}--\eqref{eq:app_r0}.
Hence $\|\partial_x^2v_s^{(1)}\|$ and $\|\partial_x\epss^{(1)}\|$ are bounded on
$[0,T_{\max})$, and so is
$\|\partial_xf_s^{(1)}\|\le\|\partial_x\epss^{(1)}\|\,\|v_s\|+\mathsf\Pi_p$.
\end{proof}

By contrast, $p$ and $P$ satisfy the quadratic equations
\eqref{eq:app_riccati} and \eqref{eq:app_dxR}. For instance, with the bounds $\rho_1^{-1}\le\theta_T$ and $|\dot z_s/z_s|\le\omega_T$ from the proof of \Cref{lem:app_ctrl}, $|\bar p|$ satisfies
\begin{equation}\label{eq:app_riccati2}
  \frac{\dd}{\dd t}|\bar p|\le\theta_T|\bar p|^2+(\omega_T+\gamma)|\bar p|,
\end{equation}
whose solutions may blow up in finite time while the state stays bounded.

\begin{proof}[Proof of \Cref{thm:wp}]
Existence, uniqueness, $z_s\in C^2$ and the blow-up of $\Nc$ are
\Cref{prop:app_local,prop:app_cont}. Let $T_{\max}<\infty$. We distinguish
two cases, according to whether the structural terms stay bounded.
\begin{itemize}
\item If $M_s=\infty$, one of the six terms of \eqref{eq:app_Nstr} tends to
$\infty$ along a sequence $t_n\to T_{\max}$, since each is finite and
continuous on $[0,T_{\max})$:
\begin{equation}\label{eq:app_sort}
\begin{aligned}
  \mathrm{dist}\bigl(z_s,\partial\Om\bigr)\to0
  &\ \Longrightarrow\ \text{(i)},\\
  \jump{\epss}\to0,\ m_{\rm phase}\to0 \text{ or } \hat m_{\rm kin}\to0
  &\ \Longrightarrow\ \text{(iii)},\\
  \inf_\Om\epss\to0 \text{ or } \sup_\Om\epss\to1
  &\ \Longrightarrow\ \text{(iv)} ,
\end{aligned}
\end{equation}
with $\jump{\epss}\ge2m_{\rm phase}$ in the second line and
\eqref{eq:app_reduce} to pass from $\hat m_{\rm kin}$ to $m_{\rm kin}$.
\item If $M_s<\infty$, then $\sup\|\hat y\|<\infty$ by \Cref{lem:app_ctrl}(a). If moreover
$\sup_{[0,T_{\max})}(\|p\|+\|P\|)<\infty$, then \Cref{lem:app_ctrl}(b) and
$\|\partial_x\hat y^{(2)}\|\le\|\partial\mathfrak{r}^{-1}\|\,\|P\|$ bound
every term of $\widehat{\Nc}$ on $[0,T_{\max})$, which contradicts
\eqref{eq:app_crit}. So $\|p\|+\|P\|\to\infty$ along a sequence $t_n\to
T_{\max}$, and, since $|p|\le M_s(|\partial_xf_s|+|v_s||\partial_x\epss|)$ and
$|P|\le\|\partial\mathfrak{r}\|\,|\partial_x\hat y^{(2)}|$,
$\|\partial_x\hat y(t_n)\|\to\infty$, which is (ii) by \eqref{eq:app_reduce}.
\end{itemize}
In the second case ($M_s<\infty$), the state stays bounded and only the
gradients blow up; this is possible because of the quadratic terms in
\eqref{eq:app_riccati2} and \eqref{eq:app_dxR}.
\end{proof}

\bibliographystyle{abbrv}
\bibliography{paper}

\end{document}